\documentclass[10pt]{article}

\title{Global well-posedness and stability of the 2D Boussinesq equations with partial dissipation near a hydrostatic equilibrium}
\author{Kyungkeun Kang\thanks{
		Yonsei University. E-mail address: \url{kkang@yonsei.ac.kr}}
	\and
	Jihoon Lee\thanks{
		Chung-Ang University. E-mail address: \url{jhleepde@cau.ac.kr}}
	\and
	Dinh Duong Nguyen\thanks{
		Yonsei University and Chung-Ang University. E-mail address:  \url{dinhduongnguyen.math@gmail.com}
	}
}

\usepackage{mathtools}
\usepackage{amssymb,bbm,amsmath,amsfonts,amsthm}
\usepackage{vmargin} 
\usepackage[linktocpage,colorlinks=true,raiselinks]{hyperref}
\hypersetup{urlcolor=black, citecolor=red, linkcolor=blue}
\usepackage{cleveref}
\usepackage{tkz-euclide}
\usepackage{multirow}

\usepackage{pgfplots}
\pgfplotsset{compat=newest}
\usepgfplotslibrary{fillbetween}
\usepackage{capt-of}

\numberwithin{equation}{section}

\DeclareMathOperator*{\esssup}{ess\,sup}

\newtheorem{theorem}{Theorem}[section]

\theoremstyle{definition}
\newtheorem{remark}{Remark}[section]
\newenvironment{AMS}{}{}
\newenvironment{keywords}{}{}

\begin{document}
	\maketitle
	
	\begin{abstract}
		The paper is devoted to investigating the well-posedness, stability and large-time behavior near the hydrostatic balance for the 2D Boussinesq equations with partial dissipation. More precisely, the global well-posedness is obtained in the case of partial viscosity and without thermal diffusion for the initial data belonging to  $H^{\delta}(\mathbb{R}^2) \times H^{s}(\mathbb{R}^2)$ for $\delta \in [s-1,s+1]$ if $s \in \mathbb{R}, s > 2$, for $\delta \in (1,s+1]$ if $s \in (0,2]$ and for $\delta \in [0,1]$ if $s = 0$. In addition, if one has either horizontal or vertical thermal diffusion then the stability and large-time behavior are provided in $H^m(\mathbb{R}^2)$, $m \in \mathbb{N}$ and in $\dot{H}^{m-1}(\mathbb{R}^2)$ with $m \in \mathbb{N}$, $m \geq 2$, respectively. 
	\end{abstract}
	
	
	
	\allowdisplaybreaks 
	
	%
	\section{Introduction}
	%
	
	Let us consider the following generalized form for the two-dimensional Boussinesq equations, which describe the motion of an incompressible flow interacting with an active scalar, and can be widely used in geophysics and fluid
	mechanics to model large scale atmospheric and oceanic flows, (see \cite{Majda_2003,Pedlosky_1987})
	\begin{equation} \label{B1} \tag{B1}
		\left\{
		\begin{aligned}
			\partial_t v_1 + v \cdot \nabla v_1 + \partial_1 P &= \nu_1 \partial_{11} v_1 + \nu_2 \partial_{22} v_1,
			\\
			\partial_t v_2 + v \cdot \nabla v_2 + \partial_2 P &= \mu_1\partial_{11} v_2 + \mu_2 \partial_{22} v_2  + g \rho,
			\\
			\partial_t \rho + v \cdot \nabla \rho &= \delta_1 \partial_{11} \rho + \delta_2 \partial_{22} \rho,
			\\
			\textnormal{div}\, v &= 0,
		\end{aligned}
		\right.
		\qquad \text{in} \quad \Omega \times (0,\infty),
	\end{equation}
	where $\Omega = \mathbb{R}^2, \mathbb{T}^2, \mathbb{T} \times \mathbb{R}$ with $\mathbb{T} = [0,1]$ being 1D periodic domain or $\Omega$ be bounded in $\mathbb{R}^2$, $v = (v_1,v_2) : \Omega \times (0,\infty) \rightarrow  \mathbb{R}^2$, $P : \Omega \times (0,\infty) \rightarrow  \mathbb{R}$ and $\rho : \Omega \times (0,\infty) \rightarrow  \mathbb{R}$ are the velocity field, pressure and temperature, respectively. For $i \in \{1,2\}$, the nonnegative viscosity and diffusivity coefficients $\nu_i,\mu_i$ and $\delta_i$ are given. Usually, the constant $g \in \mathbb{R}$ is the acceleration due to gravity and without loss of generality, we can assume that $g = 1$. The given initial data of \eqref{B1} is denoted by $(v,\rho)_{|_{t=0}} = (v_0,\rho_0)$ with $\text{div}\, v_0 = 0$. 
	
	%
	\subsection{Well-posedness}
	%
	
	\textit{Inviscid or full dissipation.} From the mathematical point of view, the 2D inviscid Boussinesq equations (i.e., \eqref{B1} with $\nu_i = \mu_i = \delta_i = 0$, $i = 1,2$) can be identified with the 3D Euler equations for axisymmetric
	swirling flows when the radius $r > 0$, see \cite{Doering_Wu_Zhao_Zheng_2018,Majda_Bertozzi_2002} and the local well-posedness (LWP) in the usual Sobolev spaces $H^m(\mathbb{R}^2) \times H^m(\mathbb{R}^2)$, $m \in \mathbb{N}$, $m \geq 3$ has been given in \cite{Chae_Nam_1997}. However, the blow-up in finite time of this local classical solution or the global well-posedness (GWP) for large data, in general, is still an open question. This is one
	of the fundamental questions in fluid mechanics, as also mentioned in \cite{Yudovich_2003}. Recently, related to this open question, finite time singularity formation for H\"{o}lder continuous or smooth initial data have been given in \cite[corner domains]{Elgindi-Jeong_2020}  and in \cite[half space]{Chen-Hou_2021}. On the other hand, it is well-known that \eqref{B1} is globally well-posed in the case of
	full dissipation, i.e., $\nu_i,\mu_i,\delta_i > 0$, $i = 1,2$, see \cite{Canno-DiBennedetto_1980,Foias-Manley-Temam_1987,Guo_1989}. 
	
	The sub-full dissipation cases, i.e., \eqref{B1} either with full viscosity and without diffusivity or without viscosity and with full diffusivity coefficients, will be summarized as follows.
	
	\textit{Full viscosity: $\nu_i = \mu_i > 0$ and $\delta_i = 0$, $i = 1,2$.} The GWP has been provided in  $H^m(\mathbb{R}^2) \times H^m(\mathbb{R}^2)$ (see \cite{Chae_2006}) and independently in $H^m(\mathbb{R}^2) \times H^{m-1}(\mathbb{R}^2)$ (see \cite{Hou_Li_2005}) for $m \in \mathbb{N}$, $m \geq 3$. For lower regularity data, GWP of weak solutions was obtained in \cite[$\Omega = \mathbb{R}^2$]{Abidi-Hmidi_2007,Boardman-Ji-Qiu-Wu_2019,Danchin-Paicu_2008,Hmidi-Keraani_2007} and in \cite[$\Omega = \mathbb{T}^2$]{Tao_Wu_Zhao_Zheng_2020}. In \cite[$\Omega = \mathbb{R}^2, \mathbb{T}^2$]{Hu-Kakuvica-Ziane_2013} the authors provided a unique global solution in $H^2(\mathbb{R}^2) \times H^1(\mathbb{R}^2)$. The GWP in the fractional Sobolev spaces: a) $H^s(\mathbb{R}^2)$, $s \in \mathbb{R}, s > 1$; b) $H^{s+1}(\mathbb{R}^2) \times H^s(\mathbb{R}^2)$, $s \in (0,1)$; and c) $W^{1,p}(\mathbb{R}^2) \times L^p(\mathbb{R}^2)$, $p \in [2,\infty)$; d) $W^{1+s,p}(\mathbb{R}^2) \times W^{s,p}(\mathbb{R}^2)$, $s \in (0,1)$, $p \in [2,\infty)$ with $sp > 2$; e) $W^{1+s,p}(\mathbb{T}^2) \times W^{s,p}(\mathbb{T}^2)$, $s \in (0,1)$, $p \in [2,\infty)$, have been obtained in \cite{Hu-Kukavica-Ziane_2015} and \cite{Kukavica-Wang-Ziane_2016}, respectively. For bounded domains $\Omega$, the GWP has been studied in \cite{Doering_Wu_Zhao_Zheng_2018,He_2012,Hu-Kakuvica-Ziane_2013, Lai-Pan-Zhao_2011}. The global regularity for temperature patch problem was also investigated in \cite[$\Omega = \mathbb{R}^2$]{Danchin-Zhang_2017,Gancedo-Garcia-Juarez_2017}. Moreover, the upper bound for large time and global attractors are given in \cite[$\Omega = \mathbb{T}^2, \mathbb{R}^2$ or $\Omega$ bounded]{Kukavica-Wang_2020} and in \cite[$\Omega = \mathbb{T}^2$]{Biswas-Foias-Larios_2017}, respectively. Recently, the authors in \cite[$\Omega = \mathbb{T}^2,\mathbb{R}^2$]{Kiselev-Park-Yao_2022} have been contructed global smooth solutions such that the $\dot{H}^s$ norm of $\theta$, $s \geq 1$, grows to infinity algebraically in time. 
	
	\textit{Full diffusivity: $\nu_i = \mu_i = 0$ and $\delta_i > 0$, $i = 1,2$.} The GWP has been given in $H^m(\mathbb{R}^2)$ (see \cite{Chae_2006}) for $m \in \mathbb{N}$, $m \geq 3$. For lower regularity data, for example Yodovich's type data, see \cite[$\Omega= \mathbb{R}^2$]{Danchin-Paicu_2009}. For the case that $\Omega$ is a bounded domain in $\mathbb{R}^2$, see \cite{Huang_2014,Zhao_2010,Zhou_Li_2017}.
	
	\textit{Partial dissipation.} The Boussinesq equations with partial dissipation have been recently getting a lot of attention. The GWP is also extended to the case that \eqref{B1} contains only horizontal viscosity or only horizontal diffusivity, see \cite[$\Omega = \mathbb{R}^2$]{Danchin_Paicu_2011,Paicu-Zhu_2020}, or with a relaxing condition on data in \cite[$\Omega = \mathbb{T}^2$]{Larios_Lunasin_Titi_2013}. If \eqref{B1} has only vertical viscosity and diffusivity 
	then the authors in \cite[$\Omega = \mathbb{R}^2$]{Cao_Wu_2013} provided the GWP in $H^2(\mathbb{R}^2)$, see also \cite{Li-Titi_2016} for lower regularity data. In addition, if \eqref{B1} admits mixed directional viscosity or full viscosity for the vertical velocity then  $H^3$ or $H^s$, $s \in \mathbb{R},s > 2$ solution has been obtained in \cite[$\Omega = \mathbb{R}^2$]{Adhikari-Cao-Shang-Wu-Xu-Ye_2016},
	respectively. The former case is also recently considered in \cite[$\Omega = \mathbb{T}^2$]{He-Ma-Sun_2022} or in \cite[$\Omega$ bounded or $\Omega = \mathbb{R}^2$]{Hu-Wang-Wu-Xiao-Yuan_2018} with a data relaxing and in \cite[$\Omega = \mathbb{R}^2$]{Ye_2020} with a fractional derivative form. Furthermore, the authors in \cite{He-Ma-Sun_2022} are also study the global attractors. 
	
	For the reader's convenience, known GWP and our results to \eqref{B1} in the case of partial viscosity and without diffusivity coefficients are given in the following table:
	
	\begin{table}[htbp]
		\centering
		\small
		\begin{tabular}{|c|c|c|l|}
			\hline
			Case & Positive & Zero & \hspace{2.7cm} Data and GWP results 
			\\\hline
			\multirow{4}{.5em}{1} & \multirow{4}{2em}{$\nu_i,\mu_i$} & \multirow{4}{.5em}{$\delta_i$} & $\bullet$ $H^m$ for $m \in \mathbb{N}, m \geq 3$ in \cite[$\mathbb{R}^2$]{Chae_2006}
			\\
			&&& $\bullet$ $H^m \times H^{m-1}$ for $m \in \mathbb{N}, m \geq 3$ in \cite[$\mathbb{R}^2$]{Hou_Li_2005}
			\\
			&&& $\bullet$ $H^s \times B^0_{2,1} \cap B^0_{p,\infty}$ for $s \in (0,2]$ and $p \in (2,\infty]$ in \cite[$\mathbb{R}^2$]{Hmidi-Keraani_2007}
			\\
			&&& $\bullet$ $L^2 \cap B^{-1}_{\infty,-1} \times B^0_{2,1}$ in \cite[$\mathbb{R}^2$]{Abidi-Hmidi_2007}
			\\
			&&& $\bullet$ $H^s$, $s \in \mathbb{R},s > 1$ or $H^{s+1} \times H^s$, $s \in (0,1)$ in \cite[$ \mathbb{R}^2$]{Hu-Kukavica-Ziane_2015}
			\\
			&&& $\bullet$ $L^2$ in \cite[$\Omega$ bounded]{He_2012}
			\\\hline
			\multirow{4}{.5em}{2} & \multirow{4}{2.1em}{$\nu_1,\mu_1$} & \multirow{4}{3.5em}{$\nu_2,\mu_2,\delta_i$} & $\bullet$ $H^1 \times H^s \cap L^\infty$ with $s \in (\frac{1}{2},1]$, $w_0 \in \sqrt{L}$ in \cite[$\mathbb{R}^2$]{Danchin_Paicu_2011}
			\\
			& & & $\bullet$ $H^s \times
			 H^{s-1}$, $s \in \mathbb{R}, s > 2$ in \cite[$\mathbb{R}^2$]{Danchin_Paicu_2011} 
			\\
			& & & $\bullet$ $H^1 \times L^2 \cap L^\infty$ with $w_0 \in \sqrt{L}$ in \cite[$\mathbb{T}^2$]{Larios_Lunasin_Titi_2013}
			\\
			& & & $\bullet$ $L^2 \times H^{s'}$ with $w_0 \in \sqrt{L} \cap H^s$, $s' \in \mathbb{R}, s \in (\frac{1}{2},s')$ in \cite[$\mathbb{R}^2$]{Paicu-Zhu_2020}
			\\
			\hline
			\multirow{4}{.5em}{3} & \multirow{4}{2.1em}{$\nu_2,\mu_1$} & \multirow{4}{3.5em}{$\nu_1,\mu_2,\delta_i$} & $\bullet$ $H^3$ in \cite[$\mathbb{R}^2$]{Adhikari-Cao-Shang-Wu-Xu-Ye_2016}
			\\
			&&& $\bullet$ $H^1 \times L^2 \cap L^\infty$ in \cite[$\Omega$ bounded or $\mathbb{R}^2$]{Hu-Wang-Wu-Xiao-Yuan_2018}
			\\
			&&& $\bullet$ $L^2 \times L^2$ in \cite[$\Omega = \mathbb{T}^2$]{He-Ma-Sun_2022}
			\\
			&&& $\bullet$ $H^\delta \times H^s$ in Theorem \ref{theo_nu2_mu1_de} for $\Omega = \mathbb{R}^2$, $\delta \in [0,1]$ if $s = 0$,
			\\
			&&& \quad $\delta \in (1,s+1]$ if 
			$s \in (0,2]$ and $\delta \in [s-1,s+1]$ if $s \in \mathbb{R}, s > 2$
			\\
			\hline
			4 & $\mu_1,\mu_2$ & $\nu_1,\nu_2,\delta_i$ & $\bullet$ $H^s$, $s \in \mathbb{R}, s > 2$ in \cite[$\mathbb{R}^2$]{Adhikari-Cao-Shang-Wu-Xu-Ye_2016}
			\\
			\hline
			\multirow{4}{.5em}{5} & $\nu_1,\nu_2 $ & $\mu_1,\mu_2,\delta_i$ &  \multirow{4}{2em}{open}
			\\
			& $\nu_1,\mu_2$ & $\nu_2,\mu_1,\delta_i$ & \multirow{4}{2em}{}
			\\
			& $\nu_2,\mu_2$ & $\nu_1,\mu_1,\delta_i$ & \multirow{4}{2em}{}
			\\
			& & $\nu_i,\mu_i,\delta_i$ & \multirow{4}{2em}{}
			\\\hline
		\end{tabular}
	\end{table}
	
	In the above table, $w_0 := \partial_1v_{02} - \partial_2v_{01}$ denotes the 2D initial vorticity and the space 
	\begin{equation*}
		\sqrt{L} := \left\{f \in L^p(\mathbb{R}^2) \quad \forall p \in [2,\infty) : \|f\|_{\sqrt{L}} < \infty \right\}
		\qquad \text{where} \qquad 
		\|f\|_{\sqrt{L}} := \sup_{p \in [2,\infty)} \frac{\|f\|_{L^p}}{\sqrt{p-1}}.
	\end{equation*}
	The definition of Besov spaces can be found in Appendix B (see Section \ref{sec:app}). Note that in the above table, the mentioned results consider only the case $\lambda_2 = 0$, while $\lambda_2 \in \mathbb{R}$ in Theorem \ref{theo_nu2_mu1_de}.
	
	%
	\subsection{Stability and large-time behavior} 
	%
	
	The stability and large-time behavior problems corresponding to \eqref{B1} have been recently attracted a lot of attention. More precisely, it is well-known that \eqref{B1} has a special stationary solution which represents the hydrostatic equilibrium and is given by  (see \cite{Majda_2003})
	\begin{equation*}
		v_{\textnormal{he}} = (0,0) \qquad \text{and} \qquad \nabla P_{\textnormal{he}} = g\rho_{\textnormal{he}}e_2 \qquad \text{for} \quad e_2 = (0,1),
	\end{equation*}
	where the latter relation is known as the hydrostatic equations (see \cite{Gill_1982}). An example\footnote{In the case of $d$ dimensions with $d \geq 2$, $e_2$ and $x_2$ are replaced by $e_d = (0,...,1)$ and $x_d$, respectively.} is 
	\begin{equation*}
		v_{\textnormal{he}} = (0,0), \qquad \rho_{\textnormal{he}} = \lambda x_2 \qquad \text{and} \qquad P_{\textnormal{he}} = \frac{\lambda}{2}x_2^2 \qquad \text{for} \quad  \lambda \in \mathbb{R}, g = 1.
	\end{equation*}
	
	We are interested in the stability of \eqref{B1} around $(v_{\textnormal{he}},\rho_{\textnormal{he}},P_{\textnormal{he}})$, i.e.,  we focus on the GWP of the following system, for $g = 1$, $\theta = \rho - \rho_{\textnormal{he}}$ and $\pi = P - P_{\textnormal{he}}$
	\begin{equation} \label{B2} \tag{B2}
		\left\{
		\begin{aligned}
			\partial_t v_1 + v \cdot \nabla v_1 + \partial_1 \pi &= \nu_1 \partial_{11} v_1 + \nu_2 \partial_{22} v_1,
			\\
			\partial_t v_2 + v \cdot \nabla v_2 + \partial_2 \pi &= \mu_1\partial_{11} v_2 + \mu_2 \partial_{22} v_2  + \lambda\theta,
			\\
			\partial_t \theta + v \cdot \nabla \theta + \lambda v_2 &= \delta_1 \partial_{11} \theta + \delta_2 \partial_{22} \theta,
			\\
			\textnormal{div}\, v &= 0,
		\end{aligned}
		\right.
		\qquad \text{in} \quad \Omega \times (0,\infty),
	\end{equation}
	with the initial data $(v,\theta)_{|_{t=0}} = (v_0,\theta_0)$ is sufficiently small and satisfies $\text{div}\, v_0 = 0$ and we also investigate the large-time behavior of the perturbations. In particular, we would like to know the contribution of the viscosity and diffusivity coefficients on the stability and large-time behavior of the perturbations. Let us give a brief summary on the stability results in Sobolev spaces in two dimensions as follows:
	
	\begin{table}[htbp]
		\centering
		\begin{tabular}{|c|c|c|l|}
			\hline
			Case & Positive & Zero & \hspace{.5cm} Domain, data and stability results
			\\
			\hline
			\multirow{3}{.5em}{1} & \multirow{3}{3.5em}{$\nu_2,\mu_1,\delta_1$} & \multirow{3}{3.5em}{$\nu_1,\mu_2,\delta_2$} & $\bullet$ $\Omega = \mathbb{R}^2$, $H^1$ in \cite{Ji_Li_Wei_Wu_2019}
			\\
			&&& $\bullet$ $\Omega = \mathbb{R}^2$, $H^2$ in \cite{Chen-Lui_2022,Wei_Li_2021}
			\\
			&&& $\bullet$ $\Omega = \mathbb{R}^2$, $H^m$ for $m \in \mathbb{N}$ in Theorem \ref{theo_nu2_mu1_de2}
			\\
			\hline
			2 & $\nu_2,\mu_1, \delta_2$ & $\nu_1,\mu_2,\delta_1$& $\bullet$ $\Omega = \mathbb{R}^2$, $H^m$ for $m \in \mathbb{N}$ in Theorem \ref{theo_nu2_mu1_de2}
			\\
			\hline
			\multirow{2}{.5em}{3} & \multirow{2}{3.5em}{$\nu_2,\mu_2,\delta_1$} & \multirow{2}{3.5em}{$\nu_1,\mu_1,\delta_2$} & $\bullet$ $\Omega = \mathbb{R}^2$, $H^1$ and $H^2$ in \cite{BenSaid_Pandey_Wu_2022} 
			\\
			&&& $\bullet$ $\Omega = \mathbb{T} \times \mathbb{R}$, $H^2$ in \cite{OussamaBenSaid_MonaBenSaid_2021}
			\\
			\hline
			4 & $\nu_1,\mu_1,\delta_2$ & $\nu_2,\mu_2,\delta_1$ & $\bullet$ $\Omega = \mathbb{T} \times \mathbb{R}$, $H^2$ in \cite{Adhikari_BenSaid_Pandey_Wu_2022}
			\\
			\hline
			5 & $\nu_1,\mu_1,\delta_1$ & $\nu_2,\mu_2,\delta_2$ & $\bullet$ $\Omega = \mathbb{T} \times \mathbb{R}$, $H^1$ and $H^2$ in  \cite{Dong_Wu_Xu_Zhu_2021}
			\\
			\hline
		\end{tabular}
	\end{table}
	In the above table, Case 5 (only horizontal dissipation) has been recently studied in \cite{Ji_Yan_Wu_2022} in three dimensions. We should mention that the case with $\nu_i,\mu_i > 0$ and $\delta_i = 0$, which has been recently investigated in \cite[$\Omega = \mathbb{T}^2$]{Tao_Wu_Zhao_Zheng_2020} as well, where the authors provided a unique global solution to \eqref{B2}, $(v,\theta) \in L^\infty(0,\infty;H^2 \times L^2)$ if $(v_0,\theta_0) \in H^2\times L^2\cap L^\infty$ and in addition  if $(v_0,\theta_0)$ is small in $L^2$ then $(v,\theta)$ is also small in $H^2 \times L^2$ for large time. They also provided the behavior of $(v,\theta)$ in $H^1 \times L^2$ as time goes to infinity. For Case 1 in the above table, the authors in \cite{Chen-Lui_2022} also provided the decay in time of $\partial_2v_1,\partial_1v_2$ and $\partial_1\theta$ in $L^2$ norm, but without explicit rate. Our result (see Theorem \ref{theo_nu2_mu1_de2}) provides $\dot{H}^{m-1}$ large-time behavior for $m \in \mathbb{N},m \geq 2$, but also without explicit rate. We also list open cases in Remark \ref{rm2}.
	
	%
	\subsection{Main results}
	%
	
	In this paper, we focus on the well-posedness, stability and large-time behavior of solutions to $\eqref{B1}$ and $\eqref{B2}$ on the whole space $\mathbb{R}^2$. Before going further, we rewrite both \eqref{B1} and \eqref{B2} in the following form 
	\begin{equation} \label{B} \tag{B}
		\left\{
		\begin{aligned}
			\partial_t v_1 + v \cdot \nabla v_1 + \partial_1 \pi &= \nu_1 \partial_{11} v_1 + \nu_2 \partial_{22} v_1,
			\\
			\partial_t v_2 + v \cdot \nabla v_2 + \partial_2 \pi &= \mu_1\partial_{11} v_2 + \mu_2 \partial_{22} v_2  + \lambda_1\theta,
			\\
			\partial_t \theta + v \cdot \nabla \theta + \lambda_2 v_2 &= \delta_1 \partial_{11} \theta + \delta_2 \partial_{22} \theta,
			\\
			\textnormal{div}\, v &= 0,
		\end{aligned}
		\right.
		\qquad \text{in} \quad \mathbb{R}^2 \times (0,\infty),
	\end{equation}
	with the initial data $(v,\theta)_{|_{t=0}} = (v_0,\theta_0)$ satisfying $\text{div}\, v_0 = 0$. Here $\lambda_1$ and $\lambda_2$ are given real numbers that $\lambda_1 = 1$ and  $\lambda_2 = 0$ or $\lambda_i =  \lambda$ is corresponding to $\eqref{B1}$ or \eqref{B2}. In the rest of the paper, for $s \in \mathbb{R}$ and $p \in [1,\infty]$, we will usually write $H^s$, $L^p$  instead of $H^s(\mathbb{R}^2)$ and $L^2(\mathbb{R}^2)$ for simplicity. Sometimes, we also write $X$ instead of $X \times ... \times X$ for a functional space $X$. Our first result is about the GWP of \eqref{B} and is given as follows. 
	
	\begin{theorem}[Global well-posedness] \label{theo_nu2_mu1_de} 
		Let $\nu_1 = \mu_2 = \delta_i = 0$, $\nu_2,\mu_1,\lambda_i,\delta,s \in \mathbb{R}$ for $i \in \{1,2\}$ with $\nu_2,\mu_1 > 0$ and $\delta,s \geq 0$. Assume that 
		\begin{equation*}
			(v_0,\theta_0) \in H^\delta \times H^s 
			\qquad \text{with} \quad  \textnormal{div}\, v_0 = 0 \quad \text{and}\quad
			\begin{cases}
				\delta \in [s-1,s+1] &\text{if} \quad s > 2,
				\\
				\delta \in (1,s+1] &\text{if} \quad s \in (0,2],
				\\
				\delta \in [0,1] &\text{if} \quad s = 0.
			\end{cases}
		\end{equation*}
		Then \eqref{B} admits a unique solution $(v,\theta)$ satisfying for any $T \in (0,\infty)$ 
		\begin{equation*}
			(v,\theta) \in C([0,T];H^\delta \times H^s)
			\qquad  \text{and} \qquad 
			v \in L^2(0,T;H^{\delta+1}).
		\end{equation*}
		In particular, for $t \in (0,T)$
		\begin{equation*}
			\|v(t)\|^2_{H^\delta} + \|\theta(t)\|^2_{H^s} + \int^t_0 \|v\|^2_{H^{\delta+1}} \,d\tau \leq C(T,\delta,s,\nu_2,\mu_1,\lambda_i,v_0,\theta_0).
		\end{equation*}
	\end{theorem}
	
	Next, we study the stability and large-time behavior of \eqref{B}. It is needed to assume that either $\delta_1 > 0$ or $\delta_2 > 0$. At the moment, it seems to us that it is difficult to get rid of this assumption. It seems to us that the large-time behavior issue is much harder to obtain than the stability one, in which we are able to show the decay in time only in $\dot{H}^{m-1}$ norm for $m \geq 2$, instead of in $H^m$ norm and without explicit rate of convergence. Our second result is stated as follows.
	
	\begin{theorem}[Stability and large-time behavior] \label{theo_nu2_mu1_de2}
		Let $\nu_1 = \mu_2 = 0$, $\nu_2,\mu_1,\lambda_i = \lambda,\delta_i \in \mathbb{R}$ for $i \in \{1,2\}$ with $\nu_2,\mu_1,\lambda > 0$ and $\delta_i \geq 0$. Assume that $(v_0,\theta_0) \in H^m$ with $m \in \mathbb{N}$ and $\textnormal{div}\,v_0 = 0$. There exists a constant $\epsilon_0 = \epsilon_0(\nu_2,\mu_1,\delta_i) > 0$ such that if 
		\begin{equation*}
			\|v_0\|^2_{H^m} + \|\theta_0\|^2_{H^m} \leq \epsilon^2_0
		\end{equation*}
		and either $\delta_1 = 0,\delta_2 > 0$ or $\delta_1 > 0,\delta_2 = 0$, then \eqref{B} has a unique global solution $(v,\theta)$ satisfying for $t > 0$
		\begin{equation*}
			\|v(t)\|^2_{H^m} + \|\theta(t)\|^2_{H^m} + 2\int^t_0 \nu_2\|\partial_2 v_1\|^2_{H^m} + \mu_1\|\partial_1 v_2\|^2_{H^m} + \delta_i\|\partial_i \theta\|^2_{H^m} \,d\tau \leq 2\epsilon^2_0.
		\end{equation*}
		In addition, if $m \geq 2$ then for $t > 0$, $C = C(m,\lambda,\nu_2,\mu_1,\delta_i)$, $j \in \{1,2\}$ with $j \neq i$ and $\delta_i > 0$
		\begin{equation*}
			\int^t_0  \|\partial_j\theta\|^2_{\dot{H}^{m-2}}\,d\tau \leq C(\epsilon_0^2 + \epsilon^4_0) \qquad \text{and} \qquad \|v(t)\|_{\dot{H}^{m-1}},\|\theta(t)\|_{\dot{H}^{m-1}}  \to 0 \quad \text{as}\quad t \to \infty.
		\end{equation*}
	\end{theorem}

	\begin{remark} \label{rm1} We add some comments on Theorem \ref{theo_nu2_mu1_de} as follows:
		\begin{enumerate}
			\item[0.] Let us explain about the form of the initial data and the strategy of the proof:
			
			$\bullet$ The first case $s = 0$. To close $H^\delta$ estimate of $v$, formally we need to bound the terms $(\Lambda^\delta \theta e_2,\Lambda^\delta v)$ and $(\Lambda^\delta(v \cdot \nabla v),\Lambda^\delta v)$. Since we do not have any regularity on $\theta$ then it is natural to move all derivative of $\theta$ to $v$. It leads to bound the terms $(\theta e_2, \Lambda^{2\delta} v)$ and $(v \cdot \nabla v,\Lambda^{2\delta} v)$. Therefore, it is needed $2\delta \leq \delta + 1$ or $\delta \leq 1$ to absorb to the term on the left. In fact, in the case $\delta = 1$, we use the vanishing of the nonlinear term. The range of this case is the segment from $(0,0)$ to $(0,1)$, including the ending points. 
			
			$\bullet$ The second case $s > 0$.  Firstly, similar to the previous case, we also need to estimate the term $(\Lambda^s \theta e_2,\Lambda^{2\delta-s})$, which yields $2\delta-s \leq \delta+1$ or $\delta \leq s+1$ to close the estimate. Secondly, to bound the term $(\Lambda^s(v \cdot \nabla \theta),\Lambda^s \theta)$ (or $\lambda_2(\Lambda^s v_2,\Lambda^s \theta)$ if $\lambda_2 \neq 0$ and similarly for $J^s$ instead of $\Lambda^s$), and somehow we need to control the term $\|\Lambda^s v\|_{L^2}$ by the one $\|\Lambda^{\delta+1}v\|_{L^2}$ on the left. Thus, one should assume that $\delta + 1 \geq s$ or $\delta \geq s-1$. Thirdly, if $s \in (0,2)$ then we should apply the Brezis-Gallouet type inequality (see \cite{Brezis-Gallouet_1980,Hu-Kukavica-Ziane_2015}) to bound the norm $\|\nabla v\|_{L^\infty}$, where it is needed to control the norm  $\|v\|_{H^{s_0+2}}$ in the logarithmic term for some $s_0 \in (0,1]$. Therefore, we need $\delta + 1 \geq  2 + s_0$ or $\delta > 1$. For $s > 2$, the norms $\|\nabla v\|_{L^\infty}$ and $\|\nabla \theta\|_{L^\infty}$ can be bounded by $\|v\|_{H^s}$ and $\|\theta \|_{H^s}$, respectively. Thus, the a priori estimates follow easily. Formally, from these three cases, one can close the estimate for: a) the segment from $(0,0)$ to $(0,1)$ including the ending points; b) $\delta \in [s-1,s+1]$, $\delta > 1$ and $s > 0$; which are the dark black lines and gray region in Figure 1.
			
			$\bullet$ The case $s \in (0,2]$ and $\delta \in [0,1]$ with $\delta > s-1$. In this case, it seems to us that even the local existence is unknown (the global existence and uniqueness of $L^2$ weak solutions in this case are followed by the $L^2$ data one), see the region under the gray one in Figure 1, excluding the segment from $(0,0)$ to $(0,1)$. The main difficulty is we are not able to control the term $\|\nabla v\|_{L^1_tL^\infty_x}$, where with the help from partial viscosity terms, we usually need $v_0 \in H^{1^+}$. It is left as an open question for the interested readers.
			
			$\bullet$ We consider an approximate system by using Fourier truncation method. Then, by working carefully on each case, we obtain uniform bounds in terms of the regularization parameter. As usual, we can pass to the limit and prove the uniqueness of solutions for $s > 1$. The case $s \in [0,1]$ is much more complicated, especially for the case $s = \delta = 0$, where in order to obtain the uniqueness, we need to localize the $\theta$ equation in a suitable way and use the Littlewood–Paley decomposition to get a bound on $\|v\|_{L^2_tL^\infty_x}$, which allows us to bound $p^{-1}\|\nabla v\|_{L^1_tL^p_x}$ for $p \in [2,\infty)$. 
	
			\begin{center}
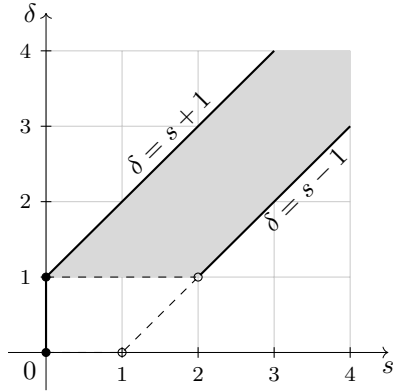

				\begin{tikzpicture} 
					\draw[step=1,help lines,black!20] (0,0) grid (4,4);
				
					\draw[very thin, black,->] (-0.5,0) -- (4.5,0) node[below]{$s$};
					\draw[very thin,black,->] (0,-0.5) -- (0,4.5) node[left]{$\delta$};
				
					\foreach \x in {1,2,...,4}
					\draw[shift={(\x,0)}] (0pt,2pt) -- (0pt,-2pt) node[below] {\footnotesize $\x$};
				
					\foreach \y in {1,2,...,4}
					\draw[shift={(0,\y)}] (2pt,0pt) -- (-2pt,0pt) node[left] {\footnotesize $\y$};
					
					\fill[fill=black!15] (0,1) to (2,1) to (4,3) to (4,4) to (3,4) to (0,1);
					
					\draw[fill=black] (0,0) circle (1.5pt) node[below left] {$0$};
					\draw[fill=black] (0,1)  circle (1.5pt)  {};
					\draw[thick,color=black] (0,0) -- (0,1);
					\draw[thick,color=black] (0,1) -- (3,4); 
					
					\draw [] (2.2,3.5) node[left, rotate=45] {$\delta = s+1$};
					\draw[dashed,color=black] (0,1) -- (1.973,1);
					\draw[dashed,color=black] (1.026,0.026) -- (1.973,0.973);
					\draw[dashed,color=black] (0,0) -- (0.973,0);
					\draw[thick,color=black] (2.026,1.026) -- (4,3) node[below=0.25cm, left, rotate=45] {$\delta = s-1$};
					\draw[color=black] (2,1)  circle[radius=1.5pt] node[below left] {};
					\draw[color=black] (1,0)  circle[radius=1.5pt] node[below left] {};
				\end{tikzpicture}
				\captionof{figure}{Relation between $s$ and $\delta$.}
			\end{center}

			\item[1.] We remark that the LWP in $H^s$ for $s \in \mathbb{R}, s > 2$ of the inviscid version of \eqref{B} follows as that of in \cite{Chae_Nam_1997} in $H^m$ for $m \in \mathbb{N}, m \geq 3$, $\lambda_2 = 0$ and $\lambda_1$ is replaced by a function $f$, see also  \cite{Kim-Lee_2022}, with $\lambda_1 = \lambda_2 = 1$ and a velocity damping term instead of partial vicosity. However, the GWP in this case is open. In addition, as it can be seen in the proof of Theorem \ref{theo_nu2_mu1_de} given in Section \ref{sec:wp}  that the result also holds if $\lambda_1$ is replaced by a function $g$ with 
			\begin{equation*} g \in 
				\begin{cases}
					L^\infty_{\textnormal{loc}}(0,\infty;L^\infty(\mathbb{R}^2)) &\text{if}\quad s = 0,
					\\
					L^\infty_{\textnormal{loc}}(0,\infty;\dot{W}^{\delta,\infty}(\mathbb{R}^2)) &\text{if}\quad s > 0.
				\end{cases}
			\end{equation*}
		    
			\item[2.] As mentioned before, if $\lambda_2 = \nu_1 = \mu_2 = \delta_i = 0$ and $\lambda_1,\nu_2,\mu_1 > 0$ then the authors in \cite[$\Omega$ bounded or $\Omega = \mathbb{R}^2$]{Hu-Wang-Wu-Xiao-Yuan_2018}  and \cite[$\Omega = \mathbb{T}^2$]{He-Ma-Sun_2022} proved the GWP to \eqref{B} with $H^1 \times L^2 \cap L^\infty$ and $L^2$ data, respectively. Note that the proof in \cite{Hu-Wang-Wu-Xiao-Yuan_2018} is given for the case $\Omega$ bounded with different types of boundary conditions and they only mentioned that the case $\Omega = \mathbb{R}^2$ follows as a corollary. In addition, their proof used $L^\infty$ bound of $\theta$ which is obtained directly if $\lambda_2 = 0$ and $\theta_0 \in L^\infty$. However, if $\lambda_2 \neq 0$ then it is not clear (at least for us) how their proof can be applied directly. Moreover, \cite{He-Ma-Sun_2022} gave a proof in the case $\Omega = \mathbb{T}^2$, the authors also mentioned that it is possible to extend their result to the whole space, but without giving a proof. Furthermore, their proof is also can not be applied directly in our case. Especially, the uniqueness for $L^2$ data is obtained by working more carefully in the case of whole space and $\lambda_2 \neq 0$, where we introduce a new way to localize the $\theta$ equation and also use the idea in \cite{Boardman-Ji-Qiu-Wu_2019}, where the authors considered the full Laplacian of $v$.
			
			\item[3.] The GWP of \eqref{B} in the case of with full viscosity and without diffusivity coefficients has been considered in \cite[$\Omega = \mathbb{T}^2$]{Tao_Wu_Zhao_Zheng_2020} with $H^2 \times L^2 \cap L^\infty$ data,
			in  \cite[$\Omega$ bounded]{Doering_Wu_Zhao_Zheng_2018} under the assumption $(v_0,\theta_0) \in H^2 \times H^1 \cap L^\infty$ and in \cite[$\Omega = \mathbb{R}^2$]{Boardman-Ji-Qiu-Wu_2019}, \cite[$\Omega$ bounded]{He_2012} for lower regularity $L^2$ data, but only for $\lambda_2 = 0$. As pointed out in the introduction, the authors in \cite{Hu-Kukavica-Ziane_2015} have been provided global in time estimates either in $H^s \times H^{s-1}$ with $s \in (1,2)$ or in $H^s$ with $s \in \mathbb{R},s > 1$ for $\lambda_2 = 0$. For the latter case, they gave a sketch of the proof for $s\in (1,2]$, 
			confirmed that the case $s \in (2,3)$ follows in the same way as $s \in (1,2)$ 
			and referred to \cite{Chae_2006} for $s \in \mathbb{R}, s \geq 3$. 
			It seems to us that Theorem \ref{theo_nu2_mu1_de} for $s = 0$ and $\delta \in [0,1]$ or $s > 0$ and $s-1 \leq \delta < s$ and $s < \delta < s+1$, $\lambda_2 \in \mathbb{R}$ and on the whole space are new points and we extend the known results in the first table given in the introduction. The authors in \cite{Fefferman-McCormick-Robinson-Rodrigo_2014} have been suggested a question on finding a "simpler" model (rather than the non-resistive MHD) in which the failure of local existence in $(H^{s-1} \times H^s)(\mathbb{R}^d)$ with $s > \frac{d}{2}$, $d \in \{2,3\}$, can be obtained. We show here that it is not the case for \eqref{B}, at least in the 2D case and for $s > 2$. If $\mathbb{R}^2$ is replaced by a bounded $\Omega \subset \mathbb{R}^2$ with smooth boundary then the technique in \cite{He_2012} can be applied to investigate $\eqref{B}$ for $L^2$ data. 
			 
			\item[4.] In \eqref{B}, if $\theta \equiv 0$, $\lambda_2 = 0$ and $\nu_2,\mu_1 > 0$ then by providing the corresponding kernel, using the usual Fourier splitting method and assuming $v_0 \in L^1$, we obtain the Navier-Stokes equations with partial viscosity and can prove the $L^2$ decay in time of $v$ with explicit rate as $\log(e+t)^{-\frac{m}{2}}$ for $m \geq 3$ (see  Appendix C in Section \ref{sec:app}) as the full dissipation case in \cite{Schonbek_1986}. However, in \eqref{B} if $\theta \neq 0$, $\lambda_1 = 1$, $\nu_2,\mu_1 > 0$ and $\lambda_2 = \delta_i = 0$ then the GWP is given by Theorem \ref{theo_nu2_mu1_de}, and under further suitable assumptions on $(v_0,\theta_0)$ the technique in \cite{Kiselev-Park-Yao_2022} can be applied to provide the growth of $\theta$ in $\dot{H}^{s}$ norm for $s \in \mathbb{R},s > 1$.
			
			\item[5.] In Case 6, for $\nu_2,\mu_2 > 0$ and $\nu_1 = \mu_1 = \delta_i = 0$, it has been considered in \cite{Adhikari-Cao-Shang-Wu-Xu-Ye_2016}, where the authors have been provided some useful estimates, but have not been obtained GWP yet. In addition, the case $\mu_1,\delta_2 > 0$ and $\nu_i = \mu_2 = \delta_1 = 0$ is also open (see \cite{Wu-Xu-Zhuan_2017}), in this case a global regularity criteria (even for $\mu_1 = 0$) can be found in \cite{Adhikari-Cao-Shang-Wu-Xu-Ye_2016}.
		\end{enumerate}
	\end{remark}
	
	\begin{remark} \label{rm2} We add some comments on Theorem \ref{theo_nu2_mu1_de2} as follows:
		\begin{enumerate}
			\item[0.] Let us explain about the strategy of the proof:
			
			$\bullet$ We first introduce the corresponding approximate system, which allows us to prove local existence and then doing all calculations. Here, we use again the idea in \cite{Majda_Bertozzi_2002}, where the authors considered the Euler and Navier-Stokes equations for $m \in \mathbb{N},m \geq 3$, but it seems to us that their proof does not work directly for the cases $m = 1$ and $m = 2$, where the proof is not standard, especially without the smallness on the initial data. 
			
			$\bullet$ We then define a suitable energy form $E_m(t)$ in which we start from $m = 1$, $m=2$ and $m \geq 3$ and obtain 
			\begin{equation*}
				E^\epsilon_m(t) \leq E^\epsilon_m(0) + C(\nu_2,\mu_1,\delta_i) (E^\epsilon_m(t))^\frac{3}{2} \qquad \text{for} \quad t \in (0,T_\epsilon),
			\end{equation*}
			where $T_\epsilon$ comes from the local existence and $\epsilon$ denotes regularization parameter, which by using the smallness of initial data and the bootstrapping argument allow us to conclude that $T_\epsilon = \infty$ and obtain the uniform (in terms of $t$ and $\epsilon$) global bound of $E^\epsilon_m(t)$ by $2E_m(0)$. We then can pass to the limit and prove the uniqueness of solutions as usual.
			
			$\bullet$ To obtain the large-time behavior in $\dot{H}^{m-1}$ norm for $m \in \mathbb{N}, m \geq 2$, we prove  
			\begin{align*}
				\left(\|v\|^2_{\dot{H}^{m-1}} + \|\theta\|^2_{\dot{H}^{m-1}}\right)(t-s) &\leq C(\nu_2,\mu_1,\delta_i,\epsilon_0) (t-s) \qquad \text{for} \quad 0 \leq s < t < \infty,
				\\
				\int^\infty_0  \|v\|^2_{\dot{H}^{m-1}} + \|\theta\|^2_{\dot{H}^{m-1}} \,dt &\leq C(\nu_2,\mu_1,\delta_i,\epsilon_0),
			\end{align*}
			which yields the convergence  $\|v(t)\|_{\dot{H}^{m-1}} + \|\theta(t)\|_{\dot{H}^{m-1}} \to 0$ as $t \to \infty$.
			
			\item[1.] As mentioned in the introduction, $H^1$ and $H^2$ stability have recently investigated in \cite{Chen-Lui_2022,Ji_Li_Wei_Wu_2019,Wei_Li_2021} for the case $\delta_1 > 0$ and $\delta_2 = 0$. In fact, they only provided the main estimates, not the full proof. To the best of our knowledge, the case $\delta_1 = 0$ and $\delta_2 > 0$ has not been studied yet. So, we provide a complete proof and apply our proof to the other case when $\delta_1 > 0$ and $\delta_2 = 0$. 
			
			\item[2.] In the proof of Theorem \ref{theo_nu2_mu1_de2} (see Step 8, Section \ref{sec:s_ltb}), the decay in time of $\partial_2v_1$ and $\partial_1v_2$ in $\dot{H}^{m-2}$ norm for $m \in \mathbb{N},m \geq 2$ is obtained and then the decay of $v$ in $\dot{H}^{m-1}$ norm follows as a consequence. Remark that the authors in \cite{Chen-Lui_2022} only provide the decay in $L^2$ norm of $\partial_2v_1,\partial_1v_2$ and $\partial_1\theta$. In addition, the decay in time in $\dot{H}^{m-1}$ norm for $m \geq 2$ of $(v,\theta)$ also implies the decay in other norms such as in $L^p$ norm for $p \in (2,\infty]$, in $\dot{W}^{m-1,q}$ for $q \in (2,\infty]$. Furthermore, there are more advantages in the study large-time behavior in the case $\mathbb{R}^2$ is replaced by $\Omega = \mathbb{T}^2, \mathbb{T} \times \mathbb{R}$, see \cite{OussamaBenSaid_MonaBenSaid_2021,Tao_Wu_Zhao_Zheng_2020}, especially, in the case of horizontal dissipation on space domain $\Omega = \mathbb{T} \times \mathbb{R}$. 
			
			\item[3.] Our proof can be applied to the following cases (at least for $H^1$ or $H^2$ stability) 
			\begin{align*}
				&a.\quad \nu_i,\delta_1 > 0  \quad \text{and}\quad \mu_i = \delta_2 = 0; 
				\\
				&b. \quad \mu_i,\delta_2 > 0 \quad \text{and}\quad \nu_i = \delta_1 = 0;
				\\
				&c. \quad \delta_i > 0 \quad \text{and}\quad \nu_i = \mu_i = 0.
			\end{align*}
			It seems to us that the stability and large-time behavior in the following cases have not been investegated yet:
			\begin{align*}
				&A.\quad \nu_i,\delta_2 > 0 \quad \text{and}\quad \mu_i = \delta_1 = 0;
				&&B.\quad \mu_i,\delta_1 > 0 \quad\text{and}\quad \nu_i = \delta_2 = 0;
				\\
				&C.\quad \nu_2, \mu_2,\delta_2 > 0 \quad\text{and}\quad  \nu_1 = \mu_1 = \delta_1 = 0;
				&&D.\quad \nu_1,\mu_2, \delta_1 > 0 \quad\text{and}\quad \nu_2 = \mu_1 = \delta_2 = 0;
				\\
				&E.\quad \nu_1,\mu_2, \delta_2 > 0 \quad\text{and}\quad \nu_2 = \mu_1 = \delta_1 = 0;
				&&F.\quad \nu_i,\mu_1 > 0 \quad\text{and}\quad \mu_2 = \delta_i = 0;
				\\
				&G.\quad \nu_i,\mu_2 > 0 \quad\text{and}\quad \mu_1 = \delta_i = 0;
				&&H.\quad \nu_1, \mu_i > 0 \quad\text{and}\quad \nu_2 = \delta_i= 0;
				\\
				&I.\quad \nu_2,\mu_i > 0 \quad\text{and}\quad \nu_1 = \delta_i = 0.
			\end{align*}
		\end{enumerate}
	\end{remark}
	\begin{remark}
		\begin{enumerate} \label{r3} We add further comments on the GWP and stability as follows:
			\item[1.] We remark that the assumption $\nu_2 > 0$ and $\mu_1 > 0$ implies the global in time control of $L^2$ of the vorticity $w = \partial_1 v_2 - \partial_2 v_1$ and thus yields $L^2$ of the full velocity gradient $\nabla v$ as well by using the divergence-free condition of $v$. This is also true for higher regularity in Sobolev norms. This observation is mentioned and used repeatedly, see for example the proof of 
			Theorem 1.1, Step 2a. 
			
			\item[2.] It is worth mentioning stability results proved in \cite[$\Omega = \mathbb{T} \times \mathbb{R}$]{Adhikari_BenSaid_Pandey_Wu_2022},  \cite[$\Omega = \mathbb{R}^2$]{Lai-Wu-Zhong_2021} and \cite[ $\Omega = \mathbb{R}^2$]{BenSaid_Pandey_Wu_2022}, which were obtained by the effect of enhanced dissipation caused by the damped wave structure of the corresponding systems with the dissipation terms given by $(\partial_{11}v,\partial_{22} \theta)$, $(\partial_{22}v,\theta)$ and $(\partial_{22}v,\partial_{11}\theta)$, respectively.
			More precisely, this hidden wave structure provides more smoothing and stabilizing properties compared to the original systems. It also allows us to define suitable energy forms (which contain extra crucial terms) and to obtain good energy bounds in which the usual bootstrap argument can be applied using the smallness of the initial data. 
			
			\item[3.] Previously, in Remark 1.2-3, we listed some cases in which the stability has not been considered yet. However, as pointed out by the referee, the cases $A$ (for $\nu_1 = \nu_2 > 0$ and $\delta_2 > 0$) and $B$ (for $\mu_1 = \mu_2 > 0$ and $\delta_1 > 0$) can be solved by using the techniques of enhanced dissipation provided in \cite[$\Omega = \mathbb{T} \times \mathbb{R}$]{Adhikari_BenSaid_Pandey_Wu_2022} and \cite[ $\Omega = \mathbb{R}^2$]{BenSaid_Pandey_Wu_2022}. In addition, in bounded domains $\Omega$ with suitable boundary conditions, the cases $F$ ($\nu_1,\nu_2,\mu_1 > 0$) and $I$ ($\nu_2,\mu_1,\mu_2 > 0$) can be resolved as in \cite[$\Omega$ bounded]{Doering_Wu_Zhao_Zheng_2018} as well. However, it is not clear, mainly due to the lack of a Poincar\'{e}-type inequality, whether or not the same results could be proved in the latter cases with $\Omega = \mathbb{R}^2$. On the other hand, due to a different structure in the case $G$ ($\nu_1,\nu_2,\mu_2 > 0$) compared to $F$ and $I$, the $H^2$ stability does not seem to be obvious only by adapting the idea in \cite[$\Omega$ bounded]{Doering_Wu_Zhao_Zheng_2018}.
			
		\end{enumerate}
	\end{remark}

	The rest of the paper is organized as follows: The proofs of Theorems \ref{theo_nu2_mu1_de} and \ref{theo_nu2_mu1_de2}, and some technical estimates are provided in Sections \ref{sec:wp}, \ref{sec:s_ltb} and \ref{sec:app}, respectively. 
	
	%
	\section{Proof of Theorem \ref{theo_nu2_mu1_de}} \label{sec:wp}
	%
	
	In this section, we will provide a proof of Theorem \ref{theo_nu2_mu1_de}, which uses the ideas in \cite{Boardman-Ji-Qiu-Wu_2019, Fefferman-McCormick-Robinson-Rodrigo_2014,He-Ma-Sun_2022,Hu-Kukavica-Ziane_2015,Hu-Wang-Wu-Xiao-Yuan_2018,Majda_Bertozzi_2002}. 
	
	\begin{proof}[Proof of Theorem \ref{theo_nu2_mu1_de}]
		The proof contains several steps as follows.
		
		\textbf{Step 1: Approximate system and local existence.} Let us fix $n \in \mathbb{R}$ with $n > 0$. Assume that $(v_0,\theta_0) \in H^\delta \times H^s$ with $\text{div}\,v_0 = 0$. An approximate system of \eqref{B} is given by 
		\begin{equation} \label{B_app} 
			\frac{d}{dt} (v^n,\theta^n) =
			(F^n_1,F^n_2)(v^n,\theta^n),
			\quad \text{div}\,v^n = 0 \quad \text{and} 
			\quad (v^n,\theta^n)(0) = T_n(v_0,\theta_0),
		\end{equation}
		where 
		\begin{align*}
			F^n_1(v^n,\theta^n) &:= -\mathbb{P}(T_n(v^n \cdot \nabla v^n)) + \mathbb{P}(\nu_2\partial_{22}v^n_1,\mu_1\partial_{11}v^n_2) + \lambda_1\mathbb{P}(\theta^n e_2),
				\\
			F^n_2(v^n,\theta^n) &:= -T_n(v^n \cdot \nabla \theta^n) - \lambda_2 v^n_2.
		\end{align*}	
		Here $T_n$ and $\mathbb{P}$ are the usual Fourier truncation operator and Leray projection\footnote{As usual, the Fourier transform, $T_n$ and $\mathbb{P}$ are defined by
		\begin{align*}
			\mathcal{F}(f)(\xi) &:= \int_{\mathbb{R}^2} \exp\{-i\xi\cdot x\} f(x) \,dx \quad \text{for } \xi \in \mathbb{R}^2,
			\\
			\mathcal{F}(T_n(f))(\xi) &:= \textbf{1}_{B_n}(\xi) \mathcal{F}(f)(\xi) \quad \text{for } n \in \mathbb{R}, n > 0,\xi \in \mathbb{R}^2,
			\\
			\mathbb{P}(f) &:= f + \nabla(-\Delta)^{-1} \text{div} \,f.
		\end{align*} 
		Here $\textbf{1}_{B_n}$ is the characteristic function of $B_n$, where $B_n$ is the ball of radius $n$ centered at the origin.}, respectively. 
		For $s \in \mathbb{R}, s \geq 0$, we define
		\begin{align*}
			H^s_n &:= \left\{h \in H^s : \text{supp}(\mathcal{F}(h)) \subseteq B_n\right\}, \quad 
			V^s_n := \left\{h \in H^s_n : \text{div}\, h = 0\right\},
			\\
			F^n &: V^\delta_n \times H^s_n \to V^\delta_n \times H^s_n
			\qquad \text{with}\qquad
			(v^n,\theta^n) \mapsto F^n(v^n,\theta^n) := (F^n_1,F^n_2).
		\end{align*}
		The space $V^\delta_n \times H^s_n$ is equipped with the following norm\footnote{For $s \in \mathbb{R}$, $\mathcal{F}(J^sf)(\xi) := (1 + |\xi|^2)^{\frac{s}{2}} \mathcal{F}(f)(\xi)$  for $\xi \in \mathbb{R}^2$ and $\|f\|_{H^s} := \|J^s f\|_{L^2}$ with $H^0 \equiv L^2$.}
		\begin{equation*}
			\|(v^n,\theta^n)\|_{\delta,s} := \|v^n\|_{H^\delta} + \|\theta^n\|_{H^s}. 
		\end{equation*}
		It can be checked that $F^n$ is well-defined and is locally Lipschitz continuous as well. Then the Picard theorem (see \cite[Theorem 3.1]{Majda_Bertozzi_2002}) implies that there exists a unique solution $(v^n,\theta^n) \in C^1([0,T^n_*),V^\delta_n \times H^s_n)$ for some $T^n_* > 0$. In addition, if $T^n_* < \infty$ then (see \cite[Theorem 3.3]{Majda_Bertozzi_2002})
		\begin{equation*}
			\lim_{t\to T^n_*} \left(\|v^n(t)\|_{H^\delta} + \|\theta^n(t)\|_{H^s}\right) = \infty.
		\end{equation*}
		
		\textbf{Step 2: Global existence and uniform bound.}
		We will assume $T^n_* < \infty$ and prove that
		\begin{equation*}
			\esssup_{t \in (0,T^n_*)} \left(\|v^n(t)\|^2_{H^\delta} + \|\theta^n(t)\|^2_{H^s}\right)  < \infty,
		\end{equation*}
		which leads to a contradiction with the previous step and $T^n_* = \infty$. Note that if $(v^n,\theta^n) \in V^\delta_n \times H^s_n$ then $T_n(v^n,\theta^n) = (v^n,\theta^n)$ in $L^2$ sense. To avoid repeating, we will give the full proof only for some cases here, and skip other cases in which whose proof share similar idea. However, their full proofs can be found in Appendix A (see Section \ref{sec:app}). 
		 
		\textbf{2a) The case $s = 0$ and $\delta \in [0,1]$.} We will consider the cases $\delta = 0$, $\delta \in (0,1)$ and $\delta = 1$, respectively, as follows. 
		
		$\bullet$ For $\delta = 0$, by multiplying $(v^n,\theta^n)$ to \eqref{B_app}, we obtain the energy estimates
		\begin{align*}
			\frac{1}{2}\frac{d}{dt} \|v^n\|^2_{L^2} + \min\{\nu_2,\mu_1\} \|\nabla v^n\|^2_{L^2} &\leq |\lambda_1|\|v^n\|_{L^2}\|\theta^n\|_{L^2},
			\\
			\frac{1}{2}\frac{d}{dt} \|\theta^n\|^2_{L^2}  &\leq |\lambda_2|\|v^n\|_{L^2}\|\theta^n\|_{L^2},
		\end{align*}
		where we used the identity that $\|\nabla v^n\|_{L^2} = \|w^n\|_{L^2}$ with $w^n := \partial_1v^n_2 - \partial_2v^n_1$, the 2D vorticity. It implies that for $t \in (0,T^n_*)$ 
		\begin{equation*}
			\|v^n(t)\|^2_{L^2} + \|\theta^n(t)\|^2_{L^2} + \int^t_0 \|v^n\|^2_{H^1} \,d\tau \leq C(T^n_*,\nu_2,\mu_1,\lambda_i,v_0,\theta_0).
		\end{equation*}
		
		$\bullet$ For $\delta \in (0,1)$, we apply $\Lambda^\delta$ to the equation of $v^n$ and multiply $\Lambda^\delta v^n$ to find that\footnote{Recall that for $s \in \mathbb{R}$, $\|f\|_{\dot{H}^s} := \|\Lambda^s f\|_{L^2}$ where  $\mathcal{F}(\Lambda^sf)(\xi) := |\xi|^s\mathcal{F}(f)(\xi)$ for $\xi \in \mathbb{R}^2$.} 
		\begin{equation*}
			\frac{1}{2}\frac{d}{dt} \|v^n\|^2_{\dot{H}^\delta}  + \min\{\nu_2,\mu_1\}\|v^n\|^2_{\dot{H}^{\delta+1}} \leq I_{01} + I_{02},
		\end{equation*}
		where we used the fact that $\nabla w^n = (\Delta v^n_2,-\Delta v^n_1)$ to obtain the second term on the left-hand side and for some $\epsilon \in (0,1)$, since $2\delta \leq \delta + 1$
		\begin{align*}
			I_{01} &:= \lambda_1\int_{\mathbb{R}^2} \Lambda^\delta(\mathbb{P}(\theta^n e_2)) \cdot \Lambda^\delta v^n\,dx \leq C(T^n_*,\nu_2,\mu_1,\lambda_i,v_0,\theta_0,\epsilon) + \epsilon\min\{\nu_2,\mu_1\} \|v^n\|^2_{\dot{H}^{\delta+1}};
			\\
			I_{02} &:= -\int_{\mathbb{R}^2} \Lambda^\delta(\mathbb{P} (T_n (v^n \cdot \nabla v^n)) \cdot \Lambda^\delta v^n\,dx
			\leq \|v^n\|_{L^\frac{2}{1-\delta}} \|\nabla v^n\|_{L^2} \|\Lambda^{2\delta} v\|_{L^\frac{2}{\delta}} 
			\\
			&\leq C(\delta,\nu_2,\mu_1,\epsilon)\|v^n\|^2_{\dot{H}^\delta} \|\nabla v^n\|^2_{L^2} + \epsilon \min\{\nu_2,\mu_1\}\|v^n\|^2_{\dot{H}^{\delta+1}}, 
		\end{align*}
		here we used the following Sobolev inequalities (see \cite{Bahouri-Chemin-Danchin_2011})
		\begin{align*}
			&&\|f\|_{L^{p_0}} &\leq C(p_0,s_0) \|f\|_{\dot{H}^{s_0}}  & &\text{for}\quad s_0 \in [0,1), p_0 = \frac{2}{1-s_0},&&
			\\
			&&\|f\|_{\dot{H}^{s_1}} &\leq C(s_1,s_2) \|f\|^{\alpha_0}_{L^2} \|f\|^{1-\alpha_0}_{\dot{H}^{s_2}} &
			&\text{for}\quad s_1,s_2 \in (0,\infty), s_1 < s_2, \alpha_0 = 1 - \frac{s_1}{s_2}.&&
		\end{align*}
		By choosing $\epsilon = \frac{1}{4}$, for $t \in (0,T^n_*)$
		\begin{equation*}
			\|v^n(t)\|^2_{H^\delta} + \|\theta^n(t)\|^2_{L^2} + \int^t_0 \|v^n\|^2_{H^{\delta+1}} \,d\tau \leq C(T^n_*,\delta,\nu_2,\mu_1,\lambda_i,v_0,\theta_0).
		\end{equation*}
		
		$\bullet$ Similarly, if $\delta = 1$ then
		\begin{equation*}
			\frac{d}{dt} \|v^n\|^2_{\dot{H}^1} + \min\{\nu_2,\mu_1\} \|v^n\|^2_{\dot{H}^2} \leq C(\nu_2,\mu_1,\lambda_1)\|\theta^n\|^2_{L^2},
		\end{equation*}
		which yields for $t \in (0,T^n_*)$
		\begin{equation*}
			\|v^n(t)\|^2_{H^1} + \|\theta^n(t)\|^2_{L^2} + \int^t_0 \|v^n\|^2_{H^2}\,d\tau \leq C(T^n_*,\nu_2,\mu_1,\lambda_i,v_0,\theta_0).
		\end{equation*}
		
		\textbf{2b) The case $s \in (0,1)$ and $\delta \in (1,s+1]$.} Similarly,
		\begin{equation*} 
			\frac{1}{2}\frac{d}{dt}\|v^n\|^2_{\dot{H}^\delta}  + \min\{\nu_2,\mu_1\}\|v^n\|^2_{\dot{H}^{\delta+1}} \leq I_1 + I_2,
		\end{equation*}
		where for some $\epsilon \in (0,1)$, since $2\delta - s \leq \delta + 1$
		\begin{align*}
			I_1 &:= \lambda_1 \int_{\mathbb{R}^2} \Lambda^s (\mathbb{P}(\theta^n e_2)) \cdot \Lambda^{2\delta-s} v^n \,dx 
			\\
			&\leq C(T^n_*,\nu_2,\mu_1,\lambda_i,v_0,\theta_0,\epsilon) \left(\|\theta^n\|^2_{\dot{H}^s} + 1\right) + \epsilon \min\{\nu_2,\mu_1\} \|v^n\|^2_{\dot{H}^{\delta+1}},
		\end{align*}
		and for $\sigma := \delta - 1 \in (0,s] \subset (0,1)$,
		\begin{align*}
			 I_2 &:= -\int_{\mathbb{R}^2} \Lambda^\sigma(v^n \cdot \nabla v^n) \cdot \Lambda^{\delta+1} v^n \,dx
			 \\
			 &\leq C(\delta) \|\Lambda^{\delta+1}v^n\|_{L^2} \times
			 \begin{cases}
			 	\|\Lambda^{\sigma} v^n\|_{L^4}\|\nabla v^n\|_{L^4} + \|v^n\|_{L^\frac{4}{1-2\sigma}}\|\Lambda^\sigma\nabla v^n\|_{L^\frac{4}{1+2\sigma}} \quad &\text{if} \quad \sigma \in (0,\frac{1}{2}),
			 	\\
			 	\|\Lambda^\sigma v^n\|_{L^4}\|\nabla v^n\|_{L^4} + \|v^n\|_{L^6}\|\Lambda^\sigma \nabla v^n\|_{L^3} \quad &\text{if}\quad \sigma \in [\frac{1}{2},1),
			 \end{cases}
		\end{align*}
		here we used the following inequality (see \cite{Grafakos-Oh_2014}) for $1 < p_i,q_i \leq \infty$, $i \in \{1,2\}$, $s_0 > 0$ and $C = C(s_0,p_i,q_i) > 0$ 
		\begin{equation*}
			\|\Lambda^{s_0}(fg)\|_{L^2} \leq C\left(\|\Lambda^{s_0} f\|_{L^{p_1}}\|g\|_{L^{q_1}} + \|f\|_{L^{p_2}}\|\Lambda^{s_0} g\|_{L^{q_2}} \right) \qquad \text{with} \quad \frac{1}{p_i} + \frac{1}{q_i} = \frac{1}{2}.
		\end{equation*}
		Moreover, 
		\begin{align*}
			\|\Lambda^\sigma v^n\|_{L^4}, \|v^n\|_{L^\frac{4}{1-2\sigma}} &\leq C(\delta)\|\Lambda^{\sigma + \frac{1}{2}} v^n\|_{L^2} \leq C(\delta)\|v^n\|^\frac{3}{2(\delta+1)}_{L^2} \|\Lambda^{\delta+1} v^n\|^\frac{2\delta-1}{2(\delta+1)}_{L^2},
			\\
			\|\nabla v^n\|_{L^4}, \|\Lambda^\sigma \nabla v^n\|_{L^\frac{4}{1+2\sigma}} &\leq C(\delta)\|\Lambda^\frac{1}{2} \nabla v^n\|_{L^2} \leq C(\delta)\|\nabla v^n\|^\frac{2\delta-1}{2\delta}_{L^2} \|\Lambda^{\delta+1} v^n\|^\frac{1}{2\delta}_{L^2},
			\\
			\|v^n\|_{L^6} &\leq C\|v^n\|^\frac{1}{3}_{L^2} \|\nabla v^n\|^\frac{2}{3}_{L^2}, 
			\\
			\|\Lambda^\sigma \nabla v^n\|_{L^3} &\leq C(\delta)\|\Lambda^{\sigma + \frac{1}{3}} \nabla v^n\|_{L^2} \leq  C(\delta)\|\nabla v^n\|^\frac{2}{3\delta}_{L^2} \|\Lambda^{\delta+1} v^n\|^\frac{3\delta-2}{3\delta}_{L^2},
		\end{align*} 
		which yields
		\begin{align*}
			I_2 &\leq C(\delta)\|v^n\|^\frac{3}{2(\delta+1)}_{L^2} \|\nabla v^n\|^\frac{2\delta-1}{2\delta}_{L^2} \|\Lambda^{\delta+1} v^n\|^\frac{4\delta^2+2\delta+1}{2(\delta+1)\delta}_{L^2} + C(\delta)\|v^n\|^\frac{1}{3}_{L^2} \|\nabla v^n\|^\frac{2\delta+2}{3\delta}_{L^2} \|\Lambda^{\delta+1} v^n\|^\frac{6\delta-2}{3\delta}_{L^2}.
		\end{align*}
		Therefore, by choosing $\epsilon = \frac{1}{6}$ 
		\begin{equation*}
			\frac{d}{dt}\|v^n\|^2_{H^\delta}  + \min\{\nu_2,\mu_1\}\|v^n\|^2_{H^{\delta+1}} \leq C(T^n_*,\delta,\nu_2,\mu_1,\lambda_i,v_0,\theta_0) \left(\|\theta^n\|^2_{\dot{H}^s} + 1\right). 
		\end{equation*}
		In addition, 
		\begin{equation*}
			\frac{1}{2} \frac{d}{dt}\|\theta^n\|^2_{\dot{H}^s} = I_3 + I_4, 
		\end{equation*}
		where 
		\begin{align*}
			I_3 &:= - \lambda_2 \int_{\mathbb{R}^2} \Lambda^s v^n_2 \Lambda^s \theta^n\,dx 
			\leq |\lambda_2| \|v^n\|_{H^2} \|\theta^n\|_{\dot{H}^s};
			\\
			I_4 &:= -\int_{\mathbb{R}^2} [\Lambda^{s} 	(v^n \cdot \nabla \theta^n) - v^n \cdot \nabla \Lambda^{s}\theta^n] \Lambda^{s} \theta^n \,dx 
			\\
			&\leq C(s) \left(\|\Lambda^{s+1}v^n\|_{L^\frac{2}{s}}\|\theta^n\|_{L^\frac{2}{1-s}} + \|\nabla v^n\|_{L^\infty} \|\Lambda^s\theta^n\|_{L^2}\right)\|\Lambda^s\theta^n\|_{L^2}
			\\
			&\leq C(s) \|v^n\|_{H^2}\|\theta^n\|^2_{\dot{H}^s} \left(1 + \log \frac{\|v^n\|_{H^{\delta+1}}}{\|v^n\|_{H^2}}\right)^\frac{1}{2},
		\end{align*}
		here we used the following commutator estimate (see \cite{Kukavica-Wang-Ziane_2016}) for $\sigma_0 \in (0,1)$, $p_1,q_1,p_2 \in [2,\infty]$, $q_2 \in [2,\infty)$ and $C = C(\sigma_0,p_i,q_i)$, $i \in \{1,2\}$ 
		\begin{equation*}
			\|\Lambda^{\sigma_0}(f \cdot \nabla g) - f \cdot \nabla \Lambda^{\sigma_0}g\|_{L^2} \leq C \left(\|\Lambda^{\sigma_0+1}f\|_{L^{p_1}}\|g\|_{L^{q_1}} + \|\nabla f\|_{L^{p_2}}\|\Lambda^{\sigma_0} g\|_{L^{q_2}}\right) \qquad \text{for}\quad \frac{1}{p_i} + \frac{1}{q_i} = \frac{1}{2},
		\end{equation*}
		and the Brezis-Gallouet type inequality (see \cite{Brezis-Gallouet_1980} for $s_0 = 1$ and \cite{Hu-Kukavica-Ziane_2015} for $s_0 \in (0,1)$)
		\begin{equation*}
			\|\nabla f\|_{L^\infty} \leq C\|\mathcal{F}(\nabla f)\|_{L^1} \leq C(s_0)\|f\|_{H^2}  \left(1 + \log \frac{\|f\|_{H^{2+s_0}}}{\|f\|_{H^2}}\right)^\frac{1}{2} \qquad \forall s_0 \in (0,1],
		\end{equation*}
		by choosing $\sigma_0 = s \in (0,1)$ and $s_0 = \delta - 1 \in (0,s] \subset (0,1]$. Therefore, 
		\begin{equation*}
			\frac{d}{dt}\left(\|\theta^n\|^2_{\dot{H}^{s}} + 1\right) \leq C(s)  \left(\|\theta^n\|^2_{\dot{H}^{s}} + 1\right) \|v^n\|_{H^2}  \left(1 + \log \frac{\|v^n\|^2_{H^{\delta+1}}}{\|v^n\|^2_{H^2}}\right)^\frac{1}{2},
		\end{equation*}
		and from the fact that
		\begin{equation*}
			\int^t_0 \|v^n\|^2_{H^2} \,d\tau \leq C(T^n_*,\nu_2,\mu_1,\lambda_i,v_0,\theta_0),
		\end{equation*}
		by using a Gronwall-type inequality (see \cite[Lemma 2.3]{Hu-Kukavica-Ziane_2015}) for $t \in (0,T^n_*)$
		\begin{equation} \label{Hs_estimate}
			\|v^n(t)\|^2_{H^\delta} + \|\theta^n(t)\|^2_{H^{s}}+ \int^t_0 \|v^n\|^2_{H^{\delta+1}} \,d\tau \leq C(T^n_*,\delta,s,\nu_2,\mu_1,\lambda_i,v_0,\theta_0).
		\end{equation}
		
		\textbf{2c) The case $s = 1$ and $\delta \in (1,2]$.} 
		If $\delta \in (1,2)$ then $\delta - 1 \in (0,1)$ and $v^n$ can be bounded in the same way as in the previous one. We now focus on the case $\delta = 2$. It follows that $\sigma = 1$ and
		\begin{align*}
			I_2 
			&\leq C(\delta) \left(\|\Lambda^\sigma v^n\|_{L^4}\|\nabla v^n\|_{L^4} + \|v^n\|_{L^6}\|\Lambda^\sigma \nabla v^n\|_{L^3}\right)\|\Lambda^{\delta+1}v^n\|_{L^2}
			\\
			&\leq C \left(\|\nabla v^n\|^\frac{3}{2}_{L^2}\|\Lambda^3 v^n\|^\frac{1}{2}_{L^2} + \|v^n\|^\frac{1}{3}_{L^2} \|\nabla v^n\|_{L^2}\|\Lambda^3 v^n\|^\frac{2}{3}_{L^2}\right)\|\Lambda^3v^n\|_{L^2},
		\end{align*}
		which allows us to work as in Step 2b. In addition, 
		\begin{align*}
			\frac{1}{2}\frac{d}{dt} \|\theta^n\|^2_{\dot{H}^1} &= - \int_{\mathbb{R}^2} \nabla (T_n(v^n \cdot \nabla \theta^n)) \cdot \nabla \theta^n \,dx - \lambda_2\int_{\mathbb{R}^2} \nabla v^n_2 \cdot \nabla \theta^n \,dx
			\\
			&\leq \|\nabla v^n\|_{L^\infty} \|\theta^n\|^2_{\dot{H}^1} + |\lambda_2||v^n\|_{H^2}\|\theta^n\|_{\dot{H}^1}.
		\end{align*}
		Thus, the proof follows as in the previous case and \eqref{Hs_estimate} follows for $s = 1$ and $\delta \in (1,2]$.
		
		\textbf{2d) and 2e) The case $s \in (1,2]$ and $\delta \in [s,s+1]$.} The cases $s \in (1,2)$ and $s = 2$ share similar ideas as in Steps 2b and 2c, respectively. Since the proofs are long and to avoid repeating, we skip the details here, however it can be found in  Appendix A (see Section \ref{sec:app}).
		
		\textbf{2f) The case $s > 2$ and $\delta \in [s,s+1]$.} We separate the cases $\delta = s$ and $\delta \in (s,s+1]$ as follows.
		
		$\bullet$ The case $\delta = s$. Applying $(\Lambda^{s},J^{s})$ to \eqref{B_app} and multiplying $(\Lambda^{s} v^n,J^{s}\theta^n)$, respectively, we obtain
		\begin{equation*}
			\frac{1}{2}\frac{d}{dt}\left(\|v^n\|^2_{\dot{H}^{s}} + \|\theta^n\|^2_{H^{s}}\right) + \min\{\nu_2,\mu_1\} \|v^n\|^2_{\dot{H}^{s+1}} 
			\leq \sum^{25}_{j=22} I_j,
		\end{equation*}
		where for some $\epsilon \in (0,1)$ 
		\begin{align*}
			I_{22} &:= \lambda_1\int_{\mathbb{R}^2} \Lambda^{s}(\mathbb{P}(\theta^n e_2)) \cdot \Lambda^{s} v^n\,dx \leq |\lambda_1| \left(\|\theta^n\|^2_{\dot{H}^{s}} +  \|v^n\|^2_{\dot{H}^{s}}\right);
			\\
			I_{23} &:= -\int_{\mathbb{R}^2} \Lambda^{s}(\mathbb{P}(T_n(v^n \cdot \nabla v^n)) \cdot \Lambda^{s} v^n \,dx
			\\
			&\leq C(s,\nu_2,\mu_1,\epsilon) \left(\|v^n\|^2_{L^\infty} + \|\nabla v^n\|_{L^\infty}\right)\|v^n\|^2_{\dot{H}^{s}} + \epsilon\min\{\nu_2,\mu_1\}\|v^n\|^2_{\dot{H}^{s+1}};
			\\
			I_{24} &:= -\lambda_2\int_{\mathbb{R}^2} J^{s}v^n_2 J^{s} \theta^n \,dx \leq |\lambda_2|\left(\|v^n\|^2_{H^{s}} + \|\theta^n\|^2_{H^{s}}\right);
			\\
			I_{25} &:= -\int_{\mathbb{R}^2} J^{s}(T_n(v^n \cdot \nabla \theta^n)) J^{s} \theta^n \,dx
			\\
			&\leq C(s)\left(\|v^n\|^2_{H^{s}} + \|\theta^n\|^2_{H^{s}}\right) \left(\|\nabla v^n\|_{L^\infty} + \|\nabla \theta^n\|_{L^\infty}\right);
		\end{align*}
		here we used the following commutator estimate (see \cite{Kato-Ponce_1988}) 
		\begin{equation*}
			\|J^r(fg) - f J^r g\|_{L^2} \leq C(r)\left(\|J^r f\|_{L^2} \|g\|_{L^\infty} +  \|\nabla f\|_{L^\infty}\|J^{r-1} g\|_{L^2}\right) \qquad \forall r > 0.
		\end{equation*} 
		We apply $\Lambda^{s}$ instead of $J^{s}$ directly to the equation of $v^n$ since the proof can be used to the case $\lambda_1$ is a function as well, see the point number 1 in Remark \ref{rm1}. By choosing $\epsilon = \frac{1}{2}$, it follows that
		\begin{equation*}
			\frac{d}{dt}Y^n_s(t) + \min\{\nu_2,\mu_1\} \|v^n\|^2_{H^s} \leq C(s,\nu_2,\mu_1,\lambda_i)\left(1 + \|v^n\|^2_{L^\infty} + \|\nabla v^n\|_{L^\infty} + \|\nabla \theta^n\|_{L^\infty}\right) Y^n_s(t),
		\end{equation*}
		where 
		\begin{equation*}
			Y^n_s(t) := \|v^n(t)\|^2_{H^{s}} + \|\theta^n(t)\|^2_{H^{s}} \qquad \text{for} \quad t \in (0,T^n_*). 
		\end{equation*}
		Thus, to bound $Y^n_s(t)$, one also needs to estimate the time integral of $\|\nabla \theta^n\|_{L^\infty}$ in a suitable way. In order to do that, we consider the following $L^p$ estimate from \eqref{B_app} 
		\begin{equation*}
			\frac{1}{p}\frac{d}{dt} \|\nabla \theta^n\|^p_{L^p} \leq \|\nabla v^n\|_{L^\infty} \|\nabla \theta^n\|^p_{L^p} + \|\nabla v^n_2\|_{L^p}\|\nabla \theta^n\|^{p-1}_{L^p} \qquad \text{for}\quad p \in [1,\infty),
		\end{equation*}
		which together with letting $p \to \infty$ and using Step 2e implies that for $t \in (0,T^n_*)$
		\begin{align*}
			\|\nabla \theta^n(t)\|_{L^\infty} &\leq \left(\|\nabla \theta^n(0)\|_{L^\infty} +  \int^t_0 \|\nabla v^n\|_{L^\infty} \,d\tau\right) \exp\left(\int^t_0 \|\nabla v^n\|_{L^\infty} \,d\tau\right)
			\\
			&\leq \left(\|\theta^n(0)\|_{H^{s}} +  \int^t_0 \|v^n\|_{H^3} \,d\tau\right) \exp\left(\int^t_0 \|v^n\|_{H^3} \,d\tau\right)
			\\
			&\leq C(T^n_*,s,\nu_2,\mu_1,\lambda_i,v_0,\theta_0).
		\end{align*}
		Therefore, \eqref{Hs_estimate} follows for $s > 2$ and $\delta = s$.
		
		$\bullet$ The case $\delta \in (s,s+1]$. The proof of this case uses a similar idea as in the case $\delta = s$. The full proof can be found in Appendix A (see Section \ref{sec:app}). 
		
		\textbf{2g) The case $s > 1$ and $\delta \in [s-1,s)$, $\delta > 1$.} The proof of this case shares similar ideas as in Steps 2a-2f and is given in Appendix A (see Section \ref{sec:app}). 
		
		\textbf{2h) Uniform bound.} Collecting Step 1 and Steps 2a-2g, it implies that there is a contradiction with the assumption $T^n_* < \infty$. Therefore, $T^n_* = \infty$ and by repeating all computations from Step 2a to Step 2g, for any given $T \in (0,\infty)$ (does not depend on $n$) and for $t \in (0,T)$
		\begin{align*}
			&\|v^n(t)\|^2_{H^\delta} +  \|\theta^n(t)\|^2_{H^s} 
			+ \int^t_0 \|v^n\|^2_{H^{\delta+1}} \,d\tau \leq C(T,\delta,s,\nu_2,\mu_1,\lambda_i,v_0,\theta_0).
		\end{align*}

		\textbf{Step 3: Cauchy sequence for $s > 1$.} In this step, we will show that $(v^n,\theta^n)$ is a Cauchy sequence in $L^\infty(0,T;L^2(\mathbb{R}^2))$ for any $T \in (0,\infty)$. Since $s > 1$ then from the previous step, we also have $\delta > 1$. Let us fix $n,m \in \mathbb{R}$ with $m > n > 0$. Assume that $(v^n,\theta^n)$ and $(v^m,\theta^m)$ are two solutions to \eqref{B_app} with the initial data given by $(v^n,\theta^n)_{|_{t=0}} = T_n(v_0,\theta_0)$ and $(v^m,\theta^m)_{|_{t=0}} = T_m(v_0,\theta_0)$,\footnote{This is a new point compared to the previous arXiv (the second one) and the published versions.} respectively. It follows that 
		\begin{equation*}
			\frac{1}{2}\frac{d}{dt} \left(\|v^n-v^m\|^2_{L^2} + \|\theta^n-\theta^m\|^2_{L^2}\right) + \min\{\nu_2,\mu_1\} \|\nabla (v^n-v^m)\|^2_{L^2}= \sum^{44}_{i=42} I_i,
		\end{equation*}
		where 
		\begin{align*}
			I_{42} &:= \int_{\mathbb{R}^2} \mathbb{P}(-T_n(v^n \cdot \nabla v^n) + T_m(v^m \cdot \nabla v^m)) \cdot (v^n-v^m)\,dx,
			\\
			I_{43} &:= \int_{\mathbb{R}^2} (-T_n(v^n \cdot \nabla \theta^n) + T_m(v^m \cdot \nabla \theta^m))(\theta^n-\theta^m)\,dx,
			\\
			I_{44} &:= \int_{\mathbb{R}^2} \lambda_1\mathbb{P}((\theta^n-\theta^m)e_2) \cdot (v^n-v^m) + \lambda_2 (v^m_2-v^n_2)(\theta^n-\theta^m)\,dx
			\\
			&\leq (|\lambda_1|+|\lambda_2|)\left(\|v^n-v^m\|^2_{L^2} + \|\theta^n-\theta^m\|^2_{L^2}\right).
		\end{align*}
		
		Before going to bound the above integrals, we recall that (see \cite{Fefferman-McCormick-Robinson-Rodrigo_2014})
		\begin{align*}
			\|T_n(f) - f\|_{H^{s_1}} &\leq n^{-s_2} \|f\|_{H^{s_1+s_2}} \qquad \text{for} \quad s_1,s_2 \in \mathbb{R}, s_2 \geq 0.
		\end{align*}
		We write $I_{42} = I_{421} + I_{422} + I_{423}$, where $I_{423} = 0$ and 
		\begin{align*}
			I_{421} &:= \int_{\mathbb{R}^2} (T_m-T_n)(v^n \cdot \nabla v^n) \cdot (v^n-v^m)\,dx
			\\
			&\leq n^{-(\delta-1)}\|v^n \cdot \nabla v^n\|_{H^{\delta-1}}\|v^n-v^m\|_{L^2}
			\leq n^{-(\delta-1)}C(T);
			\\
			I_{422} &:= \int_{\mathbb{R}^2} T_m((v^m-v^n) \cdot \nabla v^n) \cdot (v^n-v^m)\,dx
			\\
			&\leq C(T,\epsilon)\|v^n-v^m\|^2_{L^2} + \epsilon \min\{\nu_2,\mu_1\}\|\nabla(v^n-v^m)\|^2_{L^2},
		\end{align*}
		by using the triangle and 2D Ladyzhensaya inequalities, $H^{s}(\mathbb{R}^2)$ is an algebra for $s > 1$ and $C(T) = C(T,\delta,s,\nu_2,\mu_1,\lambda_i,v_0,\theta_0)$ in Step 2h. Similar to $I_{42}$, we write $I_{43} := I_{431} + I_{432} + I_{433}$, where $I_{433} = 0$ and 
		for some fixed $\alpha \in (0,\min\{1,s-1\})$
		\begin{align*}
			I_{431} &:= -\int_{\mathbb{R}^2} (T_n-T_m)(v^n \cdot \nabla \theta^n)(\theta^n-\theta^m)\,dx 
			\leq n^{-(\min\{\delta,s\}-1)}C(T);
			\\
			I_{432} &:= -\int_{\mathbb{R}^2} T_m((v^n-v^m) \cdot \nabla \theta^n)(\theta^n-\theta^m)\,dx
			\\
			&\leq \|v^n-v^m\|_{L^\frac{2}{\alpha}} \|\nabla \theta^n\|_{L^\frac{2}{1-\alpha}} \|\theta^n-\theta^m\|_{L^2} 
			\\
			&\leq C(s)\|v^n-v^m\|_{\dot{H}^{1-\alpha}} \|\nabla \theta^n\|_{\dot{H}^\alpha} \|\theta^n-\theta^m\|_{L^2} 
			\\
			&\leq \epsilon\min\{\nu_2,\mu_1\}\|v^n-v^m\|^2_{H^1} + C(T,\epsilon) \|\theta^n-\theta^m\|^2_{L^2}.
		\end{align*}
		Therefore, by adding the term $\min\{\nu_2,\mu_1\}\|(v^n-v^m)(t)\|^2_{L^2}$ to both sides of the main estimate and choosing $\epsilon = \frac{1}{4}$, we obtain 
		\begin{equation*}
			\frac{d}{dt}E^{nm}(t) + \min\{\nu_2,\mu_1\}\|\nabla(v^n-v^m)\|^2_{L^2} \leq \max\{n^{-(\delta-1)},n^{-(\min\{\delta,s\}-1)}\}C(T) + C(T) E^{nm}(t),
		\end{equation*}
		which implies by using Gronwall inequality that for $t \in (0,T)$
		\begin{equation*}
			E^{nm}(t) + \int^t_0 \|\nabla(v^n-v^m)\|^2_{L^2} \,d\tau \leq C(T)E^{nm}(0) +  \max\{n^{-(\delta-1)},n^{-(\min\{\delta,s\}-1)}\} C(T),
		\end{equation*}
		where
		\begin{equation*}
			E^{nm}(t) := \|(v^n-v^m)(t)\|^2_{L^2} + \|(\theta^n-\theta^m)(t)\|^2_{L^2}.
		\end{equation*}
		That ends the proof of this step by letting $n \to \infty$ with using the mentioned property of $T_n$. 

		\textbf{Step 4: Passing to the limit.} We separately consider the cases $s > 1$ and $s \in [0,1]$ as follows. We denote $\to$, $\rightharpoonup$ and $\overset{\ast}{\rightharpoonup}$ for the usual strong, weak and weak-star convergences, respectively. From the previous step that there exists $(v,\theta)$ such that as $n \to \infty$
		\begin{align*}
			&&(v^n,\theta^n) &\to (v,\theta) &&\text{in} \quad L^\infty(0,T;L^2(\mathbb{R}^2)),&&
			\\
			&&\nabla v^n &\to \nabla v  &&\text{in} \quad L^2(0,T;L^2(\mathbb{R}^2)).&&
		\end{align*}
		
		\textbf{4a) The case $s > 1$.} We also have in this step $\delta > 1$. The above convergences imply by using the Sobolev interpolation inequalities and uniform bound that for all $s' \in (1,\min\{\delta,s\})$ as $n \to \infty$
		\begin{align*}
			&&(v^n,\theta^n) &\to (v,\theta) &&\text{in} \quad L^\infty(0,T;H^{s'}(\mathbb{R}^2)),&&
			\\
			&&\nabla v^n &\to \nabla v  &&\text{in} \quad L^2(0,T;H^{s'}(\mathbb{R}^2)),&&
			\\
			&&(\partial_{22}v^n_1,\partial_{11}v^n_2) &\to (\partial_{22}v_1,\partial_{11}v_2 )  &&\text{in} \quad L^2(0,T;H^{s'-1}(\mathbb{R}^2)).&&
		\end{align*}
		Moreover, as $n \to \infty$
		\begin{equation*}
			T_n(v^n \cdot \nabla v^n,v^n \cdot \nabla \theta^n) \to (v \cdot \nabla v,v \cdot \nabla \theta)  \qquad \text{in} \quad L^\infty(0,T;H^{s'-1}(\mathbb{R}^2)),
		\end{equation*}
		since by using the property of $T_n$, uniform bound and triangle inequality 
		\begin{align*}
			\|T_n(v^n \cdot \nabla v^n) - v \cdot \nabla v\|_{H^{s'-1}} &\leq n^{-(\delta-s')} \|v^n\|^2_{H^\delta}  + \|v^n-v\|_{H^{s'}} \left(\|v^n\|_{H^{s'}} + \|v\|_{H^{s'}}\right),
			\\
			\|T_n(v^n \cdot \nabla \theta^n) - v \cdot \nabla \theta\|_{H^{s'-1}} &\leq n^{-(\min\{\delta,s\}-s')} \|v^n\|_{H^\delta} \|\theta^n\|_{H^{s}} + \|\theta^n-\theta\|_{H^{s'}} \|v^n\|_{H^{s'}}  
			\\
			&\quad+\|v^n-v\|_{H^{s'}} \|\theta\|_{H^{s'}}.
		\end{align*}
		In addition, \eqref{B_app} gives us for $t \in (0,T)$, $\sigma \in \{s'-1,\delta-1\}$ and $\sigma' \in \{s'-1,s-1\}$
		\begin{align*}
			\int^t_0\|\partial_t v^n\|^2_{H^\sigma} \,d\tau &\leq 
			C\int^t_0 \|(v^n \cdot \nabla v^n,  \nu_2\partial_{22}v^n_1,\mu_1\partial_{11}v^n_2,\lambda_1\theta^n)\|^2_{H^\sigma} \,d\tau,
			\\
			\int^t_0\|\partial_t \theta^n\|^2_{H^{\sigma'}} \,d\tau &\leq 
			C\int^t_0 \|(v^n \cdot \nabla \theta^n, \lambda_2 v^n_2)\|^2_{H^{\sigma'}} \,d\tau,
		\end{align*}
		which together with the uniform bound and above strong convergences leads to there exists a subsequence $(v^{n_k},\theta^{n_k})$ such that as $n_k \to \infty$
		\begin{align*}
			&&(\partial_t v^{n_k},\partial_t \theta^{n_k}) &\rightharpoonup (\partial_t v,\partial_t\theta) &&\text{in} \quad L^2(0,T;H^{\delta-1}(\mathbb{R}^2) \times H^{s-1}(\mathbb{R}^2)),&&
			\\
			&&(\partial_t v^{n_k},\partial_t \theta^{n_k}) &\to (\partial_t v,\partial_t \theta) &&\text{in} \quad L^2(0,T;H^{s'-1}(\mathbb{R}^2)),&&
		\end{align*}
		which together with the above strong convergences and \eqref{B_app} implies that in $L^2(0,T;H^{s'-1}(\mathbb{R}^2))$
		\begin{align*}
			\partial_t v + \mathbb{P}(v \cdot \nabla v) &=  \mathbb{P}(\nu_2\partial_{22}v_1,\mu_1\partial_{11}v_2) + \lambda_1\mathbb{P}(\theta e_2), 
			\\
			\partial_t \theta + v \cdot \nabla \theta  &=  -\lambda_2 v_2.
		\end{align*}
		In addition, it can be checked that as $n \to \infty$
		\begin{align*}
			\textnormal{div}\, v^n &\to \textnormal{div}\, v  &&\text{in}\quad L^2(0,T;H^{s'}(\mathbb{R}^2)),
			\\
			(v^n,\theta^n)(0) = T_n(v_0,\theta_0) &\to (v_0,\theta_0) & &\text{in}\quad H^\delta(\mathbb{R}^2) \times H^{s}(\mathbb{R}^2),
		\end{align*}
		which leads to $\textnormal{div}\, v = 0$ and $(v,\theta)(0) = (v_0,\theta_0)$. Then the theorem de Rham (see \cite{Temam_2001}) ensures the existence of a scalar function $\pi$ such that (at least in the sense of distributions)
		\begin{equation*}
			\partial_t v + v \cdot \nabla v + \nabla \pi =  (\nu_2\partial_{22}v_1,\mu_1\partial_{11}v_2) + \lambda_1\theta e_2.
		\end{equation*}
		From the uniform bound in Step 2, we also have as $n_k \to \infty$ 
		\begin{align*}
			&&(v^{n_k},\theta^{n_k}) &\overset{\ast}{\rightharpoonup} (v,\theta)  &&\text{in} \quad L^\infty(0,T;H^\delta(\mathbb{R}^2) \times H^{s}(\mathbb{R}^2)), &&
			\\
			&&\nabla v^{n_k} &\rightharpoonup  \nabla v  &&\text{in} \quad L^2(0,T;H^\delta(\mathbb{R}^2)), &&
		\end{align*}
		which implies that for $s > 1$, $(v,\theta) \in L^\infty(0,T;H^\delta(\mathbb{R}^2)\times H^{s}(\mathbb{R}^2))$ and $v \in  L^2(0,T;H^{\delta+1}(\mathbb{R}^2))$ with  
		\begin{equation*}
			\|v(t)\|^2_{H^\delta} + \|\theta(t)\|^2_{H^{s}} + \int^t_0 \|v\|^2_{H^{\delta+1}} \,d\tau \leq C(T,\delta,s,\nu_2,\mu_1,\lambda_i,v_0,\theta_0) \qquad \text{for}\quad t \in (0,T).
		\end{equation*}
		In fact, after possibly being redefined on a set of measure zero  $v \in C([0,T];H^\delta(\mathbb{R}^2))$  (see \cite{Evans_2010}) since $v \in  L^2(0,T;H^{\delta+1}(\mathbb{R}^2))$ and $\partial_t v \in L^2(0,T;H^{\delta-1}(\mathbb{R}^2))$. Furthermore, we first see that $\theta$ is weak continuous in time with values in $H^{s}(\mathbb{R}^2)$ from the uniform bound in Step 2. In order to  prove that $\theta \in C([0,T];H^{s}(\mathbb{R}^2))$, it sufficies to show that $\|\theta(t)\|_{H^{s}}$ is continuous in time.
		
		$\bullet$ $\|\theta(t)\|_{H^{s}}$ is right-continuous in time. If $s \in (1,2]$ then similar to the proofs of Steps 2d and 2e (with choosing $\epsilon' = \delta -1$) given in Appendix A, since $\delta > 1$ it follows that for $0 \leq t_1 \leq t_2 \leq T$
		\begin{equation*}
			\|\theta(t_2)\|_{H^{s}} \leq \left(\|\theta(t_1)\|_{H^{s}} + C(\lambda_2) \int^{t_2}_{t_1} \|v\|_{H^2}\,d\tau\right) \exp\left(C(\delta)\int^{t_2}_{t_1} \|v\|_{H^{\delta+1}}\,d\tau\right).
		\end{equation*}
		If $s > 2$ then similar to Step 2f, since $s \leq \delta + 1$ we find that for $0 \leq t_1 \leq t_2 \leq T$
		\begin{equation*}
			\|\theta(t_2)\|_{H^{s}} \leq \left(\|\theta(t_1)\|_{H^{s}} + C(s,\lambda_2) \int^{t_2}_{t_1} \|v\|_{H^{\delta+1}}(\|\nabla \theta\|_{L^\infty} + 1) \,d\tau\right) \exp\left(C(\delta,s)\int^{t_2}_{t_1} \|v\|_{H^{\delta+1}}\,d\tau\right).
		\end{equation*}
		The argument in Step 2f allows us to bound $\|\nabla \theta\|_{L^\infty}$ by a constant $C(T,\delta,s,\nu_2,\mu_1,v_0,\theta_0)$. Therefore, in both cases $\|\theta(t)\|_{H^{s}}$ is right-continuous in time.
		
		$\bullet$ $\|\theta(t)\|_{H^{s}}$ is left-continuous in time. By noting that the equation of $\theta$ is time-reversible\footnote{As usual, it is understood in the following sense. Assume that $(v(x,t),\theta(x,t))$ satisfies the $\theta$ equation in \eqref{B} with $\delta_1 = \delta_2 = 0$, i.e.,
		\begin{equation*}
			\partial_t\theta(x,t) + v(x,t)\cdot \nabla \theta(x,t) + \lambda_2v_2(x,t) = 0 \qquad \text{for} \quad (x,t) \in \mathbb{R}^2 \times (0,T).
		\end{equation*}
		If $\theta(x,t)$ is replaced by $\theta(x,-t)$ then 
		\begin{equation*}
			\partial_s\theta(x,s) + v(x,s)\cdot \nabla \theta(x,s) + \lambda_2v_2(x,s) = 0 \qquad \text{for} \quad (x,s) \in \mathbb{R}^2 \times (-T,0).
		\end{equation*}
		Here, we also replaced $v(x,t)$ by  $-v(x,-t)$.
		See also \cite{Fefferman-McCormick-Robinson-Rodrigo_2014}, as in the case of the magnetic field equation in the non-resistive MHD system.}. Then $\|\theta(t)\|_{H^{s}}$ is left-continuous in time as well. 
		
		\textbf{4b) The case $s \in [0,1]$.} It sufficies to consider the case $s = 0$ and other cases follow as a consequence. It follows from the uniform bound in Step 2 that for $t \in (0,T)$
		\begin{align*}
			\|v^n(t)\|^2_{L^2} + \|\theta^n(t)\|^2_{L^2} + \int^t_0 \|v^n\|^2_{H^1} \,d\tau &\leq C(T,\nu_2,\mu_1,\lambda_i,v_0,\theta_0).
		\end{align*}
		In addition, for $(\phi,\varphi) \in H^1(\mathbb{R}^2) \times H^2(\mathbb{R}^2)$ with $\phi = (\phi_1,\phi_2)$, $\|\phi\|_{H^1} \leq 1$ and $\|\varphi\|_{H^2} \leq 1$, it yields for $\tau \in (0,T)$\footnote{Here $(\cdot,\cdot)$ is the standard $L^2$ inner product.}
		\begin{align*}
			\int^{\tau}_0 |(\partial_tv^n,\phi)|^2 +  |(\partial_t\theta^n,\varphi)|^2 \,dt &\leq C\int^{\tau}_0 \|v^n\|^2_{H^1} + \|\theta^n\|^2_{L^2}\,dt \leq C(T,\nu_2,\mu_1,\lambda_i,v_0,\theta_0),
		\end{align*}
		which implies that\footnote{As usual, for $s \in \mathbb{R}$ with $s > 0$, the space $H^{-s}(\mathbb{R}^2)$ can be considered as the dual space of $H^s(\mathbb{R}^2)$, see \cite{Bahouri-Chemin-Danchin_2011}.}
		\begin{equation*}
			(\partial_tv^n,\partial_t\theta^n) \quad \text{is uniformly bounded in} \quad L^2(0,T;H^{-1}(\mathbb{R}^2) \times H^{-2}(\mathbb{R}^2)).
		\end{equation*}
		Therefore, there exists a subsequence (still denoted by) $(v^n,\theta^n)$ and $(v,\theta)$ such that as $n \to \infty$
		\begin{align*}
			&&(v^n,\theta^n) &\overset{\ast}{\rightharpoonup} (v,\theta)  &&\text{in} \quad L^\infty(0,T;L^2(\mathbb{R}^2)), &&
			\\
			&&v^n &\rightharpoonup  v  &&\text{in} \quad L^2(0,T;H^1(\mathbb{R}^2)), &&
			\\
			&&(\partial_tv^n,\partial_t\theta^n) &\rightharpoonup  (\partial_t v,\partial_t \theta) &&\text{in} \quad L^2(0,T;H^{-1}(\mathbb{R}^2) \times H^{-2}(\mathbb{R}^2)). &&
		\end{align*}
		Recall that the injections $H^1 \hookrightarrow L^2 \hookrightarrow H^{-1}$ and $L^2 \hookrightarrow H^{-1} \hookrightarrow H^{-2}$ are locally compact by using the Rellich–Kondrachov and Schauder theorems (see \cite{Brezis_2011}) then an application of the Aubin-Lions lemma (see \cite{Boyer-Fabrie_2013}) implies that as $n \to \infty$
		\begin{align*}
			(v^n,\theta^n) \to (v,\theta) \qquad \text{(locally in space) in} \quad L^2(0,T;L^2(\mathbb{R}^2) \times H^{-1}(\mathbb{R}^2)).
		\end{align*}
		Furthermore, from \eqref{B_app} $(v^n,\theta^n)$ satisfies
		\begin{align*}
			a) \quad &\int^T_0 \int_{\mathbb{R}^2} v^n \cdot \partial_t\phi \,dxdt - \int^T_0 \int_{\mathbb{R}^2} \mathbb{P}(T_n(v^n \cdot \nabla v^n)) \cdot \phi \,dxdt + \int^T_0 \int_{\mathbb{R}^2}\mathbb{P}(\nu_2\partial_{22}v^n_1,\nu_1\partial_{11}v^n_2) \cdot \phi \,dxdt 
			\\
			\quad &+ \lambda_1\int^T_0 \int_{\mathbb{R}^2} \mathbb{P}(\theta^ne_2) \cdot \phi \,dxdt = - \int_{\mathbb{R}^2} v^n(0) \cdot \phi(0) \,dx,
			\\
			b) \quad &\int^T_0 \int_{\mathbb{R}^2} \theta^n \partial_t\varphi \,dxdt - \int^T_0 \int_{\mathbb{R}^2} T_n(v^n \cdot \nabla \theta^n) \varphi \,dxdt - \lambda_2\int^T_0 \int_{\mathbb{R}^2} v^n_2 \varphi \,dxdt  = - \int_{\mathbb{R}^2} \theta^n(0)\varphi(0) \,dx,
		\end{align*}
		where $\phi = (\phi_1,\phi_2) \in C^\infty_c([0,T) \times \mathbb{R}^2;\mathbb{R}^2)$ with $\textnormal{div}\,\phi = 0$ and $\varphi \in C^\infty_c([0,T) \times \mathbb{R}^2;\mathbb{R})$. By using the above weak and strong convergences, as $n \to \infty$, we can pass to the limit for the linear terms easily. It remains to check the convergence of the nonlinear terms. Moreover, we find that\footnote{Here $v \otimes u := (v_iu_j)_{1\leq i,j\leq 2}$ for $v = (v_1,v_2)$ and $u = (u_1,u_2)$.},\footnote{There are more details here compared to the published version.}
		\begin{align*}
			\textnormal{NL}_1 &= \left|\int^T_0 \int_{\mathbb{R}^2} (T_n(v^n \cdot \nabla v^n)- v \cdot \nabla v) \cdot \phi \,dxdt\right| 
			\\
			&\leq \left|\int^T_0 \int_{\mathbb{R}^2} (T_n(v^n \otimes v^n) - v^n \otimes v^n) : \nabla \phi \,dxdt\right|  
			+ \left|\int^T_0 \int_{\mathbb{R}^2} ((v^n-v) \otimes v^n) : \nabla \phi \,dxdt\right| 
			\\
			&\quad + \left|\int^T_0 \int_{\mathbb{R}^2} (v \otimes (v^n-v)) : \nabla \phi \,dxdt\right|
			\\
			&\leq \|T_n(v^n \otimes v^n) - v^n \otimes v^n\|_{L^2(0,T;H^{-1}(\mathbb{R}^2))}  \|\nabla \phi\|_{L^2(0,T;H^1(\mathbb{R}^2))}
			\\
			&\quad +\|v^n-v\|_{L^2_{t,x}(\textnormal{supp}(\phi))} \|v^n\|_{L^\infty(0,T;L^2(\mathbb{R}^2))}  \|\nabla \phi\|_{L^2_tL^\infty_x((0,T) \times \mathbb{R}^2)}
			\\
			&\quad + \|v^n-v\|_{L^2_{t,x}(\textnormal{supp}(\phi))} \|v\|_{L^\infty(0,T;L^2(\mathbb{R}^2))}  \|\nabla \phi\|_{L^2_tL^\infty_x((0,T) \times \mathbb{R}^2)}
		\end{align*}
		$\to 0$ as $n \to \infty$, since the uniformly bounded in terms of $n$ of $\|v^n\|_{L^2_tL^\infty_x}$ (see the appendix)
		\begin{equation*}
			\|T_n(v^n \otimes v^n) - v^n \otimes v^n\|_{L^2(0,T;H^{-1}(\mathbb{R}^2))}  \leq \frac{1}{n} \|v^n \otimes v^n\|_{L^2},
		\end{equation*}
		and 
		\begin{align*}
			\textnormal{NL}_2 &= \left|\int^T_0 \int_{\mathbb{R}^2} (T_n(v^n \cdot \nabla \theta^n) - v \cdot \nabla \theta) \varphi \,dxdt\right| 
			\\
			&\leq \left|\int^T_0 \int_{\mathbb{R}^2} (T_n(v^n \theta^n) - v^n\theta^n) \cdot \nabla \varphi \,dxdt\right|
			+\left|\int^T_0 \int_{\mathbb{R}^2} (v^n-v) \theta^n \cdot \nabla \varphi \,dxdt\right| 
			\\
			&\quad + \left|\int^T_0 \int_{\mathbb{R}^2} v (\theta^n-\theta) \cdot \nabla \varphi \,dxdt\right|
			\\
			&\leq \|T_n(v^n\theta^n) - v^n\theta^n\|_{L^2(0,T;H^{-1}(\mathbb{R}^2))}  \|\nabla \varphi\|_{L^2(0,T;H^1(\mathbb{R}^2))}
			\\ &\quad + \|v^n-v\|_{L^2_{t,x}(\textnormal{supp}(\varphi))} \|\theta^n\|_{L^\infty(0,T;L^2(\mathbb{R}^2))}  \|\nabla \varphi\|_{L^2_tL^\infty_x((0,T) \times \mathbb{R}^2)}
			\\
			&\quad + \|\theta^n-\theta\|_{L^2_tH^{-1}_x(\textnormal{supp}(\varphi))} \|v\|_{L^2(0,T;H^1(\mathbb{R}^2))}  \|\nabla \varphi\|_{L^\infty_tL^\infty_x((0,T) \times \mathbb{R}^2)}
		\end{align*}
		$\to 0$ as $n \to \infty$ since
		\begin{equation*}
			\|T_n(v^n\theta^n) - v^n\theta^n\|_{L^2(0,T;H^{-1}(\mathbb{R}^2))}  \leq \frac{1}{n}\|v^n \theta^n\|_{L^2_{t,x}}.
		\end{equation*}
		That means $(v,\theta)$ satisfies in the sense of distributions (similar to $(a,b)$)
		\begin{equation*}
			\left\{
			\begin{aligned}
				\partial_t v &= - \mathbb{P}(v \cdot \nabla v) +  \mathbb{P}(\nu_2\partial_{22}v_1,\nu_1\partial_{11}v_2) + \lambda_1\mathbb{P}(\theta e_2), 
				\\
				\partial_t \theta &= -v \cdot \nabla \theta  -\lambda_2 v_2,
				\\
				\textnormal{div}\,v &= 0,
			\end{aligned}
			\right.
			\qquad \text{in}\quad \mathbb{R}^2 \times (0,T),
		\end{equation*}
		with $(v,\theta)(0) = (v_0,\theta_0)$. In addition, $(v,\theta) \in C([0,T];L^2(\mathbb{R}^2))$, which shares the same bound for $s \in [0,1]$ as that of $(v^n,\theta^n)$ given in Step 2. A scalar pressure $\pi$ is recovered by using the theorem de Rham (see \cite{Temam_2001}) such that $(v,\theta,\pi)$ satisfies \eqref{B} in the sense of distributions. The continuity of $v$ follows easily while the continuity of $\theta$ follows as a consequence since $\theta$ satisfies a transport-type equation with $v \in C([0,T];L^2) \cap L^2(0,T;H^1)$.
		
		$\bullet$ The continuity in time of $\|v(t)\|_{H^\delta}$ and $\|\theta(t)\|_{H^s}$ for $s \in (0,1]$ and $\delta \in (1,s+1)$. We have $\|v(t)\|_{H^\delta}$ is continuous in time as in Step 4a. If $s \in (0,1]$ then similar to the proofs of Steps 2b and 2c, since $\delta > 1$ it follows that for $0 \leq t_1 \leq t_2 \leq T$
		\begin{equation*}
			\|\theta(t_2)\|_{H^{s}} \leq \left(\|\theta(t_1)\|_{H^{s}} + C(\lambda_2) \int^{t_2}_{t_1} \|v\|_{H^2}\,d\tau\right) \exp\left(C(\delta)\int^{t_2}_{t_1} \|v\|_{H^{\delta+1}}\,d\tau\right),
		\end{equation*}
		which implies that $\|\theta(t)\|_{H^s}$ is right-continuous in time. As in Step 4a, $\|\theta(t)\|_{H^s}$ is left-continuous in time as well.
		
		\textbf{Step 5: Uniqueness.} In fact, it suffices to focus on the case $s = 0$. However, to avoid borrowing known ideas, we still provide our simple proof for the case $s > 1$.
		
		\textbf{5a) The case $s > 1$.} We assume that $(v,\theta_1)$ and $(u,\theta_2)$ are two solutions to \eqref{B} with the same initial data. Therefore, 
		\begin{equation*}
			\frac{1}{2}\frac{d}{dt} \left(\|v-u\|^2_{L^2} + \|\theta_1-\theta_2\|^2_{L^2}\right) + \min\{\nu_2,\mu_1\} \|\nabla (v-u)\|^2_{L^2}= \sum^{47}_{i=45} I_i,
		\end{equation*}
		where for $\theta := \theta_1 - \theta_2$
		\begin{align*}
			I_{45} &:= -\int_{\mathbb{R}^2} (v-u) \cdot \nabla v \cdot (v-u)\,dx, 
			\\
			I_{46} &:= -\int_{\mathbb{R}^2}(v-u) \cdot \nabla \theta_1 \theta \,dx,
			\\
			I_{47} &:= (\lambda_1-\lambda_2)\int_{\mathbb{R}^2} \theta (v_2-u_2)\,dx.
		\end{align*}
		Since $\delta > 1$ then $I_{45}$, $I_{46}$ and $I_{47}$ can be estimated in the same way as $I_{422}$, $I_{432}$ and $I_{44}$, respectively. Thus, the uniqueness follows as in Step 3.
		
		\textbf{5b) The case $s \in [0,1]$.} It is enough to focus on the case $s = 0$ and other cases follow as a consequence. We will use the ideas in \cite[$\mathbb{R}^2$]{Boardman-Ji-Qiu-Wu_2019}, \cite[$\mathbb{T}^2$]{He-Ma-Sun_2022,Larios_Lunasin_Titi_2013} and \cite[bounded domains]{He_2012,Hu-Wang-Wu-Xiao-Yuan_2018}, where all these papers considered only the case $\lambda_2 = 0$, the partial dissipation $(\partial_{22}v_1,\partial_{11}v_2)$ is replaced by $\Delta v$ in \cite{Boardman-Ji-Qiu-Wu_2019} or by $\partial_{11}v$ in \cite{Larios_Lunasin_Titi_2013} and the proof used the Poincar\'{e} inequality on $\mathbb{T}^2$ in \cite{He-Ma-Sun_2022}. It seems to us that their proof can not be applied directly in the case of $\mathbb{R}^2$ with $\lambda_2 \neq 0$. Thus, it is needed to modify their proofs and we give a proof here for the sake of completeness. Let $\eta \in C^\infty_c(\mathbb{R}^2)$\footnote{The set of smooth functions in $\mathbb{R}^2$ with compact support.} be fixed and time-independent. We define $\rho_i$ for $i \in \{ 1,2\}$ as follows\footnote{There is a minus sign here compared to the published version. Therefore, there are several places, which are needed to be modified the sign below. More precisely, there are two places: the equation satisfied by $\rho \eta$ and $I_{50}$. However, the proof can be done in the same way.}
		\begin{equation*}
			-\Delta (\rho_i\eta)(x,t) = \theta_i(x,t)  \qquad \text{for} \quad (x,t) \in \text{supp}(\eta) \times (0,T) \quad\text{and} \quad  \rho\eta_i = 0 \quad \text{otherwise},
		\end{equation*}
		where it can be seen that for each $i$, $\rho_i\eta$ is uniquely determined by $\theta_i$  and satisfies $\|\nabla(\rho_i\eta)\|_{L^2} \leq \|\theta_i\|_{L^2}$.
		Thus, for $\rho := \rho_1 - \rho_2$ (see \cite{Larios_Lunasin_Titi_2013})
		\begin{equation*}
			-\partial_t \Delta(\rho\eta) =  (v-u) \cdot \nabla \Delta(\rho_1\eta) + u \cdot \nabla \Delta(\rho\eta) - \lambda_2(v_2-u_2)
		\end{equation*}
		and
		\begin{equation*}
			\frac{1}{2}\frac{d}{dt}\|\nabla(\rho\eta)\|^2_{L^2}= \sum^{50}_{i=48} I_i, 
		\end{equation*}
		where for $p \in [2,\infty)$, $q \in (2,4]$ with $\frac{1}{p} + \frac{2}{q} = 1$ and for some $\epsilon \in (0,1)$
		\begin{align*}
			I_{48} &:=  \int_{\mathbb{R}^2} (v-u) \cdot \nabla \Delta(\rho_1\eta) \rho\eta\,dx
			\\
			&\leq \|\theta_1\|_{L^2}\|v-u\|_{L^{2p}}\|\nabla(\rho\eta)\|_{L^q} 
			\\
			&\leq C \sqrt{p} \|\theta_1\|_{L^2} \|v-u\|^\frac{1}{p}_{L^2} \|\nabla(v-u)\|^{1-\frac{1}{p}}_{L^2} \|\nabla(\rho\eta)\|^\frac{2}{q}_{L^2} \|\theta\|^{1-\frac{2}{q}}_{L^2}
			\\
			&\leq C \sqrt{p} \|\theta_1\|_{L^2} \left(\|v-u\|_{L^2} + \|\nabla(v-u)\|_{L^2}\right) \|\nabla(\rho\eta)\|^{1-\frac{1}{p}}_{L^2} \|\theta\|^\frac{1}{p}_{L^2}
			\\
			&\leq \epsilon\min\{\nu_2,\mu_1\} \left(\|v-u\|^2_{L^2} + \|\nabla(v-u)\|^2_{L^2}\right) + C(\nu_2,\mu_1,\epsilon) p \|\theta_1\|^2_{L^2} \|\nabla(\rho\eta)\|^{2-\frac{2}{p}}_{L^2} \|\theta\|^\frac{2}{p}_{L^2};
			\\
			I_{49} &:=  \int_{\mathbb{R}^2} u \cdot \nabla \Delta(\rho\eta) \rho\eta\,dx
			\leq \|\nabla u\|_{L^p} \|\nabla(\rho\eta)\|^2_{L^{q}}
			\leq C \|\nabla u\|_{L^p} \|\nabla(\rho\eta)\|^{2-\frac{2}{p}}_{L^2} \|\theta\|^\frac{2}{p}_{L^2};
			\\
			I_{50} &:= -\lambda_2\int_{\mathbb{R}^2} (v_2-u_2)\rho\eta\,dx \leq |\lambda_2| \|v-u\|_{L^2}\|\rho\eta\|_{L^2} \leq C(\lambda_2,\eta) \left(\|v-u\|^2_{L^2} + \|\nabla (\rho\eta)\|^2_{L^2}\right). 
		\end{align*}
		Here, we used the following inequality (see \cite{Kozono-Wadade_2008}) 
		\begin{equation*}
			\|f\|_{L^{p_0}} \leq C \sqrt{p_0} \|f\|^\frac{2}{p_0}_{L^2} \|f\|^{1-\frac{2}{p_0}}_{\dot{H}^1}  \qquad \text{for} \quad p_0 \in [2,\infty).
		\end{equation*}
		In addition,
		\begin{equation*}
			\frac{1}{2}\frac{d}{dt} \|v-u\|^2_{L^2}  + \min\{\nu_2,\mu_1\} \|\nabla (v-u)\|^2_{L^2} = I_{51} + I_{52},
		\end{equation*}
		where
		\begin{align*}
			I_{51} &:= -\int_{\mathbb{R}^2} (v-u) \cdot \nabla v \cdot (v-u)\,dx \leq C(\nu_2,\mu_1,\epsilon) \|v-u\|^2_{L^2}\|\nabla v\|^2_{L^2} + \epsilon\min\{\nu_2,\mu_1\} \|\nabla(v-u)\|^2_{L^2};
			\\
			I_{52} &:= \lambda_1\int_{\mathbb{R}^2} \Delta(\rho\eta) (v_2-u_2) \,dx \leq C(\lambda_1,\nu_2,\mu_1,\epsilon) \|\nabla(\rho\eta)\|^2_{L^2} + \epsilon\min\{\nu_2,\mu_1\} \|\nabla(v-u)\|^2_{L^2}.
		\end{align*}
		Therefore, by choosing $\epsilon = \frac{1}{6}$ 
		\begin{align*}
			\frac{d}{dt} Y_\delta(t) \leq C_1 \left(1 + \|\nabla v\|^2_{L^2}\right) Y_\delta(t) + C_2p M^\frac{1}{p} Y^{1-\frac{1}{p}}_\delta(t) \sup_{p \in [2,\infty)} \left(1 + \frac{\|\nabla u\|_{L^p}}{p}\right),
		\end{align*}
		where $C_1 = C(\nu_2,\mu_1,\lambda_i,\eta)$, $C_2 = C(T,\nu_2,\mu_1,v_0,\lambda_i,\theta)$ and for $\delta \in (0,1)$ 
		\begin{equation*}
			Y_\delta(t) := \|(v-u)(t)\|^2_{L^2} + \|\nabla(\rho\eta)(t)\|^2_{L^2} + \delta \qquad \forall t \in (0,T),
		\end{equation*}
		and the constant $M = C(T,\lambda_1,\lambda_2,\nu_2,\mu_1,\eta,v_0,\theta_0)$ can be chosen such that 
		\begin{equation*}
			e^2Y_\delta(t) \leq e^2\esssup_{t \in (0,T)} \left(\|(v-u)(t)\|^2_{L^2} + \|\theta(t)\|^2_{L^2} + 1\right)  \leq M.
		\end{equation*}
		Define $f(p) := C_2p M^\frac{1}{p} Y^{1-\frac{1}{p}}_\delta(t)$  for $p \in [2,\infty)$, we find that for $p_0 \in [2,\infty)$
		\begin{align*}
			f'(p_0) &= C_2M^\frac{1}{p_0} Y^{1-\frac{1}{p_0}}_\delta(t)\left(1 - \frac{1}{p_0} \log(M) + \frac{1}{p_0} \log(Y_\delta(t))\right) = 0 
			\\
			&\Leftrightarrow\quad  p_0 = \log(M) - \log(Y_\delta(t)) \geq 2,
		\end{align*}
		and choose $p = p_0$ with $f(p) = f(p_0) = C_2eY_\delta(t)(\log(M)-\log(Y_\delta(t)))$. Thus, integrating in time yields
		\begin{equation*}
			Y_\delta(t) \leq Y_\delta(0) + \int^t_0 g(\tau) h(Y_\delta(\tau))\,d\tau \qquad \text{for} \quad t \in (0,T),
		\end{equation*}
		where
		\begin{align*}
			g(t) &:= C_1 \left(1 + \|\nabla v(t)\|^2_{L^2}\right) + eC_2 \sup_{p \in [2,\infty)} \left(1 + \frac{\|\nabla u(t)\|_{L^p}}{p}\right),
			\\
			h(r) &:= r + r(\log(M) - \log(r)) \qquad \forall r \in  (0,a:= e^{-2}M).
		\end{align*}
		In addition, $h$ is a continuous and nondecreasing function with
		\begin{equation*}
			\gamma(x) := \int^a_x \frac{dr}{h(r)} = \log(a + \log(M) - \log(x)) - \log(a + \log(M)) \qquad \forall x \in (0,a),
		\end{equation*}
		which by using the Osgood lemma (see \cite{Bahouri-Chemin-Danchin_2011}) implies that
		\begin{equation*}
			 \log(a + \log(M) - \log(Y_\delta(0))) - \log(a + \log(M) - \log(Y_\delta(t)))  = \gamma(Y_\delta(0))-\gamma(Y_\delta(t))  \leq \int^t_0 g(\tau) \,d\tau =: \tilde{g}(t).
		\end{equation*} 
		Then, it follows that 
		\begin{equation*}
			Y_\delta(t) \leq Y_\delta(0)^{\exp(-\tilde{g}(t))} (\exp(a)M)^{1-\tilde{g}(t)},
		\end{equation*}
		which by letting $\delta \to 0$ and using $Y_0(0) = 0$ yields 
		\begin{equation*}
			\|(v-u)(t)\|^2_{L^2} + \|\nabla(\rho\eta)(t)\|^2_{L^2} = 0 \qquad \text{for} \quad t \in (0,T).
		\end{equation*}
		Therefore, it leads to $v = u$ and $\theta_1 = \theta_2$ as well since $\eta$ is a arbitrary function in $C^\infty_c(\mathbb{R}^2)$. It remains to check that $\tilde{g}(t) < \infty$ for $t \in (0,T)$. It is enough to check 
		\begin{equation} \label{Lp}
			\sup_{p \in [2,\infty)} \frac{1}{p} \int^T_0 \|\nabla u\|_{L^p} \,dt \leq C(T,\nu_2,\mu_1,\lambda_i,v_0,\theta_0),
		\end{equation}
		where its proof is long and can be found in Appendix B (see Section \ref{sec:app}). Thus, the proof is complete.
	\end{proof}
	
	%
	\section{Proof of Theorem \ref{theo_nu2_mu1_de2}} \label{sec:s_ltb}
	%
	
	In this section, we will give a proof of Theorem \ref{theo_nu2_mu1_de2}. As it can be seen later that we need estimates which are independent of time and of regularization parameter as well then we will use another strategy compare to the proof of Theorem \ref{theo_nu2_mu1_de}. 
	
	\begin{proof}[Proof of Theorem \ref{theo_nu2_mu1_de2}: The case $\delta_1 = 0$ and $\delta_2 > 0$.]  The proof contains several steps as follows.
		
		\textbf{Step 1: Local existence of approximate solutions.} Given an $\epsilon > 0$, we consider the following approximate system of \eqref{B} in $\mathbb{R}^2 \times (0,\infty)$
		\begin{equation} \label{B1_nu2_mu1_de2_app} 
			\left\{
			\begin{aligned}
				\partial_t v^\epsilon_1 + J^\epsilon(J^\epsilon v^\epsilon \cdot \nabla J^\epsilon v^\epsilon_1) + \partial_1\pi^\epsilon &=  \nu_2J^\epsilon(\partial_{22} J^\epsilon v^\epsilon_1),
				\\
				\partial_t v^\epsilon_2 + J^\epsilon(J^\epsilon v^\epsilon \cdot \nabla J^\epsilon v^\epsilon_2) + \partial_2\pi^\epsilon &=  \mu_1J^\epsilon(\partial_{11}J^\epsilon v^\epsilon_2) +  \lambda\theta^\epsilon,
				\\
				\partial_t \theta^\epsilon + J^\epsilon(J^\epsilon v^\epsilon \cdot \nabla J^\epsilon \theta^\epsilon) + \lambda v^\epsilon_2 &= \delta_2 J^\epsilon(\partial_{22}J^\epsilon \theta^\epsilon),
				\\
				\textnormal{div}\,v^\epsilon &= 0,
				\\
				(v^\epsilon,\theta^\epsilon)(0) &= (v_0,\theta_0),
			\end{aligned}
			\right.
		\end{equation}
		where $J^\epsilon f$ is the mollification\footnote{For $\epsilon > 0$, $x \in \mathbb{R}^2$ and $f \in L^p(\mathbb{R}^2)$ with $p \in [1,\infty]$, we define
		\begin{equation*}
			J^\epsilon f(x) := \frac{1}{\epsilon^2} \int_{\mathbb{R}^2} \rho\left(\frac{x-y}{\epsilon}\right) f(y)\,dy \quad \text{where} \quad \rho \in C^\infty_0(\mathbb{R}^2), \rho \geq 0, \int_{\mathbb{R}^2}\rho\,dx = 1.
		\end{equation*}} 
		of $f$. We define 
		\begin{equation*}
			V^m(\mathbb{R}^2) := \left\{f \in H^m(\mathbb{R}^2): \textnormal{div}\,f = 0\right\} \qquad \forall m \in \mathbb{N}_0.
		\end{equation*}
		Projecting the first two equations in \eqref{B1_nu2_mu1_de2_app} onto $V^m(\mathbb{R}^2)$ and combining with the other equation, we obtain the following ODE in $V^m(\mathbb{R}^2) \times H^m(\mathbb{R}^2)$
		\begin{equation} \label{B1_nu2_mu1_de2_ODE}
			\frac{d}{dt}(v^\epsilon,\theta^\epsilon) = (F_1^\epsilon,F_2^\epsilon)(v^\epsilon,\theta^\epsilon),
			\quad \textnormal{div}\, v^\epsilon = 0
			\quad \text{and} \quad u^\epsilon_0 =  (v^\epsilon_0,\theta^\epsilon_0) = (v_0,\theta_0),
		\end{equation}
		where $F^\epsilon := (F_1^\epsilon,F_2^\epsilon) : V^m(\mathbb{R}^2) \times H^m(\mathbb{R}^2) \to V^m(\mathbb{R}^2) \times H^m(\mathbb{R}^2)$ and 
		\begin{align*}
			F^\epsilon_1 &:= -\mathbb{P}(J^\epsilon(J^\epsilon v^\epsilon \cdot \nabla J^\epsilon v^\epsilon)) 
			+  \mathbb{P}((\nu_2J^\epsilon(\partial_{22}J^\epsilon v^\epsilon_1),\mu_1J^\epsilon(\partial_{11}J^\epsilon v^\epsilon_2))) + \lambda \mathbb{P}(\theta^\epsilon e_2),
			\\
			F^\epsilon_2 &:= -J^\epsilon(J^\epsilon v^\epsilon \cdot \nabla J^\epsilon \theta^\epsilon) - \lambda v^\epsilon_2 + \delta_2 J^\epsilon(\partial_{22}J^\epsilon \theta^\epsilon).
		\end{align*}
		By using the usual properties of $J^\epsilon$ and $\mathbb{P}$ (see \cite[Lemmas 3.4 and 3.6]{Majda_Bertozzi_2002}), it can be seen that $F^\epsilon$ is well-defined and is also locally Lipschitz continuous under the norm in $V^m(\mathbb{R}^2) \times H^m(\mathbb{R}^2)$ given by 
		\begin{equation*}
			\|(v^\epsilon,\theta^\epsilon)\|_m := \|v^\epsilon\|_{H^m} + \|\theta^\epsilon\|_{H^m}.
		\end{equation*}
		Thus, the Picard theorem gives us a unique solution $(v^\epsilon,\theta^\epsilon) \in C^1\left([0,T_\epsilon),V^m(\mathbb{R}^2) \times H^m(\mathbb{R}^2)\right)$ for some $T_\epsilon > 0$. In addition, if $T_\epsilon < \infty$ then 
		\begin{equation*}
			\lim_{t\to T_\epsilon} \|v^\epsilon(t)\|_{H^m} + \|\theta^\epsilon(t)\|_{H^m} = \infty.
		\end{equation*}
		
		\textbf{Step 2: $H^1$ estimate.} Assume that $(v_0,\theta_0) \in H^1(\mathbb{R}^2)$ then from the previous step there exists a unique solution $(v^\epsilon,\theta^\epsilon)$ to \eqref{B1_nu2_mu1_de2_ODE}  in $C^1\left([0,T_\epsilon),V^1(\mathbb{R}^2) \times H^1(\mathbb{R}^2)\right)$ for some $T_\epsilon > 0$. Taking $L^2$-inner product of \eqref{B1_nu2_mu1_de2_ODE} with $(v^\epsilon,\theta^\epsilon)$, using the properties of $J^\epsilon$ and $\mathbb{P}$, integration by parts, the divergence-free conditions of $v^\epsilon$  and $J^\epsilon v^\epsilon$ and summing up yield (in the sequel, for simplicity we write $(v,\theta)$ and $(J v,J\theta)$ instead of $(v^\epsilon,\theta^\epsilon)$ and $(J^\epsilon v^\epsilon,J^\epsilon \theta^\epsilon)$, respectively)
		\begin{equation} \label{L2_nu2_mu1_de2}
			\frac{1}{2} \frac{d}{dt} \left(\|v\|^2_{L^2} + \|\theta\|^2_{L^2}\right) + \nu_2\|\partial_2 J v_1\|^2_{L^2} + \mu_1\|\partial_1 J v_2\|^2_{L^2} + \delta_2\|\partial_2 J \theta\|^2_{L^2} = 0,
		\end{equation}
		where (and in the sequel) we also used (without saying) the following identity (also for higher derivatives)
		\begin{equation*}
			\lambda\int_{\mathbb{R}^2} \mathbb{P}(\theta e_2) \cdot v - v_2 \theta\,dx =  0.
		\end{equation*}
		Considering the equation for $(\nabla v^\epsilon,\nabla \theta^\epsilon)$ in \eqref{B1_nu2_mu1_de2_ODE} and taking $L^2$-inner product of the resulting system with $(\nabla v^\epsilon,\nabla \theta^\epsilon)$, similar to \eqref{L2_nu2_mu1_de2} we obtain
		\begin{equation} \label{H1_nu2_mu1_de2_1}
			\frac{1}{2} \frac{d}{dt} \left(\|\nabla v\|^2_{L^2} + \|\nabla \theta\|^2_{L^2}\right) + \nu_2\|\nabla \partial_2 J v_1\|^2_{L^2} + \mu_1\| \nabla \partial_1 J v_2\|^2_{L^2} + \delta_2\|\nabla \partial_2 J \theta\|^2_{L^2} = \sum^{4}_{i=1} I_i,
		\end{equation}
		where we used the divergence-free condition of $Jv$ to have 
		\begin{equation*}
			\int_{\mathbb{R}^2} \nabla(J v\cdot\nabla J v):\nabla Jv\,dx = 0\quad \text{and} \quad \int_{\mathbb{R}^2} \nabla(J v\cdot\nabla J\theta) \cdot \nabla J \theta\,dx = \int_{\mathbb{R}^2} \nabla J v \cdot\nabla J \theta \cdot \nabla J \theta\,dx
		\end{equation*}
		with
		\begin{align*}
			I_1 &:= -\int_{\mathbb{R}^2} \partial_1 J v_1 (\partial_1 J \theta)^2 \,dx, & I_2 &:= -\int_{\mathbb{R}^2} \partial_2 J v_2 (\partial_2 J\theta)^2 \,dx,
			\\
			I_3 &:= -\int_{\mathbb{R}^2} \partial_1 J v_2 \partial_1 J \theta \partial_2 J \theta \,dx, & I_4 &:= -\int_{\mathbb{R}^2} \partial_2 J v_1 \partial_1 J \theta \partial_2 J\theta \,dx.
		\end{align*}
		By using the following inequality (see \cite[Lemma 1]{Cao_Wu_2011}) that if $f,g,h \in L^2(\mathbb{R}^2)$ with $\partial_1 g, \partial_2 h \in L^2(\mathbb{R}^2)$ then 
		\begin{equation*} 
			\int_{\mathbb{R}^2} fgh \,dx \leq C \|f\|_{L^2} \|g\|^{\frac{1}{2}}_{L^2} \|\partial_1 g\|^{\frac{1}{2}}_{L^2} \|h\|^{\frac{1}{2}}_{L^2} \|\partial_2 h\|^{\frac{1}{2}}_{L^2},
		\end{equation*}
		which is applied to bound the above integrals as follows
		\begin{align*}
			I_1 &= -2\int_{\mathbb{R}^2} J v_2 \partial_1 J\theta \partial_{21} J\theta \,dx 
			\leq C\|\partial_{21} J\theta\|_{L^2} \|\partial_1 J\theta\|^{\frac{1}{2}}_{L^2} \|\partial_{21} J\theta\|^{\frac{1}{2}}_{L^2} \|Jv_2\|^\frac{1}{2}_{L^2} \|\partial_1 Jv_2\|^{\frac{1}{2}}_{L^2}
			\\
			&\leq C\left(\|\partial_1 \theta\|_{L^2} + \|v_2\|_{L^2}\right) \left(\|\partial_2 J\theta\|^2_{H^1} + \|\partial_1 Jv_2\|^2_{L^2}\right);
			\\
			I_2 &= -2\int_{\mathbb{R}^2} J v_1 \partial_{12} J \theta \partial_2 J\theta \,dx \leq C\|\partial_{12} \theta\|_{L^2} \|\partial_2 J\theta\|^{\frac{1}{2}}_{L^2} \|\partial_{12} J\theta\|^{\frac{1}{2}}_{L^2} \|Jv_1\|^\frac{1}{2}_{L^2} \|\partial_2 Jv_1\|^{\frac{1}{2}}_{L^2}
			\\
			&\leq C\left(\|\partial_2 \theta\|_{L^2} + \|v_1\|_{L^2}\right) \left(\|\partial_2 J\theta\|^2_{H^1} + \|\partial_2 Jv_1\|^2_{L^2}\right);
			\\
			I_3 &\leq C\|\partial_1 J\theta\|_{L^2} \|\partial_2 J\theta\|^{\frac{1}{2}}_{L^2} \|\partial_{22} J\theta\|^{\frac{1}{2}}_{L^2} \|\partial_1 Jv_2\|^\frac{1}{2}_{L^2} \|\partial_{11} Jv_2\|^{\frac{1}{2}}_{L^2} 
			\\
			&\leq C\|\partial_1 \theta\|_{L^2}\left(\|\partial_2 J\theta\|^2_{H^1} + \|\partial_1 Jv_2\|^2_{L^2}\right);
			\\
			I_4 &\leq C\|\partial_1 J\theta\|_{L^2} \|\partial_2 J\theta\|^{\frac{1}{2}}_{L^2} \|\partial_{22} J\theta\|^{\frac{1}{2}}_{L^2} \|\partial_2 Jv_1\|^\frac{1}{2}_{L^2} \|\partial_{12} Jv_2\|^{\frac{1}{2}}_{L^2} 
			\\
			&\leq C\|\partial_1 \theta\|_{L^2}\left(\|\partial_2 J\theta\|^2_{H^1} + \|\partial_2 Jv_1\|^2_{L^2}\right);
		\end{align*}
		which combine with \eqref{L2_nu2_mu1_de2}-\eqref{H1_nu2_mu1_de2_1} implies that for a positive constant $C$ which does not depend on $\epsilon$
		\begin{multline} \label{H1_nu2_mu1_de2_2}
			\frac{1}{2} \frac{d}{dt} \left(\|v\|^2_{H^1} + \|\theta\|^2_{H^1}\right) + \nu_2\|\partial_2 Jv_1\|^2_{H^1} + \mu_1\|\partial_1 Jv_2\|^2_{H^1} + \delta_2\|\partial_2 J\theta\|^2_{H^1} \\\leq C \left(\|\theta\|_{H^1} + \|v\|_{H^1}\right) \left(\|\partial_2 Jv_1\|^2_{H^1} + \|\partial_1 Jv_2\|^2_{H^1} + \|\partial_2 J\theta\|^2_{H^1}\right).
		\end{multline}
		
		\textbf{Step 3: $H^2$ estimate.} Assume that $(v_0,\theta_0) \in H^2(\mathbb{R}^2)$ then there exists a unique solution $(v^\epsilon,\theta^\epsilon)$ to \eqref{B1_nu2_mu1_de2_ODE} in $C^1\left([0,T_\epsilon),V^2(\mathbb{R}^2) \times H^2(\mathbb{R}^2)\right)$ for some $T_\epsilon > 0$. Similar to the previous step, we obtain 
		\begin{equation} \label{H2_nu2_mu1_de2_1}
			\frac{1}{2} \frac{d}{dt} \left(\|\Delta v\|^2_{L^2} + \|\Delta \theta\|^2_{L^2}\right) + \nu_2\|\Delta \partial_2 Jv_1\|^2_{L^2} + \mu_1\|\Delta\partial_1 Jv_2\|^2_{L^2} + \delta_2\|\Delta \partial_2 J\theta\|^2_{L^2} = \sum_{i=1}^2 J_i,
		\end{equation}
		where by using the divergence-free condition, H\"{o}lder and 2D Ladyzhensaya inequalities
		\begin{align*}
			J_1 &:=  -\int_{\mathbb{R}^2} \Delta Jv \cdot \nabla Jv \cdot \Delta Jv \,dx - 2\int_{\mathbb{R}^2} \nabla Jv:\nabla\nabla Jv_1 \Delta Jv_1 + \nabla Jv:\nabla\nabla Jv_2 \Delta Jv_2\,dx
			\\
			&\leq C \|\nabla v\|_{L^2} \|\Delta Jv\|^2_{L^4} + C\left(\|\nabla\nabla v_1\|_{L^2} + \|\nabla\nabla v_2\|_{L^2}\right) \|\nabla Jv\|_{L^4} \|\Delta Jv\|_{L^4} 
			\\
			&\leq C \|v\|_{H^2}\left(\|\partial_2 Jv_1\|^2_{H^2} + \|\partial_1 Jv_2\|^2_{H^2}\right).
		\end{align*}
		We continue with the second term on the right hand side of \eqref{H2_nu2_mu1_de2_1}
		\begin{align*}
			J_2 &:= -\int_{\mathbb{R}^2} \Delta Jv \cdot \nabla J\theta \Delta J\theta  + 2 \nabla Jv :  \nabla \nabla J\theta \Delta J\theta \,dx
			\\
			&= -\int_{\mathbb{R}^2} \Delta Jv_1 \partial_1J\theta \Delta J\theta + \Delta Jv_2 \partial_2J\theta \Delta J\theta + 2 \partial_1Jv \cdot \partial_1 \nabla J\theta \Delta J\theta 
			+ 2 \partial_2Jv \cdot \partial_2 \nabla J\theta \Delta J\theta \,dx 
			\\
			&=: \sum^4_{i=1}J_{2i}.
		\end{align*}
		The first term can be rewritten as
		\begin{align*}
			J_{21} &= -\int_{\mathbb{R}^2} \partial_{11}Jv_1 \partial_1J\theta \partial_{11}J\theta + \partial_{11}Jv_1 \partial_1J\theta \partial_{22}J\theta + \partial_{22}Jv_1 \partial_1J\theta \partial_{11}J\theta + \partial_{22}Jv_1 \partial_1J\theta \partial_{22}J\theta \,dx
			\\
			&=: \sum^4_{i=1} J_{21i},
		\end{align*}
		and by using the divergence-free condition, integration by parts and Young inequality, each term can be estimated as follows
		\begin{align*}
			J_{211} &\leq C\|\partial_{12}Jv_2\|_{L^2} \|\partial_1J\theta\|^{\frac{1}{2}}_{L^2} \|\partial_{21}J\theta\|^{\frac{1}{2}}_{L^2} \|\partial_{11}J\theta\|^{\frac{1}{2}}_{L^2} \|\partial_{211}J\theta\|^{\frac{1}{2}}_{L^2} 
			\\
			&\leq C\|\theta\|_{H^2}\left(\|\partial_1 Jv_2\|^2_{H^2} + \|\partial_2J\theta\|^2_{H^2}\right);
			\\
			J_{212} &\leq C\|\partial_{12}Jv_2\|_{L^2} \|\partial_1J\theta\|^{\frac{1}{2}}_{L^2} \|\partial_{21}J\theta\|^{\frac{1}{2}}_{L^2} \|\partial_{22}J\theta\|^{\frac{1}{2}}_{L^2} \|\partial_{122}J\theta\|^{\frac{1}{2}}_{L^2} 
			\\
			&\leq C\|\theta\|_{H^2}\left(\|\partial_1 Jv_2\|^2_{H^2} + \|\partial_2J\theta\|^2_{H^2}\right);
			\\
			J_{213} &= \int_{\mathbb{R}^2} \partial_2 Jv_1 \partial_{21}J\theta \partial_{11}J\theta + \partial_2 Jv_1 \partial_1J\theta \partial_{211}J\theta \,dx =: J_{2131} + J_{2132},
			\\
			J_{2131} &\leq C\|\partial_{21}J\theta\|_{L^2}\|\partial_2Jv_1\|^\frac{1}{2}_{L^2} \|\partial_{12}Jv_1\|^\frac{1}{2}_{L^2} \|\partial_{11}J\theta\|^{\frac{1}{2}}_{L^2} \|\partial_{211}J\theta\|^{\frac{1}{2}}_{L^2} 
			\\
			&\leq C \|\theta\|_{H^2}\left(\|\partial_2 Jv_1\|^2_{H^2} + \|\partial_2J\theta\|^2_{H^2}\right),
			\\
			J_{2132} &\leq C\|\partial_{211}J\theta\|_{L^2}\|\partial_2Jv_1\|^\frac{1}{2}_{L^2} \|\partial_{12}Jv_1\|^\frac{1}{2}_{L^2} \|\partial_1J\theta\|^{\frac{1}{2}}_{L^2} \|\partial_{21}J\theta\|^{\frac{1}{2}}_{L^2} 
			\\
			&\leq C \|\theta\|_{H^2}\left(\|\partial_2 Jv_1\|^2_{H^2} + \|\partial_2J\theta\|^2_{H^2}\right);
			\\
			J_{214} &\leq C\|\partial_{22}J\theta\|_{L^2}\|\partial_1J\theta\|^{\frac{1}{2}}_{L^2} \|\partial_{21}J\theta\|^{\frac{1}{2}}_{L^2} \|\partial_{22}Jv_1\|^\frac{1}{2}_{L^2} \|\partial_{122}Jv_1\|^\frac{1}{2}_{L^2} 
			\\
			&\leq C \|\theta\|_{H^2}\left(\|\partial_2 Jv_1\|^2_{H^2} + \|\partial_2J\theta\|^2_{H^2}\right).
		\end{align*}
		The second term is bounded by  
		\begin{align*}
			J_{22} &\leq C \|\partial_2J\theta\|_{L^2} \|\Delta Jv_2\|^{\frac{1}{2}}_{L^2} \|\partial_1 \Delta Jv_2\|^{\frac{1}{2}}_{L^2} \|\Delta J\theta\|^\frac{1}{2}_{L^2} \|\partial_2 \Delta J\theta\|^\frac{1}{2}_{L^2} 
			\\
			&\leq C \left(\|v\|_{H^2} + \|\theta\|_{H^2}\right)\left(\|\partial_1 Jv_2\|^2_{H^2} + \|\partial_2J\theta\|^2_{H^2}\right).
		\end{align*}
		The third term can be expressed by
		\begin{equation*}
			J_{23} = -2\int_{\mathbb{R}^2} \partial_1Jv_1 \partial_{11}J\theta \partial_{11}J\theta + \partial_1Jv_1 \partial_{11}J\theta \partial_{22}J\theta + \partial_1Jv_2 \partial_{12}J\theta \Delta J\theta \,dx =: \sum^3_{i=1} J_{23i}
		\end{equation*}
		and similar to the first term, each integral can be bounded by
		\begin{align*}
			J_{231} &= 2\int_{\mathbb{R}^2} \partial_2Jv_2 \partial_{11}J\theta \partial_{11}J\theta \,dx = -4 \int_{\mathbb{R}^2} Jv_2 \partial_{211}J\theta \partial_{11}J\theta \,dx
			\\
			&\leq  C\|\partial_{211}J\theta\|_{L^2}\|\partial_{11}J\theta\|^{\frac{1}{2}}_{L^2} \|\partial_{211}J\theta\|^{\frac{1}{2}}_{L^2} \|Jv_2\|^\frac{1}{2}_{L^2} \|\partial_1Jv_2\|^\frac{1}{2}_{L^2}  
			\\
			&\leq C \left(\|v\|_{H^2} + \|\theta\|_{H^2}\right)\left(\|\partial_1 Jv_2\|^2_{H^2} + \|\partial_2J\theta\|^2_{H^2}\right);
			\\
			J_{232} &= 2\int_{\mathbb{R}^2} \partial_2Jv_2 \partial_{11}J\theta \partial_{22}J\theta \,dx 
			\\
			&\leq  C\|\partial_{22}J\theta\|_{L^2}\|\partial_{11}J\theta\|^{\frac{1}{2}}_{L^2} \|\partial_{211}J\theta\|^{\frac{1}{2}}_{L^2} \|\partial_2 Jv_2\|^\frac{1}{2}_{L^2} \|\partial_{12}Jv_2\|^\frac{1}{2}_{L^2}  
			\\
			&\leq C \left(\|v\|_{H^2} + \|\theta\|_{H^2}\right)\left(\|\partial_1 Jv_2\|^2_{H^2} + \|\partial_2J\theta\|^2_{H^2}\right);
			\\
			J_{233} &\leq C \|\Delta J\theta\|_{L^2}\|\partial_{12}J\theta\|^{\frac{1}{2}}_{L^2} \|\partial_{212}J\theta\|^{\frac{1}{2}}_{L^2} \|\partial_1Jv_2\|^\frac{1}{2}_{L^2} \|\partial_{11}Jv_2\|^\frac{1}{2}_{L^2}  
			\\
			&\leq C \left(\|v\|_{H^2} + \|\theta\|_{H^2}\right)\left(\|\partial_1 Jv_2\|^2_{H^2} + \|\partial_2J\theta\|^2_{H^2}\right).
		\end{align*}
		The last term can be estimated by
		\begin{align*}
			J_{24} &\leq C \|\partial_2\nabla J\theta\|_{L^2}\|\Delta J\theta\|^{\frac{1}{2}}_{L^2} \|\partial_2 \Delta J\theta\|^{\frac{1}{2}}_{L^2} \|\partial_2Jv\|^\frac{1}{2}_{L^2} \|\partial_{12}Jv\|^\frac{1}{2}_{L^2}  
			\\
			&\leq C \left(\|v\|_{H^2} + \|\theta\|_{H^2}\right)\left(\|\partial_2 Jv_1\|^2_{H^2} + \|\partial_1 Jv_2\|^2_{H^2} + \|\partial_2 J\theta\|^2_{H^2}\right).
		\end{align*}
		Collecting all above estimates with using \eqref{L2_nu2_mu1_de2} and \eqref{H2_nu2_mu1_de2_1} we obtain
		\begin{multline} \label{H2_nu2_mu1_de2_2}
			\frac{1}{2} \frac{d}{dt} \left(\|v\|^2_{H^2} + \|\theta\|^2_{H^2}\right) + \nu_2\|\partial_2 Jv_1\|^2_{H^2} + \mu_1\|\partial_1 Jv_2\|^2_{H^2} + \delta_2\|\partial_2 J\theta\|^2_{H^2} 
			\\
			\leq C \left(\|v\|_{H^2} + \|\theta\|_{H^2}\right)\left(\|\partial_1 Jv_2\|^2_{H^2} + \|\partial_2 Jv_1\|^2_{H^2} +  \|\partial_2J\theta\|^2_{H^2}\right).
		\end{multline}
		
		\textbf{Step 4: $H^m$ estimate.} Let $m \in \mathbb{N}$ with $m \geq 3$. As in the previous parts, together with the energy estimate, we have 
		\begin{equation*}
			\frac{1}{2}\frac{d}{dt}\left(\|v\|^2_{H^m} + \|\theta\|^2_{H^m}\right) + \nu_2 \|\partial_2Jv_1\|^2_{H^m} + \mu_1 \|\partial_1Jv_2\|^2_{H^m} + \delta_2 \|\partial_2 J\theta\|^2_{H^m} = \sum^3_{i=1} K_i,
		\end{equation*}
		where by using the divergence-free condition and  calculus inequalities (see \cite[Lemma 3.4]{Majda_Bertozzi_2002}) 
		\begin{align*}
			K_1 &:= - \sum_{|\alpha| = m} \int_{\mathbb{R}^2} D^\alpha \mathbb{P}(J(Jv \cdot \nabla Jv_1)) D^\alpha v_1 \,dx 
			\\
			&\leq C(m) \left(\|\nabla Jv\|_{L^\infty} \|D^m\nabla Jv_1\|_{L^2} + \|D^m Jv\|_{L^2}\|\nabla Jv_1\|_{L^\infty}\right) \|D^m Jv_1\|_{L^2}
			\\
			&\leq C(m)\|v\|_{H^m} \left(\|D^mJv\|_{L^2} + \|D^{m+1}Jv\|_{L^2}\right)\|D^mJv\|_{L^2}
			\\
			&\leq C(m)\|v\|_{H^m} \left(\|\partial_2Jv_1\|^2_{H^m} + \|\partial_1Jv_2\|^2_{H^m}\right);
			\\
			K_2 &:= - \sum_{|\alpha| = m} \int_{\mathbb{R}^2} D^\alpha \mathbb{P}(J(Jv \cdot \nabla Jv_2)) D^\alpha v_2\,dx
			\leq C(m)\|v\|_{H^m} \left(\|\partial_2Jv_1\|^2_{H^m} + \|\partial_1Jv_2\|^2_{H^m}\right);
			\\
			K_3 &:= - \sum_{|\alpha| = m} \int_{\mathbb{R}^2} D^\alpha J(Jv \cdot \nabla J\theta) D^\alpha \theta\,dx. 
		\end{align*}
		It remains to estimate $K_3$. If $D^\alpha$ contains at least one time of $\partial_2$ then we decompose $K_3$ into $K_3 = K_{31} + K_{32}$, where
		\begin{align*}
			K_{31} &:= - \sum_{|\alpha| = m} \int_{\mathbb{R}^2} D^{\alpha-1} (\partial_2 Jv \cdot \nabla J\theta) D^{\alpha-1} \partial_2 J\theta\,dx 
			\\
			&\leq \sum_{|\alpha| = m} \|D^{\alpha-1} (\partial_2 Jv \cdot \nabla J\theta)\|_{L^2} \|D^{m-1} \partial_2 J\theta\|_{L^2}
			\\
			&\leq C(m)\left(\|\partial_2 Jv\|_{L^\infty} \|D^{m-1} \nabla J\theta\|_{L^2} + \|D^{m-1}\partial_2Jv\|_{L^2} \|\nabla J\theta\|_{L^\infty}\right)  \|D^{m-1} \partial_2J \theta\|_{L^2}
			\\
			&\leq C(m) \|\theta\|_{H^m}\left(\|\partial_2Jv_1\|^2_{H^m} + \|\partial_1Jv_2\|^2_{H^m} + \|\partial_2J\theta\|^2_{H^{m-1}}\right),
		\end{align*}
		here, we used the fact that since $m \geq 3$
		\begin{equation*}
			\|\nabla Jv\|_{L^\infty} \leq C(m) \|\nabla Jv\|_{H^m} \leq C(m)\|\partial_1v_2-\partial_2v_1\|_{H^m};
		\end{equation*}
		and similarly, 
		\begin{align*}
			K_{32} &:= - \sum_{|\alpha| = m} \int_{\mathbb{R}^2} \left(D^{\alpha-1} (Jv \cdot \nabla \partial_2J\theta) - Jv \cdot \nabla D^{\alpha-1}\partial_2J\theta\right) D^{\alpha-1} \partial_2 J\theta\,dx
			\\
			&\leq \sum_{|\alpha| = m}\|D^{\alpha-1} (Jv \cdot \nabla \partial_2J\theta) - Jv \cdot \nabla D^{\alpha-1}\partial_2J\theta\|_{L^2} \|D^{m-1} \partial_2 J\theta\|_{L^2}
			\\
			&\leq C(m)\left(\|\nabla Jv\|_{L^\infty} \|D^{m-1}\nabla \partial_2J\theta\|_{L^2} +  \|D^{m-1}Jv\|_{L^2} \|\nabla \partial_2J\theta\|_{L^\infty}\right) \|D^{m-1} \partial_2 J\theta\|_{L^2}.
			\\
			&\leq C(m) \|v\|_{H^m}\|\partial_2J\theta\|^2_{H^m}.
		\end{align*}
		Otherwise, if $D^\alpha$ only contains $\partial_1$ then 
		\begin{align*}
			K_3 &= -m\int_{\mathbb{R}^2} \partial_1^m (Jv_1 \partial_1J\theta + Jv_2 \partial_2J\theta) \partial_1^m J\theta\,dx = K_{34} + K_{35},
		\end{align*}
		where by using the Leibniz differentiation formula 
		\begin{align*}
			K_{35} &:= -m\int_{\mathbb{R}^2} \partial_1^m (Jv_2 \partial_2J\theta) \partial_1^m J\theta\,dx = -m \sum^m_{n = 0} C(n,m) \int_{\mathbb{R}^2} \partial_1^n (Jv_2) \partial_1^{m-n}(\partial_2J\theta) \partial^m_1\theta \,dx
			\\
			&\leq \sum^m_{n = 0} C(n,m) \|\partial_1^{m-n}(\partial_2J\theta)\|_{L^2} \|\partial_1^n (Jv_2)\|^\frac{1}{2}_{L^2}\|\partial_1^n (\partial_1Jv_2)\|^\frac{1}{2}_{L^2} \|\partial_1^m (J\theta)\|^\frac{1}{2}_{L^2} \|\partial_1^m (\partial_2J\theta)\|^\frac{1}{2}_{L^2}
			\\
			&\leq C(m) \left(\|v\|_{H^m} + \|\theta\|_{H^m}\right) \left( \|\partial_1Jv_2\|^2_{H^m} + \|\partial_2J\theta\|^2_{H^m}\right).
		\end{align*}
		Similarly, we rewrite $K_{34} = K_{341} + K_{342}$, where
		\begin{align*}
			K_{341} &:=  -m \int_{\mathbb{R}^2} \partial_1^m (\partial_2Jv_2 \partial_1J\theta) \partial_1^{m-1} J\theta\,dx
			\\
			&= -\sum^m_{n = 0} C(n,m)\int_{\mathbb{R}^2} \partial_1^n (\partial_2Jv_2) \partial^{m-n+1}_1(J\theta) \partial_1^{m-1} J\theta\,dx
			\\
			&= \sum^m_{n = 0} C(n,m)\int_{\mathbb{R}^2} \partial_1^n Jv_2\left[ \partial^{m-n+1}_1(\partial_2J\theta) \partial_1^{m-1} J\theta + \partial^{m-n+1}_1(J\theta) \partial_1^{m-1} \partial_2J\theta \right]\,dx
			=: \sum^m_{n=0} K_{341n}.
		\end{align*}
		If $n = 0$ then we write $K_{3410} = K_{34101} + K_{34102}$, where
		\begin{align*}
			K_{34101} &:= -C(m) \int_{\mathbb{R}^2} 2Jv_2 \partial^m_1(\partial_2J\theta) \partial_1^m J\theta \,dx
			\\
			&\leq C(m)\|\partial_1^{m} \partial_2 J\theta\|_{L^2} \|Jv_2\|^\frac{1}{2}_{L^2}\|\partial_1Jv_2\|^\frac{1}{2}_{L^2} \|\partial_1^{m} J\theta\|^\frac{1}{2}_{L^2} \|\partial_1^{m} \partial_2 J\theta\|^\frac{1}{2}_{L^2}
			\\
			&\leq C(m) \left(\|v\|_{H^m} + \|\theta\|_{H^m}\right) \left( \|\partial_1Jv_2\|^2_{H^m} + \|\partial_2J\theta\|^2_{H^m}\right);
			\\ 
			K_{34102} &:= -C(m) \int_{\mathbb{R}^2} \partial_1Jv_2[\partial^m_1(\partial_2J\theta)\partial^{m-1}_1J\theta +  \partial_1^{m-1} (\partial_2J\theta) \partial^m_1J\theta]\,dx
			\\
			&\leq C(m) \left(\|v\|_{H^m} + \|\theta\|_{H^m}\right) \left( \|\partial_1Jv_2\|^2_{H^m} + \|\partial_2J\theta\|^2_{H^m}\right).
		\end{align*}
		If $n = m$ then 
		\begin{align*}
			K_{341m} &= C(m)\int_{\mathbb{R}^2} \partial_1^m (Jv_2)\left[ \partial_1(\partial_2J\theta) \partial_1^{m-1} J\theta + \partial_1(J\theta) \partial_1^{m-1} \partial_2J\theta \right]\,dx
			\\
			&\leq C(m) \|\theta\|_{H^m}\left( \|\partial_1Jv_2\|^2_{H^m} + \|\partial_2J\theta\|^2_{H^m}\right).
		\end{align*}
		If $1 \leq n \leq m-1$ then
		\begin{align*}
			\sum^{m-1}_{n=1} K_{341n} &\leq C(m) \left(\|v\|_{H^m} + \|\theta\|_{H^m}\right)\left(\|\partial_1Jv_2\|^2_{H^m} + \|\partial_2J\theta\|^2_{H^m}\right).
		\end{align*}
		We now estimate the last term $K_{342}$ as follows
		\begin{align*}
			K_{342} &:= m\int_{\mathbb{R}^2} \partial_1^m (Jv_1 \partial_{11}J\theta)  \partial_1^{m-1} J\theta\,dx 
			\\
			&= \sum^m_{n=0} C(n,m) \int_{\mathbb{R}^2} \partial_1^n (Jv_1) \partial^{m-n+2}_1(J\theta)  \partial_1^{m-1} J\theta\,dx
			=: \sum^m_{n=0} K_{342n}.
		\end{align*}
		If $n = 0$ then 
		\begin{align*}
			K_{3420} &= C(m) \int_{\mathbb{R}^2} Jv_1 \partial^{m+2}_1(J\theta)  \partial_1^{m-1} J\theta\,dx
			\\
			&= C(m) \int_{\mathbb{R}^2} \partial_2Jv_2 \partial^{m+1}_1(J\theta)  \partial_1^{m-1} J\theta - Jv_1 \partial^{m+1}_1(J\theta)  \partial_1^{m} J\theta\,dx =: K_{34201} + K_{34202},
		\end{align*}
		where
		\begin{align*}
			K_{34201} &= -C(m) \int_{\mathbb{R}^2} Jv_2 [\partial^{m+1}_1(\partial_2J\theta)  \partial_1^{m-1} J\theta + \partial^{m+1}_1(J\theta)  \partial_1^{m-1} \partial_2J\theta]\,dx
			\\
			&= C(m) \int_{\mathbb{R}^2} \partial_1Jv_2 [\partial^{m}_1(\partial_2J\theta)  \partial_1^{m-1} J\theta + \partial^{m-1}_1(\partial_2J\theta)  \partial_1^m J\theta] + 2Jv_2\partial^{m}_1(J\theta)  \partial_1^m \partial_2J\theta\,dx
			\\
			&\leq C(m) \left(\|v\|_{H^m} + \|\theta\|_{H^m}\right)\left(\|\partial_1Jv_2\|^2_{H^m} + \|\partial_2J\theta\|^2_{H^m}\right);
			\\
			K_{34202} &= -C(m)\int_{\mathbb{R}^2} Jv_1 \partial_1(\partial^m_1J\theta)^2 \,dx 
			= C(m)\int_{\mathbb{R}^2} Jv_2 \partial_1^{m} (\partial_2 J\theta) \partial_1^{m} J\theta \,dx
			\\
			&\leq C(m) \left(\|v\|_{H^m}+ \|\theta\|_{H^m}\right)\left(\|\partial_1Jv_2\|^2_{H^m} + \|\partial_2J\theta\|^2_{H^m}\right).
		\end{align*}
		If $n = 1$ then similar to the estimates of $K_{34101}$ and $K_{34102}$
		\begin{align*}
			K_{3421} &= -C(m) \int_{\mathbb{R}^2} \partial_2 (Jv_2) \partial^{m+1}_1(J\theta)  \partial_1^{m-1} J\theta\,dx
			\\
			&= C(m) \int_{\mathbb{R}^2} Jv_2[ \partial^{m+1}_1(\partial_2J\theta)  \partial_1^{m-1} J\theta + \partial^{m+1}_1(J\theta)  \partial_1^{m-1} \partial_2J\theta]\,dx
			\\
			&\leq C(m) \left(\|v\|_{H^m} + \|\theta\|_{H^m}\right)\left(\|\partial_1Jv_2\|^2_{H^m} + \|\partial_2J\theta\|^2_{H^m}\right).
		\end{align*}
		If $2 \leq n \leq m$ then $m-n+2 \leq m$ and 
		\begin{align*}
			\sum^m_{n=2} K_{342n} &= -\sum^m_{n=2} C(n,m) \int_{\mathbb{R}^2} \partial_1^{n-1} (\partial_2Jv_2) \partial^{m-n+2}_1(J\theta)  \partial_1^{m-1} J\theta\,dx
			\\
			&= \sum^m_{n=1} C(n,m) \int_{\mathbb{R}^2} \partial_1^{n-1} (Jv_2)[ \partial^{m-n+2}_1(\partial_2J\theta)  \partial_1^{m-1} J\theta + \partial^{m-n+2}_1(J\theta)  \partial_1^{m-1} \partial_2J\theta]\,dx
			\\
			&\leq C(m) \left(\|v\|_{H^m} + \|\theta\|_{H^m}\right)\left(\|\partial_1Jv_2\|^2_{H^m} + \|\partial_2J\theta\|^2_{H^m}\right).
		\end{align*}
		Therefore, 
		\begin{multline} \label{Hm_nu2_mu1_de2}
			\frac{1}{2}\frac{d}{dt}\left(\|v\|^2_{H^m} + \|\theta\|^2_{H^m}\right) + \nu_2 \|\partial_2Jv_1\|^2_{H^m} + \mu_1 \|\partial_1Jv_2\|^2_{H^m} + \delta_2 \|\partial_2 J\theta\|^2_{H^m} 
			\\
			\leq C(m) \left(\|v\|_{H^m} + \|\theta\|_{H^m}\right)\left(\|\partial_2Jv_1\|^2_{H^m} + \|\partial_1Jv_2\|^2_{H^m} + \|\partial_2J\theta\|^2_{H^m}\right).
		\end{multline}
	
		\textbf{Step 5: Bootstrapping argument.}
		We define for all $t \in (0,T_\epsilon)$ and $m \in \mathbb{N}$
		\begin{align*}
			E^\epsilon_m(t) &:= \esssup_{s\in [0,t]} \left(\|v^\epsilon(s)\|^2_{H^m} + \|\theta^\epsilon(s)\|^2_{H^m}\right) 
			\\
			&\quad + 2\int_0^t \nu_2\|\partial_2 J^\epsilon v^\epsilon_1\|^2_{H^m} + \mu_1\|\partial_1 J^\epsilon v^\epsilon_2\|^2_{H^m} + \delta_2\|\partial_2 J^\epsilon\theta^\epsilon\|^2_{H^m} \,ds.
		\end{align*}
		Integrating in time \eqref{H1_nu2_mu1_de2_2}, \eqref{H2_nu2_mu1_de2_2} and \eqref{Hm_nu2_mu1_de2}, there exists $C_1 = C_1(\nu_2,\mu_1,\delta_2,m) > 0$ such that
		\begin{equation} \label{Hm_nu2_mu1_de2_main}
			E^\epsilon_m(t) \leq E^\epsilon_m(0) + C_1 (E^\epsilon_m(t))^\frac{3}{2} \qquad \forall t \in (0,T_\epsilon).
		\end{equation}
		To the end of this step, we aim to prove: \textit{Claim: Let $S := \{t \in (0,T^\epsilon) : E^\epsilon_m(t) \leq 2\epsilon^2_0\}$. Then $S = (0,T_\epsilon)$ and $T_\epsilon = \infty$.}
		
		\textit{5a) Hypothesis implies conclusion.} Assume that for some $t \in (0,T_\epsilon)$
		\begin{equation} \label{Hm_nu2_mu1_de2_2}
			E^\epsilon_m(t) \leq (4C^2_1)^{-1}.
		\end{equation}
		Therefore, by choosing $\epsilon_0 > 0$ such that $4C_1\epsilon_0 \leq 1$, it follows from  \eqref{Hm_nu2_mu1_de2_main} and \eqref{Hm_nu2_mu1_de2_2}  that 
		\begin{equation} \label{Hm_nu2_mu1_de2_3}
			E^\epsilon_m(t) \leq 2E^\epsilon_m(0)  \leq 2\epsilon^2_0 \leq (8C^2_1)^{-1}.
		\end{equation}
		
		\textit{5b) Conclusion is stronger than hypothesis.} Assume that \eqref{Hm_nu2_mu1_de2_3} holds for some $t_0 \in (0,T_\epsilon)$. For a given $\delta_0 > 0$, by the continuity
		\begin{equation*}
			E^\epsilon_m(t) < E^\epsilon_m(t_0) + \delta_0 \leq (8C^2_1)^{-1} + \delta_0 \qquad \forall t \in (t_0-t_{\delta_0},t_0+t_{\delta_0}),
		\end{equation*}
		which yields \eqref{Hm_nu2_mu1_de2_2} if we choose $\delta_0 \leq (8C^2_1)^{-1}$.
		
		\textit{5c) Conclusion is closed.} Let $t_n,t \in (0,T_\epsilon)$ such that $t_n \to t$ as $n \to \infty$. If $E^\epsilon_m(t_n) \leq 2\epsilon^2_0$ for all $n \in \mathbb{N}$ then by the continuity we obtain $E^\epsilon_m(t) \leq 2\epsilon^2_0$ as well.
		
		\textit{5d) Base case.} By the continuity in time of $(v^\epsilon,\theta^\epsilon)$ in $V^m \times H^m$, we can find some $T'_\epsilon \in (0,T_\epsilon)$ 
		\begin{equation*} 
			E^\epsilon_m(t) \leq 2E^\epsilon_m(0) \leq 2\epsilon^2_0 \leq (4C^2_1)^{-1} \qquad \forall t \in (0,T'_\epsilon).
		\end{equation*}
		This implies that $S$ is a non-empty set. We then apply the abstract bootstrap principle (see \cite[Proposition 1.21]{Tao_2006}) to obtain the first part of the claim. While the second part follows immediately by using Step 1.
		
		\textbf{Step 6: Passing to the limit.} From the previous steps and \eqref{B1_nu2_mu1_de2_ODE}, we find that for $m \in \mathbb{N}$
		\begin{align*}
			(v^\epsilon,\theta^\epsilon) \quad &\text{is uniformly bounded in}\quad  L^\infty(0,\infty;H^m(\mathbb{R}^2));
			\\
			(\partial_2J^\epsilon v^\epsilon_1, \partial_1J^\epsilon v^\epsilon_2, \partial_2J^\epsilon \theta^\epsilon) \quad   &\text{is uniformly bounded in} \quad L^2(0,\infty;H^m(\mathbb{R}^2));
			\\
			(\partial_tv^\epsilon,\partial_t\theta^\epsilon) \quad &\text{is uniformly bounded in}  \quad L^2(0,T;H^{m-1}(\mathbb{R}^2)) \qquad \forall T \in (0,\infty).
		\end{align*}
		Then there exist a subsequence (still denoted by) $(v^\epsilon, \theta^\epsilon)$ and $(v,\theta)$ such that as $\epsilon \to 0$
		\begin{align*}
			&&(v^\epsilon,\theta^\epsilon) &\xrightharpoonup{*} (v,\theta) & &\text{in} \quad L^\infty(0,\infty;H^m(\mathbb{R}^2)),&&
			\\
			&&(\partial_2J^\epsilon v^\epsilon_1, \partial_1J^\epsilon v^\epsilon_2, \partial_2J^\epsilon \theta^\epsilon)
			&\rightharpoonup (\partial_2v_1,\partial_1v_2,\partial_2\theta) & &\text{in} \quad L^2(0,\infty;H^m(\mathbb{R}^2)),&&
			\\
			&&(\partial_tv^\epsilon,\partial_t\theta^\epsilon) &\rightharpoonup
			(\partial_t v,\partial_t\theta) & &\text{in}\quad L^2(0,T;H^{m-1}(\mathbb{R}^2)) \qquad \forall T \in (0,\infty).&&
		\end{align*}
		Therefore, for $m = 1$, by using the local compact embedding $H^1 \hookrightarrow L^2 \hookrightarrow H^{-1}$ and applying the Aubin-Lions lemma, we conclude that, for a given $R \in \mathbb{N}$, there exists a
		subsequence $(v^{\epsilon,R},\theta^{\epsilon,R})$ such that for all $T \in (0,\infty), x_0 \in \mathbb{R}^2$ as $\epsilon \to 0$
		\begin{align*}
			&&(v^{\epsilon,R},\theta^{\epsilon,R}) &\rightarrow (v,\theta) &&\text{in} \quad  L^2(0,T;L^2(B_R(x_0))),&&
			\\
			&&(\partial_2J^{\epsilon} v^{\epsilon,R}_1, \partial_1J^\epsilon v^{\epsilon,R}_2, \partial_2J^\epsilon \theta^{\epsilon,R})
			&\rightarrow (\partial_2v_1,\partial_1v_2,\partial_2\theta) & &\text{in}\quad L^2(0,T;L^2(B_R(x_0))).&&
		\end{align*}
		Thanks to the above convergences, we can apply the Cantor diagonal argument in both $n := \lfloor\epsilon^{-1}\rfloor$\footnote{Here for $x \in \mathbb{R}$, $\lfloor \cdot \rfloor$ denotes the usual floor function.} and $R$ to show that for all $r \in (0,\infty)$ (up to another subsequence) as $\epsilon \to 0$ 
		\begin{align*}
			&&(v^\epsilon,\theta^\epsilon) &\rightarrow (v,\theta) &&\text{in} \quad  L^2(0,T;L^2(B_r(x_0))),&&
			\\
			&&(\partial_2J^{\epsilon} v^\epsilon_1, \partial_1J^\epsilon v^\epsilon_2, \partial_2J^\epsilon \theta^\epsilon)
			&\rightarrow (\partial_2v_1,\partial_1v_2,\partial_2\theta) & &\text{in}\quad L^2(0,T;L^2(B_r(x_0)).&&
		\end{align*}
		By the aid of these strong convergences,
		and noticing that $(\nabla v^\epsilon,\nabla \theta^\epsilon) \rightharpoonup (\nabla v, \nabla \theta )$ in $L^2(0,T;L^2(\mathbb{R}^2))$ for any $T \in (0,\infty)$, one can easily see
		that the nonlinear terms 
		\begin{align*}
			\mathbb{P}(J^\epsilon(J^\epsilon v^\epsilon \cdot \nabla J^\epsilon v^\epsilon)) \text{ and } J^\epsilon(J^\epsilon v^\epsilon \cdot \nabla J^\epsilon \theta^\epsilon) \quad \text{ converge to } \quad \mathbb{P}(v \cdot \nabla v) \text{ and } v \cdot \nabla \theta,
		\end{align*}
		respectively, in $\mathcal{D}'([0,T) \times \mathbb{R}^2)$\footnote{The dual space of $C^\infty_0([0,T) \times \mathbb{R}^2)$.}. Therefore, in $\mathcal{D}'([0,T) \times \mathbb{R}^2)$ as $\epsilon \to 0$, \eqref{B1_nu2_mu1_de2_ODE} converges to 
		\begin{equation*} 
			\frac{d}{dt}\left(
			\begin{matrix}
				v
				\\
				\theta
			\end{matrix}\right) = \left(
			\begin{matrix}
				-\mathbb{P}(v \cdot \nabla v)  + \mathbb{P}(\nu_2\partial_{22} v_1,\mu_1\partial_{11}v_2) + \lambda \mathbb{P}(\theta e_2)
				\\
				- v \cdot \nabla \theta - \lambda v_2 + \delta_2 \partial_{22} \theta
			\end{matrix}\right)
		\end{equation*}
		and it can be checked that the initial data $(v(0),\theta(0)) = (v_0,\theta_0)$. In fact, as in the proof of Theorem \ref{theo_nu2_mu1_de} (Steps 3, 4a and 4b), for $m \geq 3$, by using the properties of $J^\epsilon$, we have $(v^\epsilon,\theta^\epsilon)$ is a Cauchy sequence in $L^2(0,\infty;L^2(\mathbb{R}^2))$, then the above convergence can be understood in a stronger sense, but we skip the details (see also in \cite{Majda_Bertozzi_2002}). In addition, from the first equation, i.e., 
		\begin{equation*}
			\mathbb{P}(\partial_t v + v \cdot \nabla v - \lambda\theta e_2 - (\nu_2\partial_{22} v_1,\mu_1\partial_{11}v_2)) = 0,
		\end{equation*}
		which together with the theorem de Rham implies that there exists a scalar function $\pi$ such that $(v,\theta,\pi)$ satisfies \eqref{B} in the sense of distributions. Using the Lions-Magenes lemma (see \cite[Lemma 1.2, Chapter 3]{Temam_2001}) it can be seen that $(v,\theta) \in C([0,T];H^m(\mathbb{R}^2))$ since $(v,\theta) \in L^2(0,T;H^m(\mathbb{R}^2))$ and $(\partial_t v,\partial_t\theta) \in L^2(0,T;H^{m-1}(\mathbb{R}^2))$ for any $T \in (0,\infty)$. Moreover, $(v,\theta)$ satisfies $\forall t \in (0, \infty)$
		\begin{equation} \label{Hm_nu2_mu1_de2_4} 
			\|v(t)\|^2_{H^m} + \|\theta(t)\|^2_{H^m}
			+ 2\int_0^t \nu_2\|\partial_2v_1\|^2_{H^m} + \mu_1\|\partial_1 v_2\|^2_{H^m} + \delta_2\|\partial_2 \theta\|^2_{H^m} \,ds \leq 2\epsilon^2_0.
		\end{equation}
	
		\textbf{Step 7: $H^1$ uniqueness.} In fact, this step can be done as in the proof of Theorem \ref{theo_nu2_mu1_de}. However, with the help of $\delta_2 > 0$, we can have a simple proof and provide it here. It is enough to show the uniqueness of $H^1$ solutions. Assume that $(v,\theta_1,\pi_1)$ and $(u,\theta_2,\pi_2)$ are two solutions to \eqref{B} with the same initial data. If we denote $V := v-u$, $\Theta := \theta_1-\theta_2$ and $\pi := \pi_1-\pi_2$ then it follows that 
		\begin{equation*} 
			\frac{1}{2} \frac{d}{dt} \left(\|V\|^2_{L^2} + \|\Theta\|^2_{L^2}\right) + \nu_2\|\partial_2 V_1\|^2_{L^2} + \mu_1\|\partial_1 V_2\|^2_{L^2} + \delta_2\|\partial_2 \Theta\|^2_{L^2} = H_1 + H_2,
		\end{equation*}
		where 
		\begin{align*}
			H_1 &:= -\int_{\mathbb{R}^2} V \cdot \nabla v \cdot V \,dx \leq C(\nu_2,\mu_1) \|\nabla v\|^2_{L^2} \|V\|^2_{L^2} +  \frac{\nu_2}{4}\|\partial_2V_1\|^2_{L^2} + \frac{\mu_1}{6}\|\partial_1V_2\|^2_{L^2},
			\\
			H_2 &:= -\int_{\mathbb{R}^2} V \cdot \nabla \theta_1 \Theta \,dx = -\int_{\mathbb{R}^2} V_1 \partial_1\theta_1 \Theta + V_2 \partial_2\theta_1 \Theta \,dx =: H_{21} + H_{22},
			\\
			H_{21} &\leq C \|\partial_1\theta_1\|_{L^2} \|V_1\|^{\frac{1}{2}}_{L^2} \|\partial_1V_1\|^{\frac{1}{2}}_{L^2} \|\Theta\|^{\frac{1}{2}}_{L^2} \|\partial_2\Theta\|^{\frac{1}{2}}_{L^2}
			\\
			&\leq C(\nu_2,\mu_1,\delta_2) \|\partial_1\theta_1\|^2_{L^2}\left(\|V\|^2_{L^2} + \|\Theta\|^2_{L^2}\right) + \frac{\nu_2}{4}\|\partial_2V_1\|^2_{L^2} + \frac{\mu_1}{6}\|\partial_1V_2\|^2_{L^2} + \frac{\delta}{4} \|\partial_2\Theta\|^2_{L^2};
			\\
			H_{22} &\leq C\|\partial_2\theta_1\|_{L^2} \|V_2\|^{\frac{1}{2}}_{L^2} \|\partial_1V_2\|^{\frac{1}{2}}_{L^2} \|\Theta\|^{\frac{1}{2}}_{L^2} \|\partial_2\Theta\|^{\frac{1}{2}}_{L^2}
			\\
			&\leq C(\mu_1,\delta_2) \|\partial_2\theta_1\|^2_{L^2} \left(\|V\|^2_{L^2} + \|\Theta\|^2_{L^2}\right) +  \frac{\mu_1}{6}\|\partial_1V_2\|^2_{L^2} + \frac{\delta}{4} \|\partial_2\Theta\|^2_{L^2}.
		\end{align*}
		Therefore, for $t \geq 0$
		\begin{equation*}
			\frac{1}{2} \frac{d}{dt} Y(t) \leq C(\nu_2,\mu_1,\delta_2) \left(\|\nabla v\|^2_{L^2} + \|\nabla \theta_1\|^2_{L^2}\right) Y(t),
		\end{equation*}
		where 
		\begin{equation*}
			Y(t) := \|V(t)\|^2_{L^2} + \|\Theta(t)\|^2_{L^2} \qquad\text{with}\quad Y(0) = 0.
		\end{equation*}
		Moreover, from the previous steps we know that 
		\begin{equation*}
			\int^T_0 \|\nabla v\|^2_{L^2} + \|\nabla \theta_1\|^2_{L^2} \,dt < \infty \qquad \forall T \in (0,\infty),
		\end{equation*}
		which implies $V = (0,0)$ and $\Theta = 0$ by using Gronwall inequality. 
		
		\textbf{Step 8: $\dot{H}^{m-1}$ large-time behavior.} In this step, we will prove:
		\begin{align*}
			 &\textbf{1)} \quad \int^\infty_0 f(t) \,dt \leq C(m,\lambda,\nu_2,\mu_1,\delta_2)(\epsilon_0^2 + \epsilon^4_0),
			 \\
			 &\textbf{2)} \quad f(t) - f(s) \leq C(m,\lambda,\nu_2,\mu_1,\delta_2)(\epsilon_0^2 + \epsilon^3_0)(t-s) \qquad \text{for} \quad 0 \leq s \leq t < \infty,
		\end{align*}
		where
		\begin{equation*}
			f(t) := \|\partial_2v_1(t)\|^2_{\dot{H}^{m-2}} + \|\partial_1v_2(t)\|^2_{\dot{H}^{m-2}} +  \|\theta(t)\|^2_{\dot{H}^{m-1}} \qquad \forall m \geq 2.
		\end{equation*}
		We then apply \cite[Lemma 2.3]{Lai-Wu-Zhong_2021} to conclude that
		\begin{equation*}
			\|v(t)\|_{\dot{H}^{m-1}} + \|\theta(t)\|_{\dot{H}^{m-1}} \to 0 \quad \text{ as } \quad t \to \infty.
		\end{equation*}
		
		\textbf{8a) The first part.} We first see that \eqref{Hm_nu2_mu1_de2_4} yields for $m \geq 2$
		\begin{equation*}
			\int^\infty_0 \|\partial_2v_1(t)\|^2_{\dot{H}^{m-2}} + \|\partial_1v_2(t)\|^2_{\dot{H}^{m-2}} + \|\partial_2\theta\|^2_{\dot{H}^{m-2}}\,dt 
			\leq  C(m,\nu_2,\mu_1,\delta_2)\epsilon_0^2. 
		\end{equation*}
		We claim that 
		\begin{equation*}
			\int^\infty_0  \|\partial_1\theta\|^2_{\dot{H}^{m-2}}\,dt \leq C(m,\lambda,\nu_2,\mu_1,\delta_2)(\epsilon_0^2 + \epsilon^4_0),
		\end{equation*}
		which and the previous estimate give us the first part. We now prove the claim. It can be seen from \eqref{B} that for $s \in (0,\infty)$
		\begin{align*}
			&\lambda\int^s_0 \|\partial_1 \theta\|^2_{\dot{H}^{m-2}} \,dt = \sum^7_{i=1} R_{i} := \sum_{|\alpha| = m-2} \int^s_0 \frac{d}{dt}\int_{\mathbb{R}^2} D^\alpha \partial_1 v_2 D^\alpha \partial_1\theta\,dxdt 
			\\
			&-\delta_2\sum_{|\alpha| = m-2} \int^s_0 \int_{\mathbb{R}^2} D^\alpha \partial_{122}\theta D^\alpha \partial_1v_2\,dxdt  
			-\lambda\sum_{|\alpha| = m-2} \int^s_0 \int_{\mathbb{R}^2} D^\alpha\partial_{12}(-\Delta)^{-1} \partial_2\theta D^\alpha \partial_1\theta \,dxdt 
			\\
			& + \lambda \int^s_0 \|\partial_1v_2\|^2_{\dot{H}^{m-2}} \,dt - \sum_{|\alpha| = m-2} \int^s_0 \int_{\mathbb{R}^2} D^\alpha \partial_1\langle \mathbb{P}(\nu_2\partial_{22}v_1,\mu_1\partial_{11}v_2),e_2\rangle D^\alpha \partial_1\theta\,dxdt 
			\\
			&+ \sum_{|\alpha| = m-2} \int^s_0 \int_{\mathbb{R}^2} D^\alpha \partial_1\langle \mathbb{P}(v \cdot \nabla v),e_2\rangle D^\alpha \partial_1\theta\,dxdt
			+ \sum_{|\alpha| = m-2} \int^s_0 \int_{\mathbb{R}^2} D^\alpha \partial_1(v \cdot \nabla \theta) D^\alpha \partial_1v_2\,dxdt.  
		\end{align*}
		We first find that for some $\epsilon \in (0,1)$
		\begin{align*}
			R_1 &\leq  \|\partial_1v_2(s)\|_{\dot{H}^{m-2}} \|\partial_1\theta(s)\|_{\dot{H}^{m-2}} + \|\partial_1v_2(0)\|_{\dot{H}^{m-2}} \|\partial_1\theta(0)\|_{\dot{H}^{m-2}};
			\\
			R_2,R_3,R_4 &\leq C(\lambda,\delta_2) \int^s_0 \|\partial_1v_2\|^2_{\dot{H}^{m-1}} + \|\partial_1v_2\|^2_{\dot{H}^{m-2}} +  \|\partial_2\theta\|^2_{\dot{H}^{m-1}} + \|\partial_2\theta\|^2_{\dot{H}^{m-2}} \,dt;
			\\
			R_5 &= - \sum_{|\alpha| = m-2} \int^s_0 \int_{\mathbb{R}^2} D^\alpha (\mu_1 \partial_{111}v_2 + \partial_{12}(-\Delta)^{-1}(\nu_2\partial_{122}v_1 + \mu_1\partial_{211}v_2)) D^\alpha \partial_1\theta \,dxdt
			\\
			&\leq C(\lambda,\nu_2,\mu_1,\epsilon) \int^s_0 \|\partial_2v_1\|^2_{\dot{H}^{m-1}} +  \|\partial_1v_2\|^2_{\dot{H}^m} + \|\partial_1v_2\|^2_{\dot{H}^{m-1}}\,dt + \epsilon \lambda \int^s_0 \|\partial_1\theta\|^2_{\dot{H}^{m-2}} \,dt,
		\end{align*}
		where we used the standard double Riesz transform (see \cite{Stein1970})  
		\begin{equation*}
			\|\partial_{ij}(-\Delta)^{-1} f\|_{L^p(\mathbb{R}^2)} \leq C_p \|f\|_{L^p(\mathbb{R}^2)} \qquad \text{for} \quad i,j \in \{1,2\}, p \in (1,\infty).
		\end{equation*}
		We continue with $R_6$ as follows
		\begin{equation*}
			R_6 = \sum_{|\alpha| = m-2} \int^s_0 \int_{\mathbb{R}^2} D^\alpha [\partial_1(v \cdot \nabla v_2) + \partial_{12}(-\Delta)^{-1}(\partial_1(v \cdot \nabla v_1) + \partial_2(v \cdot \nabla v_2))] D^\alpha \partial_1\theta\,dxdt = R_{61} + R_{62}, 
		\end{equation*}
		where for $m = 2$
		\begin{align*}
			R_{61} 
			= \int^s_0 \int_{\mathbb{R}^2} \partial_1(v \cdot \nabla v_2) \partial_1\theta \,dxdt 
			\leq C(\lambda,\epsilon) \int^s_0 \|v\|^2_{H^2}\|\partial_1v_2\|^2_{\dot{H}^1} \,dt + \epsilon\lambda \int^s_0 \|\partial_1\theta\|^2_{L^2} \,dt 
		\end{align*}
		and for $m \geq 3$, we rewrite $R_{61} = R_{611} + R_{612}$ and obtain
		\begin{align*}
			R_{611} &:= \sum_{|\alpha| = m-2} \int^s_0 \int_{\mathbb{R}^2} D^\alpha (\partial_1v \cdot \nabla v_2) D^\alpha \partial_1\theta\,dxdt 
			\\
			&\leq C(m) \int^s_0 \|D^{m-1}v\|_{L^2}\|\nabla v\|_{L^\infty}  \|D^{m-2}\partial_1\theta\|_{L^2} \,dt
			\\
			&\leq C(\lambda,\epsilon) \int^s_0 \|v\|_{H^m} \left(\|\partial_2v_1\|^2_{\dot{H}^{m-2}} +  \|\partial_1v_2\|^2_{\dot{H}^{m-2}}\right)\,dt + \epsilon \lambda \int^s_0 \|\partial_1\theta\|^2_{\dot{H}^{m-2}} \,dt;
			\\
			R_{612} &:= \sum_{|\alpha| = m-2} \int^s_0 \int_{\mathbb{R}^2} D^\alpha (v \cdot \nabla \partial_1v_2) D^\alpha \partial_1\theta\,dxdt
			\\
			&\leq C(m) \int^s_0 \left(\|D^{m-2}v\|_{L^2}\|\nabla\partial_1v_2 \|_{L^\infty} + \|v\|_{L^\infty} \|D^{m-1} \partial_1v_2\|_{L^2}\right) \|D^{m-2}\partial_1\theta\|_{L^2} \,dt
			\\
			&\leq C(\lambda,\epsilon) \int^s_0 \|v\|^2_{H^m}   \|\partial_1v_2\|^2_{H^m}\,dt + \epsilon \lambda \int^s_0 \|\partial_1\theta\|^2_{\dot{H}^{m-2}} \,dt.
		\end{align*}
		We then estimate $R_{62}$ by
		\begin{equation*}
			R_{62} = -\sum_{|\alpha| = m-2} \int^s_0 \int_{\mathbb{R}^2} D^\alpha [\partial_{11}(-\Delta)^{-1}(v \cdot \nabla v_1) + \partial_{12}(-\Delta)^{-1}(v \cdot \nabla v_2))] D^\alpha \partial_{12}\theta\,dxdt,
		\end{equation*}
		where for $m = 2$
		\begin{align*}
			R_{62} \leq C \int^s_0 \|v\|_{H^2} \left(\|\partial_2v_1\|^2_{L^2} + \|\partial_1v_2\|^2_{L^2} + \|\partial_2\theta\|^2_{\dot{H}^1}\right) \,dt
		\end{align*}
		and for $m \geq 3$
		\begin{align*}
			R_{62} &\leq C(m) \int^s_0 \left(\|D^{m-2}v\|_{L^2}\|\nabla v\|_{L^\infty} + \|v\|_{L^\infty} \|D^{m-1} v\|_{L^2}\right) \|D^{m-1}\partial_2\theta\|_{L^2} \,dt
			\\
			&\leq C(m) \int^s_0 \|v\|_{H^m} \left(\|\partial_2v_1\|^2_{H^{m-2}} + \|\partial_1v_2\|^2_{H^{m-2}} + \|\partial_2\theta\|^2_{\dot{H}^{m-1}}\right) \,dt.
		\end{align*}
		It remains to bound $R_7$. Similar to $R_6$, we find that for $m = 2$
		\begin{align*}
			R_7 &\leq \int^s_0 \|v\|_{L^\infty}\|\nabla\theta\|_{L^2}\|\partial_{11}v_2\|_{L^2}\,dt 
			\\
			&\leq C(\lambda,\epsilon) \int^s_0 \left(\|v\|_{H^2} + \|v\|^2_{H^2}\right)\left(\|\partial_1v_2\|^2_{\dot{H}^1} + \|\partial_2\theta\|^2_{L^2}\right) \,dt + \epsilon\lambda \int^s_0 \|\partial_1\theta\|^2_{L^2} \,dt
		\end{align*}
		and for $m \geq 3$
		\begin{align*}
			R_7 &\leq C(m)\int^s_0 \left(\|v\|_{H^m} + \|\theta\|_{H^m}\right)\left(\|\partial_2v_1\|^2_{H^{m-1}} + \|\partial_1v_2\|^2_{H^{m-1}} + \|\partial_2\theta\|^2_{\dot{H}^{m-2}}\right) \,dt \\
			&\quad + C(\lambda,\epsilon)\int^s_0 \|v\|^2_{H^m} \|\partial_1v_2\|^2_{\dot{H}^{m-1}}\,dt  + \epsilon\lambda \int^s_0 \|\partial_1\theta\|^2_{\dot{H}^{m-2}}\,dt.
		\end{align*}
		Therefore, by choosing $\epsilon = \frac{1}{10}$ and using \eqref{Hm_nu2_mu1_de2_4}, the claim follows.
		
		\textbf{8b) The second part.} Similar to the previous part, we have
		\begin{align*}
			&\frac{1}{2}\frac{d}{dt}\left(\|D^{m-2}\partial_2v_1\|^2_{L^2} + \|D^{m-2}\partial_1v_2\|^2_{L^2} + \|D^{m-1}\theta\|^2_{L^2}\right) = \sum^9_{i=1} G_i :=
			\\
			&\quad\delta_2\sum_{|\alpha| = m-2}  \int_{\mathbb{R}^2} D^\alpha \nabla \partial_{22}\theta \cdot D^\alpha \nabla \theta\,dx + \lambda \sum_{|\alpha| = m-2}  \int_{\mathbb{R}^2} D^\alpha \partial_1\langle\mathbb{P}(\theta e_2),e_2\rangle D^\alpha \partial_1v_2\,dx
			\\
			&+ \lambda \sum_{|\alpha| = m-2}  \int_{\mathbb{R}^2} D^\alpha \partial_2\langle\mathbb{P}(\theta e_2),e_1\rangle D^\alpha \partial_2v_1\,dx - \lambda\sum_{|\alpha| = m-2}  \int_{\mathbb{R}^2} D^\alpha \nabla v_2 \cdot D^\alpha \nabla \theta\,dx
			\\
			&+ \sum_{|\alpha| = m-2}  \int_{\mathbb{R}^2} D^\alpha \partial_1\langle \mathbb{P}(\nu_2\partial_{22}v_1,\mu_1\partial_{11}v_2),e_2\rangle D^\alpha \partial_1v_2\,dx 
			\\
			&+ \sum_{|\alpha| = m-2}  \int_{\mathbb{R}^2} D^\alpha \partial_2\langle \mathbb{P}(\nu_2\partial_{22}v_1,\mu_1\partial_{11}v_2),e_1\rangle D^\alpha \partial_2v_1\,dx - \sum_{|\alpha| = m-2}  \int_{\mathbb{R}^2} D^\alpha \nabla (v \cdot \nabla \theta) \cdot D^\alpha \nabla \theta\,dx
			\\
			& - \sum_{|\alpha| = m-2}  \int_{\mathbb{R}^2} D^\alpha \partial_2\langle \mathbb{P}(v \cdot \nabla v),e_1\rangle D^\alpha \partial_2v_1\,dx
			- \sum_{|\alpha| = m-2}  \int_{\mathbb{R}^2} D^\alpha \partial_1\langle \mathbb{P}(v \cdot \nabla v),e_2\rangle D^\alpha \partial_1v_2\,dx.
		\end{align*}
		For the linear terms, it can be easily seen that for $m \geq 2$
		\begin{equation*}
			\sum^6_{i=1}G_i \leq C(\lambda,\nu_2,\mu_1)\left(\|v(t)\|^2_{H^m} + \|\theta(t)\|^2_{H^m}\right).
		\end{equation*}
		The term $G_7$ is bounded by using the divergence-free condition as follows
		\begin{align*}
			G_7 (m = 2) &= -\int_{\mathbb{R}^2} \nabla v \cdot \nabla \theta \cdot \nabla \theta\,dx \leq C\left(\|v(t)\|^2_{H^2} + \|\theta(t)\|^2_{H^2}\right);
			\\
			G_7 (m \geq 3) &= -\sum_{|\alpha| = m-2} \int_{\mathbb{R}^2} [D^{\alpha + 1} (v \cdot \nabla \theta) - v \cdot \nabla D^{\alpha+1}\theta ]\cdot D^{\alpha+1} \theta\,dx
			\leq C(m)\|v(t)\|_{H^m} \|\theta(t)\|^2_{H^m}.
		\end{align*}
		It remains to estimate $G_8$ and $G_9$. These terms can be bounded in the same way. We only provide the estimate of $G_8$ by writing $G_8 = G_{81} + G_{82}$, where 
		\begin{align*}
			G_{81} (m = 2) &:= \int_{\mathbb{R}^2} v \cdot \nabla v_1 \partial_{22}v_1\,dx \leq C \|v\|^3_{H^2};
			\\
			G_{81} (m \geq 3) &:=  \sum_{|\alpha| = m-2}  \int_{\mathbb{R}^2} D^\alpha (v \cdot \nabla v_1) D^\alpha \partial_{22}v_1\,dx \leq C(m) \|v\|^3_{H^m},
		\end{align*}
		and similarly
		\begin{align*}
			G_{82} &:=  \sum_{|\alpha| = m-2}  \int_{\mathbb{R}^2} D^\alpha[ \partial_{12}(-\Delta)^{-1}(v \cdot \nabla v_1) + \partial_{22}(-\Delta)^{-1}(v \cdot \nabla v_2)] D^\alpha \partial_{12}v_1\,dx
			\leq C(m) \|v\|^3_{H^m}.
		\end{align*}
		Integrating in time and using \eqref{Hm_nu2_mu1_de2_4} lead to the second part and finishes the proof.
	\end{proof}
	
	We now give a proof for the second case.
	
	\begin{proof}[Proof of Theorem \ref{theo_nu2_mu1_de2}: The case $\delta_1 > 0$ and $\delta_2 = 0$.]
		The proof in this case uses the same idea as in the previous one. To avoid repeating the calculations, we will mention what should be changed. Following the proof above, we only need to replace $\delta_2$ by $\delta_1$ and exchange the role of $J\partial_2\theta$ and $J\partial_1\theta$ to each other with using the divergence-free condition of $v$ and $Jv$ as well.
	\end{proof}

	%
	\section{Appendix} \label{sec:app}
	%
	
	We will provide the detailed proofs of Steps 2d, 2e, 2f, 2g, \eqref{Lp} in the proof of Theorem \ref{theo_nu2_mu1_de} and the point number 4 in Remark \ref{rm1} as follows.
	
	\subsection{Appendix A: Proof of \eqref{Hs_estimate} for $s > 1$}
	
	We start by giving the proof of the case $s \in (1,2)$ and $\delta \in [s,s+1]$.
	
	\begin{proof}[Proof of Step 2d: The case $s \in (1,2)$ and $s \leq \delta \leq s+1$] We will consider the cases $\delta = s$, $\delta = s+1$ and $\delta \in (s,s+1)$, respectively, as follows.
		
		$\bullet$ If $\delta = s$ then we find that 
		\begin{equation*} 
			\frac{1}{2}\frac{d}{dt}\|v^n\|^2_{\dot{H}^s}  + \min\{\nu_2,\mu_1\}\|v^n\|^2_{\dot{H}^{s+1}} \leq I_5 + I_6,
		\end{equation*}
		where for some $\epsilon \in (0,1)$
		\begin{align*}
			I_5 &:= \lambda_1\int_{\mathbb{R}^2} \Lambda^s (\mathbb{P}(\theta^n e_2)) \cdot \Lambda^s v^n \,dx 
			\\
			&\leq C(T^n_*,\nu_2,\mu_1,\lambda_i,v_0,\theta_0,\epsilon) \left(\|\theta^n\|^2_{\dot{H}^s} + 1\right) + \epsilon \min\{\nu_2,\mu_1\} \|v^n\|^2_{\dot{H}^{s+1}};
			\\
			I_6 &:= -\int_{\mathbb{R}^2} \Lambda^{s-1}(v^n \cdot \nabla v^n) \cdot \Lambda^{s+1} v^n \,dx.
		\end{align*}
		Since $s - 1 \in  (0,1)$, $I_6$ can be bounded in the same way as $I_2$. 
		
		$\bullet$ If $\delta = s+1$ then we define $\sigma' := \delta - 2 = s - 1 \in (0,1)$ and find that 
		\begin{align*}
			\frac{1}{2}\frac{d}{dt} \|v^n\|^2_{\dot{H}^\delta} + \min\{\nu_2,\mu_1\}\|v^n\|^2_{\dot{H}^{\delta+1}} \leq I_7 + I_8,
		\end{align*}
		where for some $\epsilon \in (0,1)$
		\begin{align*}
			I_7 &:= \sum^2_{i=1} \lambda_1\int_{\mathbb{R}^2} \Lambda^s \partial_i \mathbb{P}(\theta^ne_2) \cdot \Lambda^s \partial_i v^n\,dx 
			\leq C(\nu_2,\mu_1,\lambda_1,\epsilon) \|\theta^n\|^2_{\dot{H}^s}   + \epsilon\min\{\nu_2,\mu_1\}\|v^n\|^2_{\dot{H}^{\delta+1}};
			\\
			I_8 &:= - \sum^2_{i=1} \int_{\mathbb{R}^2} \Lambda^{\sigma'}  (\partial_iv^n \cdot \nabla v^n + v^n \cdot \nabla \partial_iv^n) \cdot \Lambda^\delta \partial_iv^n \,dx =: \sum^2_{i=1} I_{81i} + I_{82i},	
			\\
			I_{81i} &\leq C(s) \|\Lambda^{\delta+1}v^n\|_{L^2} \times
			\begin{cases}
				\|\Lambda^{\sigma'} \partial_iv^n\|_{L^4}\|\nabla v^n\|_{L^4} + \|\partial_iv^n\|_{L^\frac{4}{1-2\sigma'}}\|\Lambda^{\sigma'} \nabla v^n\|_{L^\frac{4}{1+2\sigma'}} \quad &\text{if } \sigma' \in (0,\frac{1}{2}),
				\\
				\|\Lambda^{\sigma'} \partial_iv^n\|_{L^4}\|\nabla v^n\|_{L^4} + \|\partial_iv^n\|_{L^6}\|\Lambda^{\sigma'} \nabla v^n\|_{L^3} \quad &\text{if } \sigma' \in [\frac{1}{2},1),
			\end{cases}
			\\
			I_{82i} &\leq C(s)\|\Lambda^{\delta+1}v^n\|_{L^2} \times
			\begin{cases}
				\|\Lambda^{\sigma'} v^n\|_{L^4}\|\nabla \partial_iv^n\|_{L^4} + \|v^n\|_{L^\frac{4}{1-2\sigma'}}\|\Lambda^{\sigma'} \nabla \partial_iv^n\|_{L^\frac{4}{1+2\sigma'}} \quad &\text{if }  \sigma' \in (0,\frac{1}{2}),
				\\
				\|\Lambda^{\sigma'} v^n\|_{L^4}\|\nabla \partial_iv^n\|_{L^4} + \|v^n\|_{L^6}\|\Lambda^{\sigma'} \nabla \partial_iv^n\|_{L^3} \quad &\text{if } \sigma' \in [\frac{1}{2},1).
			\end{cases}
		\end{align*}
		In addition,
		\begin{align*}
			\|\Lambda^{\sigma'} \partial_i v^n\|_{L^4},\|\partial_iv^n\|_{L^\frac{4}{1-2\sigma'}} &\leq C(s)\|\nabla v^n\|^\frac{3}{2\delta}_{L^2} \|\Lambda^{\delta+1}v^n\|^\frac{2\delta-3}{2\delta}_{L^2},
			\\
			\|\nabla v^n\|_{L^4},\|\Lambda^{\sigma'}\nabla v^n\|_{L^\frac{4}{1+2\sigma'}} &\leq C(s)\|\nabla v^n\|^\frac{2\delta-1}{2\delta}_{L^2} \|\Lambda^{\delta+1} v^n\|^\frac{1}{2\delta}_{L^2},
			\\
			\|\partial_i v^n\|_{L^6} &\leq C(s)\|\nabla v^n\|^\frac{3\delta-2}{3\delta}_{L^2} \|\Lambda^{\delta+1}v^n\|^\frac{2}{3\delta}_{L^2},
			\\
			\|\Lambda^{\sigma'}\nabla v^n\|_{L^3} &\leq C(s)\|\nabla v^n\|^\frac{5}{3\delta}_{L^2} \|\Lambda^{\delta+1} v^n\|^\frac{3\delta-5}{3\delta}_{L^2},
			\\
			\|\Lambda^{\sigma'}v^n\|_{L^4},\|v^n\|_{L^\frac{4}{1-2\sigma'}} &\leq C(s)\|v^n\|^\frac{5}{2(\delta+1)}_{L^2} \|\Lambda^{s+1} v^n\|^\frac{2\delta-3}{2(\delta+1)}_{L^2},
			\\
			\|\nabla \partial_i v^n\|_{L^4},\|\Lambda^{\sigma'}\nabla \partial_i v^n\|_{L^\frac{4}{1+2\sigma'}} &\leq C(s)\|\nabla v^n\|^\frac{2\delta-3}{2\delta}_{L^2} \|\Lambda^{\delta+1} v^n\|^\frac{3}{2\delta}_{L^2},
			\\
			\|\Lambda^{\sigma'}\nabla \partial_i v^n\|_{L^3} &\leq C(s)\|\nabla v^n\|^\frac{2}{3\delta}_{L^2} \|\Lambda^{\delta+1} v^n\|^\frac{3\delta-2}{3\delta}_{L^2}.
		\end{align*}
		Thus, $I_8$ can be estimated similarly as $I_2$.
	
		$\bullet$ If $\delta \in (s,s+1)$ then we obtain
		\begin{align*}
			\frac{1}{2}\frac{d}{dt} \|v^n\|^2_{\dot{H}^\delta} + \min\{\nu_2,\mu_1\}\|v^n\|^2_{\dot{H}^{\delta+1}} \leq I_9 + I_{10},
		\end{align*}
		where for some $\epsilon \in (0,1)$, since $2\delta - s -1 \leq \delta + 1$
		\begin{align*}
			I_9 &:= \sum^2_{i=1} \lambda_1\int_{\mathbb{R}^2} \Lambda^{s-1} \partial_i \mathbb{P}(\theta^ne_2) \cdot \Lambda^{2\delta-s-1} \partial_i v^n\,dx
			\\
			&\leq C(T^n_*,\nu_2,\mu_1,\lambda_i,v_0,\theta_0,\epsilon) \left(\|\theta^n\|^2_{\dot{H}^s} +1\right) + \epsilon\min\{\nu_2,\mu_1\}\|v^n\|^2_{\dot{H}^{\delta+1}};
			\\
			I_{10} &:= - \sum^2_{i=1} \int_{\mathbb{R}^2} \Lambda^{\delta-1}  (\partial_iv^n \cdot \nabla v^n + v^n \cdot \nabla \partial_iv^n) \cdot \Lambda^{\delta-1} \partial_iv^n \,dx.
		\end{align*}
		Since $\delta \in (s,s+1)$, it implies that $\delta - 1 \in (s-1,s) \in (0,2)$. If $\delta-1 \in (1,2)$ then $I_{10}$ is exactly bounded as $I_8$. If $\delta - 1 \in (0,1)$ then we define $\sigma'' := \delta-1 \in (0,1)$ and find that   
		\begin{align*}
			I_{10} &=: -\sum^2_{i=1} I_{101i} + I_{102i},	
			\\
			I_{101i} &\leq C(\delta) \|\Lambda^{\delta}v^n\|_{L^2} \times
			\begin{cases}
				\|\Lambda^{\sigma''} \partial_iv^n\|_{L^4}\|\nabla v^n\|_{L^4} + \|\partial_iv^n\|_{L^\frac{4}{1-2\sigma''}}\|\Lambda^{\sigma''} \nabla v^n\|_{L^\frac{4}{1+2\sigma''}} \quad &\text{if } \sigma'' \in (0,\frac{1}{2}),
				\\
				\|\Lambda^{\sigma''} \partial_iv^n\|_{L^4}\|\nabla v^n\|_{L^4} + \|\partial_iv^n\|_{L^6}\|\Lambda^{\sigma''} \nabla v^n\|_{L^3} \quad &\text{if } \sigma'' \in [\frac{1}{2},1),
			\end{cases}
			\\
			I_{102i} &\leq C(\delta)\|\Lambda^{\delta}v^n\|_{L^2} 
				\left(\|\Lambda^{\sigma''} v^n\|_{L^\frac{2}{\delta-1}}\|\nabla \partial_iv^n\|_{L^\frac{2}{2-\delta}} + \|v^n\|_{L^\infty}\|\Lambda^{\sigma''} \nabla \partial_iv^n\|_{L^2}\right).
		\end{align*}
		Similarly,
		\begin{align*}
			\|\Lambda^{\sigma''} \partial_i v^n\|_{L^4},\|\partial_iv^n\|_{L^\frac{4}{1-2\sigma''}} &\leq C(\delta)\|\nabla v^n\|^\frac{1}{2\delta}_{L^2} \|\Lambda^{\delta+1}v^n\|^\frac{2\delta-1}{2\delta}_{L^2},
			\\
			\|\nabla v^n\|_{L^4},\|\Lambda^{\sigma''}\nabla v^n\|_{L^\frac{4}{1+2\sigma''}} &\leq C(\delta)\|\nabla v^n\|^\frac{2\delta-1}{2\delta}_{L^2} \|\Lambda^{\delta+1} v^n\|^\frac{1}{2\delta}_{L^2},
			\\
			\|\partial_i v^n\|_{L^6} &\leq C(\delta)\|\nabla v^n\|^\frac{3\delta-2}{3\delta}_{L^2} \|\Lambda^{\delta+1}v^n\|^\frac{2}{3\delta}_{L^2},
			\\
			\|\Lambda^{\sigma''}\nabla v^n\|_{L^3} &\leq C(\delta)\|\nabla v^n\|^\frac{2}{3\delta}_{L^2} \|\Lambda^{\delta+1} v^n\|^\frac{3\delta-2}{3\delta}_{L^2},
			\\
			\|\Lambda^{\sigma''}v^n\|_{L^\frac{2}{\delta-1}} &\leq C(\delta)\|\nabla v^n\|_{L^2},
			\\
			\|\nabla \partial_i v^n\|_{L^\frac{2}{2-\delta}} &\leq C(\delta) \|\Lambda^{\delta+1} v^n\|_{L^2},
			\\
			\|\Lambda^\delta v^n\|_{L^2} &\leq C(\delta) \|v^n\|^\frac{1}{\delta+1}_{L^2}\|\Lambda^{\delta+1}v^n\|^\frac{\delta}{\delta+1}_{L^2}
			\\
			\|v^n\|_{L^\infty}\|\Lambda^{\sigma''} \nabla \partial_iv^n\|_{L^2} &\leq C(\delta) \left(\|\Lambda^\delta v^n\|_{L^2} + \|v^n\|_{L^2}\right)\|\Lambda^{\delta+1}v^n\|_{L^2}.
		\end{align*}
		If $\delta - 1 = 1$ or $\delta = 2$ then
		\begin{align*}
			I_{10} &= - \sum^2_{i=1} \int_{\mathbb{R}^2} (\partial_iv^n \cdot \nabla v^n + v^n \cdot \nabla \partial_iv^n) \cdot \Lambda^2 \partial_iv^n \,dx
			\\
			&\leq C\left(\|\nabla v^n\|_{L^2} \|\Lambda^2v^n\|_{L^2} + \|v^n\|^\frac{1}{2}\|\nabla v^n\|^\frac{1}{2}_{L^2}\|\Lambda^\frac{3}{2}\nabla v^n\|_{L^2}\right) \|\Lambda^3 v^n\|_{L^2}
			\\
			&\leq C\left(\|\nabla v^n\|_{L^2} \|v^n\|^\frac{1}{3}_{L^2}\|\Lambda^3v^n\|^\frac{2}{3}_{L^2} + \|v^n\|^\frac{1}{2}\|\nabla v^n\|_{L^2}\|\Lambda^3 v^n\|^\frac{1}{2}_{L^2}\right) \|\Lambda^3 v^n\|_{L^2}
			\\
			&\leq C(T^n_*,\nu_2,\mu_1,\lambda_i,v_0,\theta_0,\epsilon) + 2\epsilon \min\{\nu_2,\mu_1\}\|v^n\|^2_{\dot{H}^3}.
		\end{align*}
		Choosing $\epsilon = \frac{1}{6}$ in the three cases  yields  
		\begin{align*}
			\frac{d}{dt} \|v^n\|^2_{H^\delta} + \min\{\nu_2,\mu_1\}\|v^n\|^2_{H^{\delta+1}} \leq C(T^n_*,\delta,\nu_2,\mu_1,\lambda_i,v_0,\theta_0)\left(\|\theta^n\|^2_{\dot{H}^s} + 1\right).
		\end{align*}
		Moreover, 
		\begin{equation*}
			\frac{1}{2} \frac{d}{dt}\|\theta^n\|^2_{\dot{H}^s} = I_{11} + I_{12} + I_{13}, 
		\end{equation*}
		where since $s-1 \in (0,1)$ 
		\begin{align*}
			I_{11} &:= - \sum^2_{i=1} \lambda_2 \int_{\mathbb{R}^2} \Lambda^{s-1}\partial_iv^n_2 \Lambda^{s-1} \partial_i\theta^n\,dx 
			\leq |\lambda_2| \|v^n\|_{H^2} \|\theta^n\|_{\dot{H}^{s}};
			\\
			I_{12} &:= -\sum^2_{i=1} \int_{\mathbb{R}^2} [\Lambda^{s-1} 	(v^n \cdot \nabla \partial_i\theta^n) - v^n \cdot \nabla \Lambda^{\sigma'}\partial_i\theta^n] \Lambda^{s-1} \partial_i\theta^n \,dx 
			\\
			&\leq C(s)\left(\|\Lambda^s  v^n\|_{L^\frac{2}{s-1}}  \|\nabla \theta^n\|_{L^\frac{2}{2-s}} + \|\nabla v^n\|_{L^\infty} \|\Lambda^{s} \theta^n\|_{L^2} \right)\|\Lambda^{s}\theta^n\|_{L^2}
			\\
			&\leq C(s)\left(\|v^n\|_{H^2} \|\theta^n\|_{\dot{H}^s} + \|\nabla v^n\|_{L^\infty} \|\Lambda^{s} \theta^n\|_{L^2} \right)\|\Lambda^{s}\theta^n\|_{L^2};
			\\
			I_{13} &:= -\int_{\mathbb{R}^2} \Lambda^{s-1}	(\partial_i v^n \cdot \nabla \theta^n) \Lambda^{s-1} \partial_i\theta^n \,dx 
			\\
			&\leq C(s)\left(\|\Lambda^s  v^n\|_{L^\frac{2}{s-1}}  \|\nabla \theta^n\|_{L^\frac{2}{2-s}} + \|\nabla v^n\|_{L^\infty} \|\Lambda^{s} \theta^n\|_{L^2} \right)\|\Lambda^{s}\theta^n\|_{L^2}.
		\end{align*}
		In the estimate of $\|\nabla v^n\|_{L^\infty}$, by using the Brezis-Gallouet type inequality, we can choose $2 + s_0 := \delta + 1$ for $s_0 \in (0,1]$  since $ 1 < s \leq s_0 + 1 = \delta \leq s+1 < 2$ or $ 0 < s_0 < 1$. Therefore,  		
		\begin{equation*}
			\frac{d}{dt}\left(\|\theta^n\|^2_{\dot{H}^{s}} + 1\right) \leq C(T^n_*,s,\nu_2,\mu_1,\lambda_i,v_0,\theta_0)\|v^n\|_{H^2}  \left(\|\theta^n\|^2_{\dot{H}^{s}} + 1\right) \left(1 + \log \frac{\|v^n\|^2_{H^{\delta+1}}}{\|v^n\|^2_{H^2}}\right)^\frac{1}{2}.
		\end{equation*}
		Thus, \eqref{Hs_estimate} follows for $s \in (1,2)$ and $\delta \in [s,s+1]$.
	\end{proof}
	
	We continue with the case $s = 2$ and $\delta \in [2,3]$.
	
	\begin{proof}[Proof of Step 2e: The case $s = 2$ and $2 \leq \delta \leq 3$] Similar to the previous case, we separate to the cases $\delta = 2$, $\delta = 3$ and $\delta \in (2,3)$, respectively, as follows.
		
		$\bullet$ If $\delta = 2$ then we find that 
		\begin{equation*}
			\frac{1}{2}\frac{d}{dt}\|v^n\|^2_{\dot{H}^2} + \min\{\nu_2,\mu_1\}\|v^n\|^2_{\dot{H}^3} \leq I_{14} + I_{15},
		\end{equation*}
		where for some $\epsilon \in (0,1)$\footnote{For two matrices $A = (a_{ij})$ and $B = (b_{ij})$, $A:B := \sum_{1 \leq i,j \leq 2} a_{ij}b_{ij}$.}
		\begin{align*}
			I_{14} &:= \lambda_1\int_{\mathbb{R}^2} \Delta \mathbb{P}(\theta e_2) \cdot \Delta v^n \,dx
			\\
			&\leq C(T^n_*,\nu_2,\mu_1,\lambda_i,v_0,\theta_0)\left(\|\theta^n\|^2_{\dot{H}^2} +1\right) + \epsilon \min\{\nu_2,\mu_1\}\|v^n\|_{\dot{H}^3};
			\\
			I_{15} &:=  \int_{\mathbb{R}^2}  [\nabla (v^n \cdot \nabla) v^n + (v^n \cdot \nabla) \nabla v^n] : \nabla \Delta v^n \,dx  
			\\
			&\leq C\left(\|\nabla v^n\|_{L^2}\|D^2 v^n\|_{L^2} + \|v^n\|^\frac{1}{2}_{L^2}\|\nabla v^n\|^\frac{1}{2}_{L^2} \|D^2 v^n\|^\frac{1}{2}_{L^2}\|v^n\|^\frac{1}{2}_{\dot{H}^3}\right)\|v^n\|_{\dot{H}^3}
			\\
			&\leq C(T^n_*,\nu_2,\mu_1,\lambda_i,v_0,\theta_0)\left(\|v^n\|^\frac{2}{3}_{\dot{H}^3} + \|v^n\|^\frac{5}{6}_{\dot{H}^3} \right)\|v^n\|_{\dot{H}^3}
			\\
			&\leq C(T^n_*,\nu_2,\mu_1,\lambda_i,v_0,\theta_0,\epsilon) + 2\epsilon \min\{\nu_2,\mu_1\} \|v^n\|^2_{\dot{H}^3}, 
		\end{align*}
		here we used the fact that $\|v^n\|_{\dot{H}^m} \leq C(m)\|w^n\|_{\dot{H}^{m-1}}$ for $m \in \mathbb{N}$, and the 2D Gagliardo-Nirenberg inequality (see \cite{Nirenberg_1959}) for $1 \leq p,q,r \leq \infty$, $j,m \in \mathbb{N}_0 := \mathbb{N} \cup \{0\}$, $0 \leq j < m$ 
		\begin{equation*} 
			\|D^j f\|_{L^p} \leq C(j,m,p,r,q,a)\|D^m f\|^a_{L^r} \|f\|^{1-a}_{L^q}, \qquad D^m := \sum_{|\alpha| = m} D^\alpha,
		\end{equation*}
		with $D^\alpha := \partial_1^{\alpha_1}\partial_2^{\alpha_2}$, $|\alpha| := \alpha_1 + \alpha_2$, $\alpha_1,\alpha_2 \in \mathbb{N} \cup \{0\}$ and 
		\begin{equation*}
			\frac{1}{p} = \frac{j}{2} + a\left(\frac{1}{r} - \frac{m}{2}\right) + \frac{1-a}{q} \qquad \text{with} \quad 
			\begin{cases}
				\frac{j}{m} \leq a < 1 &\text{if} \quad  m - j - \frac{2}{r} \in \mathbb{N}_0,
				\\
				\frac{j}{m} \leq a \leq 1 &\text{otherwise}.
			\end{cases}
		\end{equation*}
		
		$\bullet$ If $\delta = 3$ then we obtain
		\begin{equation*}
			\frac{1}{2}\frac{d}{dt}\|v^n\|^2_{\dot{H}^3} + \min\{\nu_2,\mu_1\}\|v^n\|^2_{\dot{H}^4} \leq I_{16} + I_{17},
		\end{equation*}
		where for some $\epsilon \in (0,1)$, similar to the case $\delta = 2$
		\begin{align*}
			I_{16} &:= \lambda_1\int_{\mathbb{R}^2} D^3(\mathbb{P}\theta^n e_2) \cdot D^3v^n\,dx  \leq C(\nu_2,\mu_1,\lambda_1,\epsilon)|\|\theta^n\|^2_{\dot{H}^2} + \epsilon \min\{\nu_2,\mu_1\}\|v^n\|_{\dot{H}^4};
			\\
			I_{17} &:= -\int_{\mathbb{R}^2} D^3(\mathbb{P} (T_n(v^n \cdot \nabla v^n)) \cdot D^3v^n \,dx
			\\
			&\leq C\left(\|\nabla v^n\|^\frac{1}{2}_{L^2}\|D^2v^n\|_{L^2}\|D^3v^n\|^\frac{1}{2}_{L^2} + \|v^n\|^\frac{1}{2}_{L^2}\|\nabla v^n\|^\frac{1}{2}_{L^2}\|D^3v^n\|^\frac{1}{2}_{L^2}\|D^4v^n\|^\frac{1}{2}_{L^2}\right)\|D^4v^n\|_{L^2}
			\\
			&\leq C\left(\|\nabla v^n\|^\frac{1}{2}_{L^2}\|D^2v^n\|_{L^2}\|v\|^\frac{1}{8}\|D^4v^n\|^\frac{3}{8}_{L^2} + \|v^n\|^\frac{1}{2}_{L^2}\|\nabla v^n\|^\frac{1}{2}_{L^2}\|v^n\|^\frac{1}{8}_{L^2}\|D^4v^n\|^\frac{7}{8}_{L^2}\right)\|D^4v^n\|_{L^2}
			\\
			&\leq C(T^n_*,\nu_2,\mu_1,\lambda_i,v_0,\theta_0,\epsilon) + 2\epsilon \min\{\nu_2,\mu_1\} \|v^n\|^2_{\dot{H}^4}. 
		\end{align*}
		
		$\bullet$ If $\delta \in (2,3)$ then we find that 
		\begin{equation*}
			\frac{1}{2}\frac{d}{dt}\|v^n\|^2_{\dot{H}^\delta} + \min\{\nu_2,\mu_1\}\|v^n\|^2_{\dot{H}^{\delta+1}} \leq I_{18} + I_{19},
		\end{equation*}
		where for some $\epsilon \in (0,1)$, $\sigma''' := \delta - 2 \in (0,1)$, since $2\delta - 3 \leq \delta$
		\begin{align*}
			I_{18} &:= \sum^2_{i=1} \lambda_1\int_{\mathbb{R}^2} \Lambda \partial_i(\mathbb{P}\theta^n e_2) \cdot \Lambda^{2\delta-3}\partial_iv^n\,dx 
			\\
			&\leq C(T^n_*,\nu_2,\mu_1,\lambda_i,v_0,\theta_0)\left(\|\theta^n\|^2_{\dot{H}^2} +1\right) + \epsilon \min\{\nu_2,\mu_1\}\|v^n\|_{\dot{H}^{\delta+1}};
			\\
			I_{19} &:= -\sum^2_{i=1} \int_{\mathbb{R}^2} \Lambda^{\sigma'''}(\partial_i v^n \cdot \nabla v^n + v^n \cdot \nabla \partial_iv^n) \cdot \Lambda^{\delta}\partial_i v^n \,dx.
		\end{align*}
		It can be seen that $I_{19}$ can be bounded in the same way as $I_8$. 
		Choosing $\epsilon = \frac{1}{6}$ in the three cases  yields  
		\begin{align*}
			\frac{d}{dt} \|v^n\|^2_{H^\delta} + \min\{\nu_2,\mu_1\}\|v^n\|^2_{H^{\delta+1}} \leq C(T^n_*,\delta,\nu_2,\mu_1,\lambda_i,v_0,\theta_0)\left(\|\theta^n\|^2_{\dot{H}^2} + 1\right).
		\end{align*} 
		Moreover,
		\begin{equation*}
			\frac{1}{2}\frac{d}{dt}\|\theta^n\|^2_{\dot{H}^2} = I_{20} + I_{21},
		\end{equation*}
		where for some $\epsilon' \in (0,1)$
		\begin{align*}
			I_{20} &:= \lambda_2\int_{\mathbb{R}^2} \Delta \theta^n  \Delta v^n_2\,dx \leq C(\lambda_2)\left(\|\theta^n\|^2_{H^2} + \|v^n\|^2_{H^\delta}\right);
			\\
			I_{21} &:= -\int_{\mathbb{R}^2} \Delta(T_n(v^n \cdot \nabla \theta^n))  \Delta\theta^n \,dx
			\\
			&\leq C\left(\|\Delta v^n\|_{L^\frac{2}{1-\epsilon'}} \|\nabla \theta^n\|_{L^\frac{2}{\epsilon'}} + \|\nabla v^n\|_{L^\infty}\|\theta^n\|_{\dot{H}^2}\right) \|\theta^n\|_{\dot{H}^2}
			\\
			&\leq C\left(\|v^n\|_{\dot{H}^{2+\epsilon'}} \|\theta^n\|_{\dot{H}^{2-\epsilon'}} + \|\nabla v^n\|_{L^\infty}\|\theta^n\|_{\dot{H}^2}\right) \|\theta^n\|_{\dot{H}^2}
			\\
			&\leq C(\epsilon')\left(\|v^n\|_{H^{2+\epsilon'}} + \|\nabla v^n\|_{L^\infty}\right) \|\theta^n\|^2_{H^2},
		\end{align*}
		which implies that for $Y^n_\delta(t) := \|v^n(t)\|^2_{H^\delta} +  \|\theta^n(t)\|^2_{H^2} + 1$
		\begin{align*}
			\frac{d}{dt}Y^n_\delta(t)  + \min\{\nu_2,\mu_1\}\|v^n\|^2_{H^{\delta+1}} \leq C(T^n_*,\delta,\nu_2,\mu_1,\lambda_i,v_0,\theta_0,\epsilon')\|v^n\|_{H^{2+\epsilon'}} Y^n_\delta(t).
		\end{align*} 
		By using Step 2d (with $s \in (1,2)$ and by choosing $s = \delta = 1 + \epsilon'$), we can bound the time integral of $\|v^n\|_{H^{2+\epsilon'}}$. Thus, \eqref{Hs_estimate} follows for $s = 2$ and $\delta \in [2,3]$.
	\end{proof}

	We now consider the case $s > 2$ and $\delta \in (s,s+1]$ as follows.
	
	\begin{proof}[Proof of Step 2f: The case $s > 2$ and $s < \delta \leq s+1$] We find that
		\begin{equation*}
			\frac{1}{2}\frac{d}{dt}\left(\|v^n\|^2_{\dot{H}^\delta} + \|\theta^n\|^2_{H^s}\right) + \min\{\nu_2,\mu_1\} \|v^n\|^2_{\dot{H}^{\delta+1}} 
			\leq \sum^{33}_{j=30} I_j,
		\end{equation*}
		where for some $\epsilon \in (0,1)$, since $2\delta - s \leq \delta + 1$
		\begin{align*}
			I_{30} &:= \lambda_1\int_{\mathbb{R}^2} \Lambda^s(\mathbb{P}(\theta^n e_2)) \cdot \Lambda^{2\delta-s} v^n\,dx 
			\\
			&\leq C(T^n_*,\nu_2,\mu_1,\lambda_i,v_0,\theta_0,\epsilon) \left(\|\theta^n\|^2_{\dot{H}^s} + 1\right) + \epsilon\min\{\nu_2,\mu_1\} \|v^n\|^2_{\dot{H}^{\delta+1}};
			\\
			I_{31} &:= -\int_{\mathbb{R}^2} \Lambda^\delta(\mathbb{P}(T_n(v^n \cdot \nabla v^n)) \cdot \Lambda^\delta v^n \,dx
			\\
			&\leq C(s,\nu_2,\mu_1,\epsilon) \left(\|v^n\|^2_{L^\infty} + \|\nabla v^n\|_{L^\infty}\right)\|v^n\|^2_{H^\delta} + \epsilon\min\{\nu_2,\mu_1\}\|v^n\|^2_{\dot{H}^{\delta+1}};
			\\
			I_{32} &:= -\lambda_2\int_{\mathbb{R}^2} J^s v^n_2 J^s \theta^n \,dx \leq |\lambda_2|\left(\|v^n\|^2_{H^s} + \|\theta^n\|^2_{H^s}\right);
			\\
			I_{33} &:= -\int_{\mathbb{R}^2} J^s(T_n(v^n \cdot \nabla \theta^n)) J^s \theta^n \,dx
			\leq C(s)\left(\|v^n\|^2_{H^s} + \|\theta^n\|^2_{H^s}\right) \left(\|\nabla v^n\|_{L^\infty} + \|\nabla \theta^n\|_{L^\infty}\right).
		\end{align*}
		By choosing $\epsilon = \frac{1}{4}$, it follows that
		\begin{equation*}
			\frac{d}{dt}Y^n_{\delta,s}(t) + \min\{\nu_2,\mu_1\} \|v^n\|^2_{H^{\delta+1}} \leq C(s,\nu_2,\mu_1,\lambda_i)\left(1 + \|v^n\|^2_{L^\infty} + \|\nabla v^n\|_{L^\infty} + \|\nabla \theta^n\|_{L^\infty}\right) Y^n_{\delta,s}(t),
		\end{equation*}
		where 
		\begin{equation*}
			Y^n_{\delta,s}(t) := \|v^n(t)\|^2_{H^\delta} + \|\theta^n(t)\|^2_{H^s} + 1 \qquad \text{for} \quad t \in (0,T^n_*). 
		\end{equation*}
		Thus, \eqref{Hs_estimate} follows for $s > 2$ and $\delta \in (s,s+1]$.
	\end{proof}
	
	We now consider the case $s > 2$ and $\delta \in [s-1,s)$ as follows.
	
	\begin{proof}[Proof of Step 2g: The case $s > 2$ and $s-1 \leq \delta < s$] We find that
		\begin{equation*}
			\frac{1}{2}\frac{d}{dt}\left(\|v^n\|^2_{\dot{H}^\delta} + \|\theta^n\|^2_{H^s}\right) + \min\{\nu_2,\mu_1\} \|v^n\|^2_{\dot{H}^{\delta+1}} 
			\leq \sum^{37}_{j=34} I_j,
		\end{equation*}
		where for some $\epsilon \in (0,1)$, since $s \leq \delta + 1$
		\begin{align*}
			I_{34} &:= \lambda_1\int_{\mathbb{R}^2} \Lambda^\delta(\mathbb{P}(\theta^n e_2)) \cdot \Lambda^{\delta} v^n\,dx 
			\leq |\lambda_1| \left(\|\theta^n\|^2_{H^s} + \|v^n\|^2_{H^{\delta}}\right);
			\\
			I_{35} &:= -\int_{\mathbb{R}^2} \Lambda^\delta(\mathbb{P}(T_n(v^n \cdot \nabla v^n)) \cdot \Lambda^\delta v^n \,dx
			\\
			&\leq C(\delta,\nu_2,\mu_1,\epsilon) \left(\|v^n\|^2_{L^\infty} + \|\nabla v^n\|_{L^\infty}\right)\|v^n\|^2_{H^\delta} + \epsilon\min\{\nu_2,\mu_1\}\|v^n\|^2_{H^{\delta+1}};
			\\
			I_{36} &:= -\lambda_2\int_{\mathbb{R}^2} J^s v^n_2 J^s \theta^n \,dx \leq C(\nu_2,\mu_1,\lambda_2,\epsilon) \|\theta^n\|^2_{H^s} + \epsilon \min\{\nu_2,\mu_1\}\|v^n\|^2_{H^{\delta+1}} ;
			\\
			I_{37} &:= -\int_{\mathbb{R}^2} J^s(T_n(v^n \cdot \nabla \theta^n)) J^s \theta^n \,dx
			\\
			&\leq C(s,\nu_2,\mu_1,\epsilon)\left(\|v^n\|_{L^\infty} + \|\nabla \theta^n\|^2_{L^\infty}\right)\|\theta^n\|^2_{H^s}  + \epsilon \min\{\nu_2,\mu_1\}\|v^n\|^2_{H^{\delta+1}}.
		\end{align*}
		By adding the term $\min\{\nu_2,\mu_1\}\|v^n(t)\|^2_{L^2}$ to both sides and choosing $\epsilon = \frac{1}{6}$, we obtain
		\begin{equation*}
			\frac{d}{dt}Y^n_{\delta,s}(t) + \min\{\nu_2,\mu_1\} \|v^n\|^2_{H^{\delta+1}} \leq C\left(\|v^n\|_{L^\infty} + \|v^n\|^2_{L^\infty} + \|\nabla v^n\|_{L^\infty} + \|\nabla \theta^n\|^2_{L^\infty}\right) Y^n_{\delta,s}(t),
		\end{equation*}
		where $C = C(\delta,s,\nu_2,\mu_1,\lambda_i)$ and
		\begin{equation*}
			Y^n_{\delta,s}(t) := \|v^n(t)\|^2_{H^\delta} + \|\theta^n(t)\|^2_{H^s}  \qquad \text{for} \quad t \in (0,T^n_*). 
		\end{equation*}
		It remains to bound the $L^1_tL^\infty_x$ norm of $(v^n,\nabla v^n)$ and  the $L^2_tL^\infty_x$ norm of $(v^n,\nabla \theta^n)$. Since $s-1 > 1$ then we can apply Step 2d with $\delta' := s-1$, which gives the bound of $\|v^n\|_{L^2(0,T^n_*;H^{\delta'+1}(\mathbb{R}^2))}$ for $\delta' + 1 = s > 2$. That is enough to bound the $L^1_tL^\infty_x$ norm of $(v^n,\nabla v^n)$, the $L^2_tL^\infty_x$ norm of $v^n$ and as well as the $L^2_tL^\infty_x$ norm of $\nabla \theta^n$ (by using Step 2g for the case $\delta = s > 2$) in terms of $C(T^n_*,\delta,s,\nu_2,\mu_1,\lambda_i,v_0,\theta_0)$. Thus, \eqref{Hs_estimate} follows for $s > 2$ and $\delta \in [s-1,s)$.
	\end{proof}
	
	We continue with the case $s = 2$ and $\delta \in (1,2)$ as follows.
	
	\begin{proof}[Proof of Step 2g: The case $s = 2$ and $\delta \in (1,2)$] We find that 
		\begin{equation*} 
			\frac{1}{2}\frac{d}{dt}\|v^n\|^2_{\dot{H}^\delta}  + \min\{\nu_2,\mu_1\}\|v^n\|^2_{\dot{H}^{\delta+1}} \leq I_{38} + I_{39},
		\end{equation*}
		where for some $\epsilon \in (0,1)$
		\begin{align*}
			I_{38} &:= \lambda_1\int_{\mathbb{R}^2} \Lambda^\delta (\mathbb{P}(\theta^n e_2)) \cdot \Lambda^\delta v^n \,dx 
			\leq |\lambda_1| \left(\|\theta^n\|^2_{H^s} + \|v^n\|^2_{H^\delta}\right);
			\\
			I_{39} &:= -\int_{\mathbb{R}^2} \Lambda^{\delta-1}(v^n \cdot \nabla v^n) \cdot \Lambda^{\delta+1} v^n \,dx.
		\end{align*}
		Since $\delta - 1 \in  (0,1)$, $I_{39}$ can be bounded in the same way as $I_2$. The $H^2$ estimate of $\theta^n$ can be done in the same way as in Step 2e (the integrals $I_{20}$ and $I_{21}$). Thus, \eqref{Hs_estimate} follows for $s = 2$ and $\delta \in (1,2)$.
	\end{proof}
	
	Finally, we finish the proof of Step 2g by giving the proof of the case $s \in (1,2)$ and $\delta \in (1,s)$.
	
	\begin{proof}[Proof of Step 2g: The case $s \in (1,2)$ and $\delta \in (1,s)$] We find that 
		\begin{equation*} 
			\frac{1}{2}\frac{d}{dt}\|v^n\|^2_{\dot{H}^\delta}  + \min\{\nu_2,\mu_1\}\|v^n\|^2_{\dot{H}^{\delta+1}} \leq I_{40} + I_{41},
		\end{equation*}
		where for some $\epsilon \in (0,1)$, since $0 < 2\delta - s \leq \delta + 1$
		\begin{align*}
			I_{40} &:= \lambda_1\int_{\mathbb{R}^2} \Lambda^s (\mathbb{P}(\theta^n e_2)) \cdot \Lambda^{2\delta-s} v^n \,dx 
			\\
			&\leq C(T^n_*,\nu_2,\mu_1,\lambda_i,v_0,\theta_0,\epsilon)\left(\|\theta^n\|^2_{\dot{H}^s} + 1\right) + \epsilon\min\{\nu_2,\mu_1\}\|v^n\|^2_{H^{\delta+1}};
			\\
			I_{41} &:= -\int_{\mathbb{R}^2} \Lambda^{\delta-1}(v^n \cdot \nabla v^n) \cdot \Lambda^{\delta+1} v^n \,dx.
		\end{align*}
		Since $\delta - 1 \in  (0,s-1) \in (0,1)$ for $s \in (1,2)$, $I_{41}$ can be bounded in the same way as $I_2$. The $H^s$ estimate of $\theta^n$ for $s \in (1,2)$ can be done in the same way as in Step 2d (the integrals $I_{11}$, $I_{12}$ and $I_{13}$) with choosing $s_0 + 1 = \delta$ or $s_0 \in (0,s-1)$. Thus, \eqref{Hs_estimate} follows for $s \in (1,2)$ and $\delta \in (1,s)$.
	\end{proof}	
	
	%
	\subsection{Appendix B: Proof of \eqref{Lp}}
	%
	
	The idea of the proof of \eqref{Lp} follows that of \cite[Proposition 3.2]{Boardman-Ji-Qiu-Wu_2019}, where the authors considered the full Laplacian $\Delta v$ instead of $(\partial_{22}v_1,\partial_{11}v_2)$. However, their proof seems to not be applied directly to our case. Thus, we give a proof here for the sake of completeness. Before going to the proof, we quickly recall the definition of the standard nonhomogeneous Besov spaces. There exist two smooth radial functions (see \cite{Bahouri-Chemin-Danchin_2011}) $\chi,\varphi : \mathbb{R}^2 \to [0,1]$ such that 
	\begin{align*}
		&\text{supp}(\chi) \subset \left\{\xi \in \mathbb{R}^2 : |\xi| \leq \frac{4}{3}\right\},
		&&\text{supp}(\varphi) \subset \left\{\xi \in \mathbb{R}^2 : \frac{3}{4} \leq |\xi| \leq \frac{8}{3}\right\},&&
		\\
		&\chi(\xi) + \sum_{j \geq 0} \varphi(2^{-j}\xi) = 1 \quad \forall \xi \in \mathbb{R}^2, &&\sum_{j \in \mathbb{Z}} \varphi(2^{-j}\xi) = 1 \quad \forall \xi \in \mathbb{R}^2 \setminus \{0\}.&&
	\end{align*}
	Define $\tilde{h} := \mathcal{F}^{-1}(\chi)$ and $h := \mathcal{F}^{-1}(\varphi)$, where $\mathcal{F}^{-1}$ denotes the inverse Fourier transform. The nonhomogeneous dyadic blocks are defined by 
	\begin{align*}
		\Delta_j f &:= 
		\begin{cases}
			0 &\text{if} \quad j \leq -2,
			\\ 
			\tilde{h}*f &\text{if} \quad j = -1,
			\\
			2^{2j} h(2^j\cdot)*f &\text{if} \quad j \geq 0,
		\end{cases}
	\end{align*}
	here, $*$ stands for the usual convolution operator. For $s \in \mathbb{R}$ and $p,q \in [1,\infty]$, 
	\begin{equation*}
		B^s_{p,q}(\mathbb{R}^2) := \left\{f \in \mathcal{S}' : \|f\|_{B^s_{p,q}(\mathbb{R}^2)} := \|2^{sj}\|\Delta_j f\|_{L^p(\mathbb{R}^2)}\|_{\ell^q(\mathbb{Z})} < \infty\right\},
	\end{equation*}
	where $\mathcal{S}'$ denotes the dual space of the Schwartz class $\mathcal{S}$. It is also convenient to use the identity $B^s_{2,2}(\mathbb{R}^2) = H^s(\mathbb{R}^2)$. In addition, we have the Littlewood–Paley decomposition
	\begin{equation*}
		f = \sum_{j\in \mathbb{Z}} \Delta_j f \qquad \text{in} \quad \mathcal{S}'.
	\end{equation*} 
	
	\begin{proof}[Proof of \eqref{Lp}]
		It follows from \eqref{B} that for $t \in (0,T)$ and $j \in \mathbb{Z}$
		\begin{equation*}
			\frac{1}{2}\frac{d}{dt}\|\Delta_j v(t)\|^2_{L^2} + \nu_2 \|\Delta_j \partial_2v_1\|^2_{L^2} + \mu_1\|\Delta_j\partial_1v_2\|^2_{L^2} 
			\leq \|\Delta_j v\|_{L^2}\left(\|\Delta_j (v \cdot \nabla v)\|_{L^2} + \|\Delta_j \theta\|_{L^2}\right).
		\end{equation*}
		It can be seen from the definition of dyadic blocks that for $j \in \mathbb{Z}$ with $j \geq 0$
		\begin{align*}
			\nu_2 \|\Delta_j \partial_2v_1\|^2_{L^2} = \nu_2\int_{\mathbb{R}^2} |\mathcal{F}(\Delta_j \partial_2v_1)|^2 \,d\xi = C\nu_2\int_{\mathbb{R}^2} |\varphi(2^{-j}\xi)|^2|\xi_2 \mathcal{F}(v_1)|^2 \,d\xi,
			\\
			\mu_1 \|\Delta_j \partial_1v_2\|^2_{L^2} = \mu_1\int_{\mathbb{R}^2} |\mathcal{F}(\Delta_j \partial_1v_2)|^2 \,d\xi = C\mu_1\int_{\mathbb{R}^2} |\varphi(2^{-j}\xi)|^2|\xi_1\mathcal{F}(v_2)|^2 \,d\xi,
		\end{align*}
		which by using the following estimate 
		\begin{equation*}
			|\xi_2 \mathcal{F}(v_1)|^2 + |\xi_1 \mathcal{F}(v_2)|^2 
			\geq \frac{1}{2}|\xi|^2|\mathcal{F}(v)|^2,
		\end{equation*}
		implies that
		\begin{equation*}
			\nu_2 \|\Delta_j \partial_2v_1\|^2_{L^2} + \mu_1 \|\Delta_j \partial_1v_2\|^2_{L^2} 
			\geq C\min\{\nu_2,\mu_1\} 2^{2j}\|\Delta_j v\|^2_{L^2},
		\end{equation*}
		and 
		\begin{align*}
			\esssup_{t \in (0,T)} \|\Delta_j v(t)\|^2_{L^2} + C\min\{\nu_2,\mu_1\}^2 \left(\int^T_0 2^{2j}\|\Delta_j v\|_{L^2} \,d\tau\right)^2
			&\leq  R^2_j \qquad \text{if} \quad j \geq 0,
			\\
			\esssup_{t \in (0,T)} \|\Delta_j v(t)\|^2_{L^2}  
			&\leq R^2_j \qquad \text{if} \quad j \geq -1,
		\end{align*}
		where
		\begin{equation*}
			R_j := \int^T_0 \|\Delta_j (v \cdot \nabla v)\|_{L^2} +  \|\Delta_j \theta\|_{L^2} \,d\tau + \|\Delta_j v(0)\|_{L^2}.
		\end{equation*}
		Furthermore,
		\begin{equation*}
			\sum_{j \geq -1} \esssup_{t \in (0,T)} \|\Delta_j v(t)\|^2_{L^2} + C\min\{\nu_2,\mu_1\}^2 \sum_{j \geq 0} \left(\int^T_0 2^{2j}\|\Delta_j v\|_{L^2} \,d\tau\right)^2
			\leq  2\sum_{j \geq -1} R^2_j.
		\end{equation*}
		We now estimate the right-hand side as follows\footnote{For $T > 0$, $s \in \mathbb{R}$, $r,p,q \in [1,\infty]$, the space $\tilde{L}^rB^s_{p,q}$ is defined in \cite{Bahouri-Chemin-Danchin_2011} with
		\begin{equation*}
			\|f\|_{\tilde{L}^rB^s_{p,q}} := \|2^{sj}\|\Delta_j f\|_{L^r(0,T;L^p(\mathbb{R}^2))}\|_{\ell^q(\mathbb{Z})} < \infty.
		\end{equation*}
		} 
		\begin{align*}
			\sum_{j \geq -1} R^2_j &\leq C \sum_{j \geq -1} \left(\int^T_0 \|\Delta_j (v \cdot \nabla v)\|_{L^2} \,d\tau\right)^2 +  C \sum_{j \geq -1} \left(\int^T_0 \|\Delta_j \theta\|_{L^2} \,d\tau \right)^2 + C\|v_0\|^2_{L^2}
			\\
			&\leq C\left(\int^T_0 \left(\sum_{j \geq -1} \|\Delta_j (v \cdot \nabla v)\|^2_{L^2} \right)^\frac{1}{2} \,d\tau\right)^2 + C\left(\int^T_0 \left(\sum_{j \geq -1} \|\Delta_j \theta\|^2_{L^2} \right)^\frac{1}{2} \,d\tau\right)^2 + C\|v_0\|^2_{L^2}
			\\
			&\leq C\left(\|v\|^2_{L^2_tL^\infty_x}   \|\nabla v\|^2_{L^2_tL^2_x} + \|\theta\|^2_{L^1_tL^2_x} + \|v_0\|^2_{L^2}\right),
		\end{align*}
		where we used the fact that $L^1(0,T;B^0_{2,2}(\mathbb{R}^2)) \subset \tilde{L}^1B^0_{2,2}$ by using the Minkowski inequality. It remains to bound $\|v\|_{L^2_tL^\infty_x}$. The Littlewood–Paley decomposition gives us
		\begin{align*}
			\int^T_0 \|v\|^2_{L^\infty} \,d\tau 
			&\leq C \int^T_0 \sum_{j \geq -1} 2^j \|\Delta_j v\|_{L^2}  \sum_{i \geq -1}  2^i \|\Delta_i v\|_{L^2} \,d\tau,
			\\
			&\leq C \int^T_0 \left(\sum_{|j-i| \leq N} + \sum_{|j-i| > N}\right) 2^j \|\Delta_j v\|_{L^2}   2^i \|\Delta_i v\|_{L^2} \,d\tau =: I_{53} + I_{54},
		\end{align*}
		where $N \in \mathbb{N}$ to be determined later and we used the following Bernstein-type estimate (see \cite{Bahouri-Chemin-Danchin_2011}) for $1 \leq q \leq p \leq \infty$ 
		\begin{equation*}
			\|f\|_{L^p} \leq C(p,q) \lambda_0^{2\left(\frac{1}{q}-\frac{1}{p}\right)} \|f\|_{L^q} \qquad \text{if} \quad \text{supp}(\mathcal{F}(f)) \subset \left\{\xi \in \mathbb{R}^2 : |\xi| \leq \lambda_0 \right\}.
		\end{equation*}		
		The terms on the right-hand side can be bounded as follows
		\begin{align*}
			I_{53} &= \int^T_0 \sum_{j \geq -1} 2^j \|\Delta_j v\|_{L^2}   \sum_{j-N \leq i \leq j+N} 2^i \|\Delta_i v\|_{L^2} \,d\tau
			\\
			&\leq C\int^T_0 \sum_{j \geq -1}  \left(2^{2j}\|\Delta_j v\|^2_{L^2}  + \sum_{j-N \leq i \leq j+N} 2^{2i} \|\Delta_i v\|^2_{L^2} \right) \,d\tau
			\leq CN \left(\|\nabla v\|^2_{L^2_tL^2_x} + \|v\|^2_{L^2_tL^2_x}\right),
		\end{align*}
		and by using Young inequality for sequences
		\begin{align*}
			I_{54} &= 2^{1-N} \sum^\infty_{j = N} \int^T_0  2^{2j} \|\Delta_j v\|_{L^2}  \sum_{i = -1} ^{j-N-1} 2^{i-(j-N)} \|\Delta_i v\|_{L^2} \,d\tau
			\\
			&\leq 2^{1-N} \left(\sum^\infty_{j = N} \left(\int^T_0  2^{2j} \|\Delta_j v\|_{L^2} \,d\tau \right)^2\right)^\frac{1}{2}   \left(\sum^\infty_{k=0} \left(\sum_{i = -1} ^{k-1} 2^{-(k-i)} \esssup_{t \in (0,T)} \|\Delta_i v\|_{L^2} \right)^2\right)^\frac{1}{2}
			\\
			&\leq 2^{1-N} \left(\sum^\infty_{j = 0} \left(\int^T_0  2^{2j} \|\Delta_j v\|_{L^2} \,d\tau \right)^2\right)^\frac{1}{2}   \left(\sum^\infty_{i=-1} \esssup_{t \in (0,T)} \|\Delta_i v\|^2_{L^2} \right)^\frac{1}{2} \sum^\infty_{k = -1} 2^{-k}.
		\end{align*}
		Therefore,
		\begin{align*}
			\|v\|^2_{L^2_tL^\infty_x}  &\leq C(\nu_2,\mu_1) 2^{-N} \left(\|v\|^2_{L^2_tL^\infty_x}   \|\nabla v\|^2_{L^2_tL^2_x} \right)
			+ CN \left(\|\nabla v\|^2_{L^2_tL^2_x} + \|v\|^2_{L^2_tL^2_x}\right)
			\\
			&\quad + C(\nu_2,\mu_1) \left(\|\theta\|^2_{L^1_tL^2_x} + \|v_0\|^2_{L^2}\right),
		\end{align*}
		which by choosing\footnote{Here $\lceil \cdot \rceil$ denotes the usual ceiling function.}
		\begin{equation*}
			N = \left\lceil\log_2\left(4 + 2C(\nu_2,\mu_1)\|\nabla v\|^2_{L^2_tL^2_x}\right)\right\rceil
		\end{equation*}
		yields
		\begin{equation*}
			\|v\|^2_{L^2_tL^\infty_x} \leq 
			C(\nu_2,\mu_1)\left((\|\nabla v\|^2_{L^2_tL^2_x} + 1)\left(\|\nabla v\|^2_{L^2_tL^2_x} + \|v\|^2_{L^2_tL^2_x}\right) +
			\|\theta\|^2_{L^1_tL^2_x} + \|v_0\|^2_{L^2} + 1\right).
		\end{equation*}
		In addition, the Littlewood–Paley decomposition and Bernstein-type estimate imply that for $p \geq 2$
		\begin{align*}
			\int^T_0 \|\nabla v\|_{L^p} \,d\tau 
			&\leq C \sum_{j \geq 0} \int^T_0  2^{2j\left(\frac{1}{2}-\frac{1}{p}\right)} \|\Delta_j \nabla v\|_{L^2} \,d\tau  + C \int^T_0\|\nabla v\|_{L^2} \,d\tau
			\\
			&\leq C \left(\sum_{j \geq 0} 2^{-\frac{4j}{p}}\right)^\frac{1}{2} \left(\sum_{j \geq 0} \left(\int^T_0  2^{2j} \|\Delta_jv\|_{L^2} \,d\tau\right)^2\right)^\frac{1}{2}  + C \int^T_0\|\nabla v\|_{L^2} \,d\tau,
		\end{align*}
		in which 
		\begin{align*}
			\sum_{j \geq 0} 2^{-\frac{4j}{p}} \leq 1 + \int^\infty_0 2^{-\frac{4x}{p}} \,dx \leq  \left(1 + \frac{1}{4\log 2}\right)p.
		\end{align*}
		Finally, it follows that for $p \in [2,\infty)$
		\begin{equation*}
			p^{-1} \int^T_0 \|\nabla v\|_{L^p} \,d\tau \leq p^{-\frac{1}{2}} \int^T_0 \|\nabla v\|_{L^p} \,d\tau \leq C(T,\nu_2,\mu_1,\lambda_i,v_0,\theta_0),
		\end{equation*}
		which ends the proof of \eqref{Lp}.
	\end{proof}
	
	%
	\subsection{Appendix C: Representation formula}
	%
	
	We now verify the point number 4 in Remark \ref{rm1}. 
	\begin{proof}[Proof of Remark \ref{rm1}: The point number 4.] 
	Let us consider the following system
	\begin{equation*}
		\partial_t \mathcal{F}(v) = \mathcal{F}(\mathbb{P}(\nu_2\partial_{22}v_1,\nu_1\partial_{11}v_2)) + \mathcal{F}(f),
		\qquad
		\textnormal{div}\, v = 0
		\qquad \text{and} \qquad \mathcal{F}(v)(t=0) = \mathcal{F}(v_0),
	\end{equation*}
	with $\textnormal{div}\, v_0 = \textnormal{div}\, f = 0$ and $\nu_1,\nu_2 > 0$. It can be seen that $\partial_t \mathcal{F}(v) = A_2(\xi)\mathcal{F}(v) + \mathcal{F}(f)$, where
	\begin{align*}
		A_2(\xi) &:= P(\xi)
		\left(
		\begin{matrix}
			-\nu_2\xi_2^2 & 0
			\\
			0 & -\nu_1\xi^2_1
		\end{matrix} 
		\right)
		P(\xi)
		=
		a(\xi) P(\xi), 
		\\
		P(\xi) &:= \frac{1}{|\xi|^2}
		\left(
		\begin{matrix}
			\xi^2_2 & -\xi_1\xi_2
			\\
			-\xi_1\xi_2 & \xi^2_1
		\end{matrix}
		\right),
		\qquad a(\xi) :=  -\frac{\nu_1\xi^4_1 + \nu_2\xi^4_2}{|\xi|^2},
	\end{align*}
	which implies that $\partial_t \mathcal{F}(v) = a(\xi) \mathcal{F}(v) + \mathcal{F}(f)$ and 
	\begin{equation*}
		\mathcal{F}(v)(\xi,t) = \exp\{a(\xi)t\} \mathcal{F}(v_0)(\xi) + \int^t_0 \exp\{a(\xi)(t-s)\} \mathcal{F}(f)(\xi,s)\,ds.
	\end{equation*}
	Moreover, it can be seen that
	\begin{equation*}
		\exp\left\{-\max\{\nu_1,\nu_2\}|\xi^2|t\right\} \leq  \exp\{a(\xi)t\} \leq \exp\left\{-\frac{1}{2}\min\{\nu_1,\nu_2\}|\xi^2|t\right\}.
	\end{equation*}
	Since the existence and uniqueness of global weak solutions have been obtained by Theorem \ref{theo_nu2_mu1_de} in this case for $L^1 \cap L^2$ data. Thus, we can follow exactly the proof given in \cite[Theorem 5.1]{Schonbek_1986}, under additional conditions to \eqref{B} such as $v_0 \in L^1(\mathbb{R}^2)$, $\lambda_2 = 0$ and $\theta \equiv 0$, to obtain the $L^2$ decay in time of $v$ with the rate as $\log(e+t)^{-\frac{m}{2}}$ for $m \in \mathbb{N}, m \geq 3$.  
	\end{proof}
	
	%
	\section*{Acknowledgements} 
	%
	
	K. Kang’s work is supported by NRF-2019R1A2C1084685. D. D. Nguyen’s work is supported by NRF-2019R1A2C1084685 and NRF-2021R1A2C1092830.
	J. Lee’s work is supported by NRF-2021R1A2C1092830. The authors are grateful to the anonymous referee
	for her or his carefully reading the paper and for useful suggestions, which are presented in Remark \ref{r3}.

	
	\end{document}